\documentclass[11pt,reqno]{amsart}

\usepackage[margin=1.2in]{geometry}

\usepackage{amsmath,amssymb,amsthm,graphicx,stmaryrd}
\usepackage{verbatim,enumitem,xcolor}
\usepackage[enableskew]{youngtab}
\usepackage{tikz,tikz-cd,genyoungtabtikz,pgfplots, bbm, mathtools}
\usepackage{caption,subcaption}
\usetikzlibrary{patterns}
\usetikzlibrary{patterns.meta}
\usetikzlibrary{arrows.meta}
\usetikzlibrary{decorations.markings}
\usepackage{hyperref}
\usepackage{tikzit}
\usepackage{overpic}
\usepackage{amsaddr}

\tikzstyle{white dot}=[fill=white, draw=black, shape=circle]
\tikzstyle{black dot}=[fill=black, draw=black, shape=circle]

\tikzstyle{new edge style 0}=[-, draw={rgb,255: red,208; green,208; blue,208}]
\tikzstyle{new edge style 1}=[-, draw={rgb,255: red,0; green,128; blue,128}]
\tikzstyle{new edge style 2}=[-, draw={rgb,255: red,255; green,128; blue,0}]

\usepackage{dsfont}

\pgfdeclarepattern{
name=hatch,
parameters={\hatchsize,\hatchangle,\hatchlinewidth},
bottom left={\pgfpoint{-.1pt}{-.1pt}},
top right={\pgfpoint{\hatchsize+.1pt}{\hatchsize+.1pt}},
tile size={\pgfpoint{\hatchsize}{\hatchsize}},
tile transformation={\pgftransformrotate{\hatchangle}},
code={
\pgfsetlinewidth{\hatchlinewidth}
\pgfpathmoveto{\pgfpoint{-.1pt}{-.1pt}}
\pgfpathlineto{\pgfpoint{\hatchsize+.1pt}{\hatchsize+.1pt}}
\pgfpathmoveto{\pgfpoint{-.1pt}{\hatchsize+.1pt}}
\pgfpathlineto{\pgfpoint{\hatchsize+.1pt}{-.1pt}}
\pgfusepath{stroke}
}
}

\tikzset{
hatch size/.store in=\hatchsize,
hatch angle/.store in=\hatchangle,
hatch line width/.store in=\hatchlinewidth,
hatch size=5pt,
hatch angle=0pt,
hatch line width=.5pt,
}

\newcommand{\E}{\mathbb E}
\newcommand{\A}{\mathbb A}
\newcommand{\Ale}{\mathcal A}
\newcommand{\R}{\mathbb R}

\renewcommand{\d}{\mathrm{d}}

\newcommand{\T}{\intercal}
\newcommand{\pf}{\mathrm{Pf}}

\newcommand{\CC}{\mathbb{C}}
\newcommand{\EE}{\mathbb{E}}

\newcommand{\HH}{\mathbb{H}}

\newcommand{\NN}{\mathbb{N}}

\newcommand{\PP}{\mathbb{P}} 
\newcommand{\RR}{\mathbb{R}}

\newcommand{\ZZ}{\mathbb{Z}}

\newcommand{\cF}{\mathcal{F}}

\newcommand{\cG}{\mathcal{G}}

\renewcommand{\a}{\alpha}
 
\renewcommand{\d}{\delta} 
\newcommand{\G}{\mathcal{G}}

\newcommand{\Om}{\Omega}
\newcommand{\om}{\omega}

\newcommand{\eps}{\varepsilon}

\newcommand{\tr}{\mathrm{tr}}

\newcommand{\zp}{z_{\partial}}
\newcommand{\tK}{\tilde{K}}
\newcommand{\tS}{\tilde{S}}
\newcommand{\bK}{\bar{K}}

\newcommand{\wtilde}{\widetilde}

\newcommand{\ol}{\overline}

\newcommand{\Ged}{\mathcal{G}_{\varepsilon}^r}
\newcommand{\Ge}{\mathcal{G}_{\varepsilon}}
\newcommand{\Ud}{U^r}
\newcommand{\Ue}{U_{\varepsilon}}
\newcommand{\Uer}{U_{\varepsilon}^r}

\newcommand{\Il}{\mathcal{I}^{\lambda}}

\newcommand{\Ee}{\mathbbm{E}_{\eps}}
\newcommand{\nez}{n_{\eps}(z)}
\newcommand{\oez}{o_{\eps}(z)}

\newcommand{\Z}{\mathcal{Z}}

\newcommand{\SL}{\mathrm{SL}}

\newcommand{\one}{\hbox{\rm 1\kern-.27em I}}

\newcommand{\sgn}{\text{sgn}}

\newcommand{\be}{\begin{equation}}
\newcommand{\ee}{\end{equation}}
\newcommand{\bestar}{\begin{equation*}}
\newcommand{\eestar}{\end{equation*}}

\newcommand{\bee}{\begin{align*}}
\newcommand{\eee}{\end{align*}}

\newtheoremstyle{slthm}
{}
{\baselineskip}
{\slshape}
{\parindent}
{\scshape}
{.}
{ }
{}

\theoremstyle{slthm}
\newtheorem{definition}{Definition}[section]
\newtheorem{theorem}[definition]{Theorem}
\newtheorem{proposition}[definition]{Proposition}
\newtheorem{lemma}[definition]{Lemma}
\newtheorem{corollary}[definition]{Corollary}
\newtheorem{remark}[definition]{Remark}

\newtheorem{example}[definition]{Example}

\title{Conformally invariant boundary arcs in double dimers}
\author[Marcin Lis, Lucas Rey and Kieran Ryan]{Marcin Lis, Lucas Rey and Kieran Ryan \\ 
\emph{\lowercase{(with an appendix by \uppercase{A}velio \uppercase{S}ep\'ulveda)}}}
\thanks{Marcin Lis: Vienna University of Technology, marcin.lis@tuwien.ac.at}
\thanks{Lucas Rey: CEREMADE, Universit{\'e} Paris-Dauphine, Université PSL, CNRS, 75016 Paris, France and DMA, {\'E}cole normale sup{\'e}rieure, Universit{\'e} PSL, CNRS, 75005 Paris, France, lucas.rey@dauphine.psl.eu}
\thanks{Kieran Ryan: Vienna University of Technology, kieran.ryan@tuwien.ac.at}
\date{\today}

\begin{document}

\maketitle

\begin{abstract}
     We consider two different versions of the double dimer model on a planar domain, where we either fold a single dimer cover on a symmetric domain onto itself across the line of symmetry, or we superimpose two independent
     dimer covers on two, almost identical, domains that differ only on a certain portion of the boundary. This results in a collection of loops and doubled edges that, unlike in the classical double dimer case of Kenyon, are accompanied by arcs emanating from the line of symmetry or the chosen portion of the boundary. We argue that these arcs together with the associated height function satisfy a discrete version of the coupling of Qian and Werner
     between the Arc loop ensemble (ALE) and two different variants of the Gaussian free field (with Dirichlet and Neumann boundary conditions).
     We also show that certain statistics of the arcs (when the loops are disregarded from the picture) converge to conformally invariant quantities in the small-mesh scaling limit, and moreover the limits are the same for the two versions of the model, and equal to the corresponding statistics of the arc loop ensemble (ALE). This gives evidence to the conjecture of~\cite{BLQ} (that concerns one of these models).
	\end{abstract}

\section{Introduction}
\subsection*{Motivation}
We consider two different but closely related versions of the double dimer model defined on planar graphs, that give rise not only to a collection of loops and doubled edges as in the original work of Kenyon~\cite{Ken14} but also to a collection of arcs emanating from and ending at a chosen piece of the boundary (see Figure~\ref{fig:Temperley}). These arcs are a new feature of the model that does not appear in the original double dimers,
and our main goal is to study their statistics (when the loops are disregarded from the picture). One of these two constructions was described in~\cite{BLQ}, where a conjecture about the scaling limit of the arcs was stated
claiming that the full collection of arcs should converge to the so-called \emph{arc loop ensemble (ALE)} introduced by Aru, Sep\'ulveda and Werner in~\cite{ASW}.
Here we develop the methods of~\cite{Ken14} to fit this setup and use them to give (partial) evidence for this conjecture. To be precise, we compute the scaling limits of certain observables of the arcs themselves, and they turn out to be conformally invariant and agree with the analogous quantities defined directly for the~ALE. 

Our motivation (outside of the dimer model) comes from the theory of the continuum \emph{Gaussian free field (GFF)} and its level sets.
The GFF with \emph{Dirichlet} boundary conditions in a domain $D$ is a Gaussian field whose covariance kernel is given by the Green function $g(x,y)$ of Brownian motion in the domain with Dirichlet boundary condition. In this article we assume that the Green function is normalised so that $g(x,y)\sim(-2\pi)^{-1} \log|x-y|$ as $x\to y$ in the interior of the domain.
 We are guided by the construction of Qian and Werner~\cite{QW} who coupled two different versions of the
GFF, with Dirichlet and \emph{Neumann (free)} boundary conditions (we refer the reader to~\cite{QW} for a definition of the latter), through a common set of level lines that form an ALE. 

Even though the GFF is not a function, but only a generalised function (a Schwartz distribution acting on test functions), a beautiful theory of its level sets has been developed in recent years~\cite{ASW, AruSep, FPS}. 
An important notion in this geometric representation of the GFF is that of the two-valued sets $\mathbb A_{-a, b}$, where $a,b>0$, $a+b\geq 2\lambda$, and where $\lambda = \sqrt{{\pi}/8}$
is the constant appearing in the height-gap phenomenon of Schramm and Sheffield~\cite{SchShe}. 
Formally they are thin local sets in the sense of~\cite{ASW} whose associated harmonic function takes only two values $-a$ and $b$. Heuristically they are random sets naturally coupled with the GFF
that could be seen as sets of those points in $D$ that are connected to $\partial D$ by a path on which the values of the field (which is not defined pointwise)
lie between $-a$ and $b$. The ALE is then the set of arcs given by the (restriction to $D$ of the) connected components of $A_{-\lambda,\lambda}$. In~\cite{QW} a surprising coupling between the ALE and two different types of the GFF was given. In this construction, the Dirichlet GFF (with zero boundary conditions) changes its value by $\pm 2\lambda$ across the arcs of the ALE in a deterministic and alternating fashion, whereas the Neumann (free boundary condition) GFF changes its value by $\pm 2\lambda$ in an i.i.d. fashion.

\begin{figure}
     \begin{overpic}[abs,unit=1mm,scale=0.8]{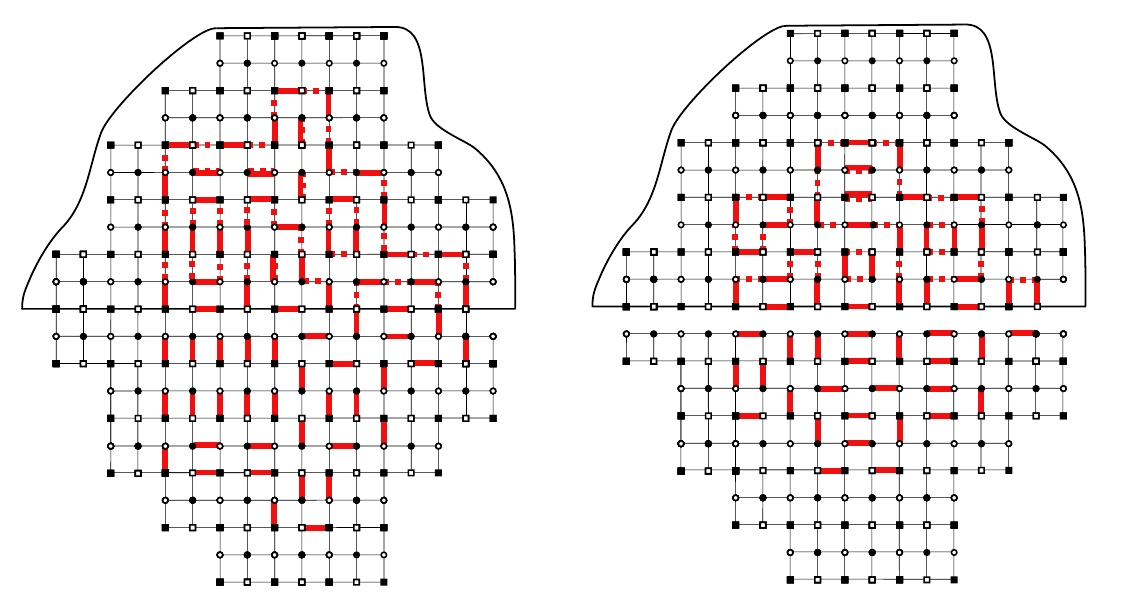}
        \put(8,60){U}
        \put(85,60){U}
     \end{overpic}
     \caption{\textbf{Left:} a (partial) dimer cover of $\Ge^r$ is represented in red, and the lower part is folded on the upper part in dotted red. A folded dimer configuration appears on the upper part: loops and arcs alternate between full red and dotted red. The graph $\Ge^r$ has Temperleyan boundary conditions: all its corners are black squares, and one of the black squares on the boundary (in this case on the right-hand side of the intersection with the real line) is removed. 
     \textbf{Right:} (partial) dimer covers of $\Ge^r \cap (\RR \times \RR_{\geq 0})$ and $\Ge^r \cap (\RR \times \RR_{<0})$ are represented in solid red. The lower part is superimposed on the upper part in dotted red. A superimposed configuration appears on the upper part: loops and arcs alternate between full red and dotted red. 
     The upper part has again Temperleyan boundary conditions, and the lower part has piecewise Temperleyan boundary conditions with two white bullet corners.
     \\
     The sets of vertices $W_0$ and $B_0$ are represented by white and black bullets, $W_1$ and $B_1$ are represented by white and black squares. }
     \label{fig:Temperley}
\end{figure}


\subsection*{Discrete models and couplings}
In this article we consider a set of discrete arcs in two different but related versions of double dimers that satisfy an analogous coupling when considered together with the associated height function (that is known to converge to the Gaussian free field in a variety of settings~\cite{Ken01,Russ,BLR,DCLQ1}). 
To be precise, recall that the \emph{dimer model} is a random perfect matching (which in the setting of this work is chosen uniformly at random) of a finite graph. We will assume that the graph is a finite connected subgraph of the rescaled square lattice
$\eps \mathbb Z^2$ that is symmetric across the real line, and whose restriction to the upper half-plane approximates a simply connected continuum domain $U$. We will denote such graphs by $\Ge^r$.\\

In the first setup, referred to as \emph{folded dimers}, that corresponds to the Neumann GFF in the coupling of~\cite{QW}, we consider the dimer model on $\Ge^r$, and fold the resulting perfect matching along the real line as in Fig.~\ref{fig:Temperley}. This results in a collection of loops, doubled edges, and arcs starting
and ending on the real line. The associated height function on the faces of the folded graph can be defined (up to a global additive constant) by assigning orientations to each 
dimer in the folded dimer cover: each dimer coming from the top half of $\Ge^r$ is oriented from a black to a white vertex (in a fixed bipartite colouring of the lattice), and each dimer coming from the bottom half is oriented from 
a white vertex to a black vertex. This results in a consistent orientation of each loop and arc, and finally one declares that the increment of the height function from a face $u$ to a neighbouring face is $+1$ (resp. $-1$) if there is 
a loop or an arc separating the two faces such that $u$ lies on the left-hand (resp. right-hand) side of the oriented dimer in the loop or arc.
Note each loop and arc can be oriented in the two possible ways with equal probability 
and independently of other loops and arcs. This corresponds to the increment of the height function being $\pm 1$, and is analogous to the coupling of Qian and Werner between the ALE and the GFF with mixed Neuman/Dirichlet boundary conditions.

One can see that the height function in this setting does indeed converge to the continuum GFF with mixed Neuman/Dirichlet boundary conditions. Indeed, it is known from the original work of Kenyon~\cite{Ken01} (and more recent generalisations~\cite{Russ,BLR}) that if the dimer model on $\Ge^r$ has \emph{Temperleyan} boundary conditions (which we assume to be the case in our arguments as is shown in Fig.~\ref{fig:Temperley} and
Fig.~\ref{fig:Kasteleyn}), then the \emph{centered} height function of the dimer model on $\Ge^r$, multiplied by $2\lambda$ to match the height gap across the continuum level lines, converges as $\varepsilon \to 0$
to $1/\sqrt2$ times the Dirichlet GFF in the symmetric domain $U^r$ whose restriction to the upper half plane is~$U$ (i.e. $U^r=U\cup \bar U \cup (\partial U \cap (\{ 0\} \times \mathbb R))$).
We note that here we take into account the fact that the height function of a single dimer model defined in~\cite{Ken01} is $4$ times the height function we are using in this article. From this definition it is clear that the expectation of the height function on~$\Ge^r$ is antisymmetric across the real line, and that the height function at a face $u$ in the
folded dimer model is by definition equal to the sum of the height functions at $u$ and at its reflection $\bar u$ in the single dimer model on~$\Ge^r$. Since the continuum GFF in $U$ with Neumann boundary conditions 
on the real line and Dirichlet boundary conditions on the rest of $\partial U$ can be obtained in the same way from $1/\sqrt 2$ times the Dirichlet GFF in $U^r$ (as e.g. described in~\cite{QW,BLQ}), we can conclude that $2\lambda$ times the height function 
defined by the discrete loops and arcs converges to exactly the same GFF with mixed boundary conditions as in the coupling of~\cite{QW}.\\

In the second setup, which we refer to as \textit{shifted dimers}, we first remove the row of vertical edges of $\Ge^r$ just below the real line (see Fig.~\ref{fig:Temperley}) and then repeat the folding procedure as in the first setup. This results in 
a number of differences compared to the previous case. The first one is that now 
the arcs emanating from the real line cannot have arbitrary orientations. Indeed, one quickly realises that due to parity reasons the arcs must deterministically alternate in orientation -- just as in the part of the coupling in~\cite{QW} concerning the Dirichlet GFF. Moreover, after the removal of edges the graph splits into two connected components $\Ge^r \cap (\RR \times \RR_{\geq 0})$ and $\Ge^r \cap (\RR \times \RR_{<0})$.
One can check that the the upper component has again Temperleyan boundary conditions, whereas the lower one has \emph{piecewise Temperleyan} boundary conditions (see Fig.~\ref{fig:Temperley}) introduced by Russkikh~\cite{Russ}, who proved convergence of $2 \lambda$ times the \emph{centered} height function to $1/\sqrt 2$ times the Dirichlet GFF. Later Berestycki and Liu showed~\cite[Theorem 5.1]{BerLiu} that the mean 
of $2\lambda$ times the height function converges to 
a harmonic function with piecewise constant boundary conditions with alternating jumps of size $\pm \lambda$ located at the points of the \emph{Temperleyan corners} 
(in our case these are the two extremities of $\partial U\cap (\{0\} \times \mathbb R)$) plus a term that involves the winding of the boundary of $U$.
Since $U$ and $\bar U$ have opposite windings of the boundary, they cancel out when the two height functions are added, and we can conclude that
the mean of $2\lambda$ times the height function of the superimposed configuration converges to the harmonic function in $U$ with $\lambda$ boundary conditions on $\partial U\cap (\{0\} \times \mathbb R)$ and zero elsewhere. 
Since the sum of two independent fields that are both $1/\sqrt{2}$ times the Dirichlet GFF (coming from the upper and lower domains) is ($1$ times) the Dirichlet GFF, we can conclude that
$2\lambda$ times the height function of our model converges to the corresponding field described in~\cite{QW}, i.e. a GFF in $U$ with boundary conditions equal to $\lambda$ on 
$\partial U\cap (\{ 0\} \times \mathbb R)$ and zero on the rest of $\partial U$.\\

We conjecture that the arcs in both models converge to the ALE in $U$ (emanating from $\partial U \cap (\{0 \}\times \mathbb R)$ as in~\cite{QW}). This extends the conjecture of~\cite{BLQ} which is the same statement only for the first ``folded'' model.
Our conjecture is strongly evidenced by the discussion above. However, convergence of the height function is known not to be enough to conclude convergence of interfaces in the double dimer model~\cite{Ken14,DubDD,BasChe}.
Finally let us mention that the two models yield different measures in the discrete setting, when one forgets the orientation of loops and arcs. Indeed, in the first model each nontrivial arc (that is not a single edge)
comes with a combinatorial factor of two corresponding to the two different orientations, and it comes with a factor of one in the second model. 
The number of nontrivial arcs is not constant (unlike the number of all arcs) and hence the measures are different. Nonetheless we conjecture that the two collections of arcs have the same scaling limit.

\subsection*{Main results}
The main observables of interest in both models will be $\nez$ -- the number of arcs that surround a point $z$, i.e.~arcs 
that separate $z$ from the part of the boundary with no arcs~$\partial U \setminus  (\{0\} \times \mathbb R)$, and $o_\varepsilon(z)$ -- the parity of $n_\varepsilon(z)$.
Our main result can be summarised as follows.
\begin{theorem}
In both models, as $\varepsilon \to 0$ the laws of $\nez$ and $o_\varepsilon(z)$ converge in distribution to random variables $n(z)$ and $o(z)$ 
whose laws are independent of the model and moreover are conformally invariant.
\end{theorem}
We refer to Theorem~\ref{thm:asymp:ne:oe} and Corollary~\ref{cor:joint:moments} for the general statements, and to Example~\ref{ex:strip} for a particular example of the infinite strip $U^r=[-\infty, +\infty]\times[-\pi/2, \pi/2]$. 

Moreover we give integral representations of the moments of $\nez$, and compute its mean explicitly. 
In particular, on the infinite strip we show that, as $\varepsilon \to 0$,
\begin{align*}
 \Ee[\nez] \to \frac{1}{4}-\frac{y^2}{\pi^2} -\frac{2}{\pi^2}\ln(\sin(y)) ,
 \end{align*}
which agrees with the mean of the same random variable defined directly for the ALE (as computed by Avelio Sep\'ulveda in Appendix~\ref{app:ALE}).
We note that our methods could be extended to analyse the distribution of the number of arcs separating two chosen points in the domain.

One of the main conceptual difficulties in approaching the problem of studying the arcs themselves is the issue of separating the loops from the arcs in the counting arguments. 
This is achieved by developing a novel variant of Kenyon's formula 
for double dimers with arcs, and afterwards choosing the right weights (or more precisely SL(2) connections) in this formula. Afterwards an involved analysis of the scaling limit of the formula by means of discrete complex analysis techniques is required.

We note that the works~\cite{Ken14,DubDD,BasChe} study multi-point functions of topological observables in the double dimer model. Basok and Chelkak~\cite{BasChe} showed that these observables contain enough information to identify the law of the loops. 
In our case, the fact that we want to disregard all the loops from the picture restricts greatly the number of useful topological observables, and we can only treat the one-point (or two-point) function. 
However, it will be very interesting to develop the methods of \cite{Ken14,DubDD,BasChe} to study the full collection of both arcs and loops, that (based on the conjecture of Kenyon~\cite{Ken14}) and our conjectures should converge to the ALE with an independent CLE$_4$ in its complement.\\

Below we include a table summarising the main properties of our two dimer models.\\
\begin{center}
\begin{tabular}{ | c | c | c | }
\hline
  & Folded dimer model & Shifted dimer model \\ [0.5ex]
 \hline\hline
 Measure notation & $\mu^1_\eps$ & $\mu^2_\eps$ \\  
 \hline
Height function b.c. in discrete & Varies & Alternates between 0 and 1  \\
 \hline
Limit GFF b.c. on bdry arc & Neumann & Dirichlet  \\
 \hline
Orientation of arcs in discrete & i.i.d. Bernoulli & Deterministic  \\
\hline
Weight on arcs in discrete & $2^{\#  arcs}$ & $1^{\# arcs}$  \\
\hline
Conjectured limit of arcs & ALE & ALE  \\
\hline
\end{tabular}
\end{center}

\subsection*{Organisation of the article}

\begin{itemize}
\item In Sect.~\ref{sec:KForm} we give a new combinatorial interpretation and also generalise Kenyon's formula for double dimers from~\cite{Ken14} to the setting 
that includes arcs. 
We note that such formulas can also be obtained from the independent work of Douglas, Kenyon and Shi~\cite{webs} who studied $n$-multiwebs on planar graphs.
\item In Sect.~\ref{section:l&a} we define the two versions of the model described above, and state the main result (Theorem~\ref{thm:asymp:ne:oe}) for both of them.
\item In Sect.~\ref{sec:proof-main-thm} we prove the main theorem in the folded dimers model.
\item In Sect.~\ref{sec:shifted} we prove the main theorem in the shifted dimers model (we elaborate on the differences with the proof in Sect.~\ref{sec:shifted}).
\item In Sect.~\ref{sec:rectangle}, we consider a case of a symmetric domain with two pieces of the boundary from which the arcs emanate, and we compute explicitly the distribution of the 
arcs that connect the two boundary pieces together. 
\item In App.~\ref{app:ALE} (due to Avelio Sep\'ulveda) a continuum computation of the expected number of arcs in ALE surrounding a point is presented.
\item In App.~\ref{appendix:A}, we recall the asymptotic estimates of the inverse Kasteleyn matrix for the dimer model in Temperleyan domains obtained by Kenyon in \cite{Ken00}. We detail some arguments concerning the behaviour near the boundary.
\item In App.~\ref{app:B}, we recall the extension of the results of App.~\ref{appendix:A} to piecewise Temperleyan domains obtained by Russkikh in \cite{Russ}. We also explain how to write the limit in terms of the conformal invariant quantities that are useful to us in folded domains.
\end{itemize}

\subsection*{Acknowledgements} The authors extend their warm thanks to Avelio Sepúlveda for his contribution of Appendix \ref{app:ALE}. ML and LR thank Misha Basok for useful discussions during the program \emph{Geometry, Statistical Mechanics, and Integrability} at IPAM, UCLA. ML and KR were supported by the FWF Standalone grant Spins, Loops and Fields P 36298 and the FWF SFB Discrete Random Structures 10.55776/F1002.
KR was supported by the Academy of Finland Centre of Excellence Programme grant number 346315 “Finnish centre of excellence in Randomness and STructures” (FiRST). LR was
partially supported by the DIMERS project ANR-18-CE40-0033 funded by the French National Research Agency.

\section{Kenyon's formula} \label{sec:KForm}
\subsection{Setup and formula}

We begin with some definitions. Let $G=(V,E)$ be a finite, connected, bipartite, planar graph. Its vertices $V$ are partitioned into white and black vertices. Let $\partial$ be a \textit{strict} subset of the vertices incident to the external face of $G$, which contains the same number of black and white vertices. We form a graph $G^\times$ which can be described as taking two copies of $G$, identifying the vertices in the two copies on $\partial$ (and not the rest of the graph), and then adding ``diagonal'' edges which join a vertex $u$ in one copy of $G$ to a vertex $v$ in the other copy, whenever $u$ and $v$ form an edge in $G$. One can think of the two copies of $G$ being ``folded'' together along $\partial$. See Figure \ref{fig-G-and-Gx} for a small example.

\begin{figure}[h]
    \resizebox{0.6\textwidth}{!}{%
        \includegraphics[]{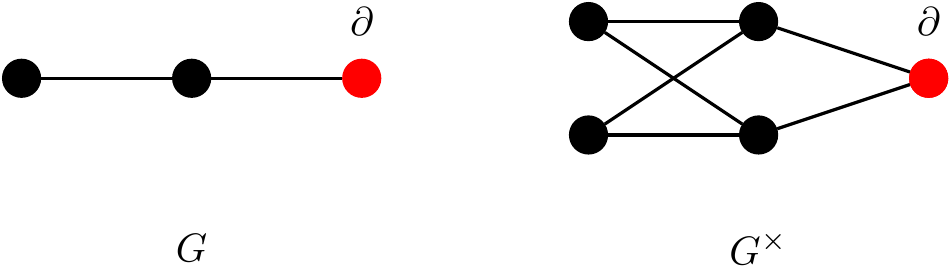}    
    }
    \caption{A very simple example of the graph $G^\times$ (right), given a graph $G$ (left). The set $\partial$ is in this case the single red vertex.}
    \label{fig-G-and-Gx}
\end{figure}

Formally, we write $G^\times=(V^\times,E^\times)$ for the graph with vertex set $V^\times:=\partial\cup((V\setminus\partial) \times \{1,2\})$, and edge set $E^\times$ given by edges:
\begin{itemize}
    \item $\{(u,i),(v,j)\}$ for all $i,j=1,2$ and $u,v\in V\setminus\partial$ such that $\{u,v\}\in E$;
    \item $\{(u,i),w\}$ for all $i=1,2$ and $u\in V\setminus\partial$, $w\in\partial$ such that $\{u,w\}\in E$;
    \item $\{w,w'\}$ for all $w,w'\in\partial$ and $\{w,w'\}\in E$.
\end{itemize}
This graph is also bipartite, and we define the colour (white or black) of $(u,i)$ as the colour of $u$. We denote by $p$ the \emph{projection} from $G^\times$ onto $G$: for all $u \in V^{\times} \cap \partial$, $p(u) = u$, for all $(u,i) \in V^{\times}$, $p(u,i) = u \in V$ and for all $\{x,y\} \in E^{\times}$, $p(\{x,y\}) = \{p(x),p(y)\} \in E$.\par 
A dimer configuration on $G^\times$ is a set of edges $m\subset E^\times$ such that every vertex in $V^\times$ is incident to exactly one edge in $m$. The projection of a dimer configuration $m$ is the set of edges $p(m) \subset E$ obtained by projecting all edges of $m$. In the case $\partial=\varnothing$, the projection of a dimer configuration on $G^\times$ is a configuration of disjoint closed loops and double edges on $G$: it is a double dimer configuration. In the case $\partial\neq\varnothing$, the projection of a dimer configuration on $G^\times$ is a configuration of disjoint closed loops, double edges, and also paths from $\partial$ to $\partial$, which we call \textit{arcs} (the arcs can be of length 1). Let $\Om$ denote the set of such configurations of loops, double edges, and arcs on $G$.\\

Let $\Phi$ be some $\SL_2(\CC)$ connection in the bulk of $G$. That is, $\Phi$ assigns each directed edge $uv$ ($u,v\in V\setminus\partial$) a matrix $\phi_{u,v}\in\SL_2(\CC)$, such that $\phi_{v,u}=\phi_{u,v}^{-1}$. For a loop $C=\{e_1,\dots,e_{2n}\}$ in $G$, define the monodromy of $\Phi$ around $C$, $\phi_C$ to be the product $\phi_{e_1}\circ\cdots\circ\phi_{e_{2n}}$.
Note this depends on the choice of $e_1$ and the orientation or the loop; $\tr(\phi_C)$, however, does not. For an edge $\{u,w\}$ with $w\in\partial$ and $u\in V\setminus\partial$, let $\psi_{\{u,w\}}\in\CC^2$. Fix some basis of each copy of $\CC^2$ so that each $\phi_{v,u}$ can be written as a matrix, and each $\psi_{\{u,w\}}$ as a column vector. 

We say a set of Kasteleyn phases on the edges of a bipartite, planar graph $G$ is a function $\xi:E\to\CC$ with $|\xi_e|=1$ for all $e\in E$, such that for every simple loop of edges $e_1\dots e_{2n}$ in $G$, one has
\be\label{kast:eq:kast-weighting}
    \frac{\xi_{e_1}\xi_{e_3}\cdots\xi_{e_{2n-1}}}{\xi_{e_2}\xi_{e_4}\cdots\xi_{e_{2n}}}=(-1)^{n+1}.
\ee
It is well known that such phases exist (see for example Theorem 4.1 of \cite{KastOrient}) and in fact can be taken to be real i.e.~signs $\pm1$ on the edges. One can show that \eqref{kast:eq:kast-weighting} holds for all loops surrounding an even number of vertices. Moreover if one replaces $\xi_{x,y}$ by $\eps(x)\eps(y)\xi_{x,y}$ for some function $\eps:V\to \{\pm1\}$, the new phases are also a set of Kasteleyn phases: this is called a gauge transform. 
Let $\bar\partial$ be a connected, strict subset of the boundary vertices, with $\partial\subset\bar\partial$.  By using a guage transform, one can assume that 
the phases on the boundary edges (incident to the external face) alternate between $+1$ and $-1$ in $\bar\partial$. 
We fix such real phases $\xi$ for the rest of this section.

We assign weights $\nu(e)$ to the (undirected) edges of $G^\times$ as follows. If $\{(u,i),(v,j)\}$ is some edge in the bulk of $G^\times$, with $u$ white and $v$ black, let $\nu_{\{(u,i),(v,j)\}}:=(\phi_{u,v})_{i,j}$, (the $i,j$ entry of $\phi_{u,v}$.
If $w\in\partial$ and $(u,i)$ is adjacent to $w$ but not in $\partial$, we set $\nu_{\{(u,i),w\}}=(\psi_{\{u,w\}})_i$. If $w,w'\in\partial$ then let $\nu_{\{w,w'\}}=1$. Let $K$ be an antisymmetric matrix whose rows and columns are indexed by vertices of $G^\times$, defined as follows. For $x, y \in V^{\times}$ with $x$ white and $y$ black, let 
\bestar
    K_{x,y}=\xi_{p(x),p(y)}\nu_{x,y},
\eestar
and let $K_{y,x}=-K_{x,y}$, where $\xi$ are the Kasteleyn phases.\\

We assign an ordering to $V^\times$ which will be used in the Pfaffian of Proposition \ref{kast:prop:kast}. Let the ordering be: first the boundary vertices, ordered as they appear clockwise in $\ol{\partial}$, 
then the white vertices in some order such that each $(u,1)$ is directly followed by $(u,2)$, and then the black vertices, again in some order such that $(v,1)$ is directly followed by $(v,2)$. We can write our assumption that the Kasteleyn phases alternate on the boundary as $\xi_{w,b} = (-1)^{\mathds{1}\{w>b\}}$ for all edges linking two vertices of $\ol{\partial}$. 

A set of Kasteleyn phases satisfying this hypothesis is depicted in Figure \ref{kast-fig-real_weights}, where $\partial=\bar\partial$ is the set of edges on the $x$-axis. 

\begin{proposition}\label{kast:prop:kast}[Kenyon's formula]
    We have that
    \be\label{kast:eq:propkast}
        \pf K = \sum_{\om\in\Om} 
        \prod_{\mathrm{loops} \ C}\tr(\phi_C)
        \prod_{\mathrm{arcs} \ A}\psi_{w_A}^\T\phi_A\psi_{b_A},
    \ee
    where $\phi_C$ is the monodromy of the connection $\Phi$ around the loop $C$, $b_A$ (resp. $w_A$) is the black (resp. white) endpoint of the arc $A$ (note that the two endpoints of each arc have different colours), and if $e_1,\dots,e_k$ are the bulk edges of $A$ ordered towards $b_A$, $\phi_A=\phi_{e_1}\cdots\phi_{e_k}$. When an arc is a single edge $e$, we interpret the factor corresponding to it as $\nu_e=1$. 
\end{proposition}

\begin{figure}[h]
    \resizebox{0.9\textwidth}{!}{%
        \includegraphics[]{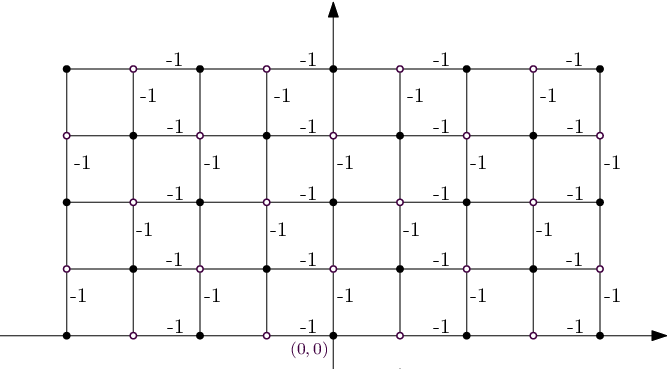}    
    }
    \caption{An example of a graph $G$, where we take $\partial=\ol{\partial}$ to be the edges on the $x$-axis, and real Kasteleyn weights $\xi$ which alternate in sign along $\ol{\partial}$ as $\xi_{w,b} = (-1)^{\mathds{1}\{w>b\}}$, where the ordering on $\partial$ increases clockwise. (We only write the $-1$ phases in the diagram for clarity; the rest are 1).}
    \label{kast-fig-real_weights}
\end{figure}

\begin{remark} We note an independent work of Douglas, Kenyon and Shi \cite[Theorem 4.1]{webs} considers similar interpretations of double dimers as in our proof of this result.
\end{remark}
 
\begin{remark}\label{kast:rmk:prop}
    A version of the proposition above also holds for other Kasteleyn phases $\zeta$ (in particular complex ones) if they can be obtained from the real weighting $\xi$ above by a gauge transformation $\zeta_{u,v}=\eps(u)\eps(v)\xi_{u,v}$, for some $\eps:V \to \CC$. This corresponds, for each $u\in V$, to multiplying the row and column of $K$ corresponding to $(u,1)$ and $(u,2)$ (or just $u$ if $u \in \partial$) by $\eps(u)$, which multiplies the Pfaffian by $\eps(u)^2$ (or $\eps(u)$ if $u \in \partial$). One then obtains the right hand side of the proposition, multiplied by the scalar $\prod_{u\in V \setminus \partial}\eps^2(u)\prod_{u \in \partial}\eps(u)$.
\end{remark}

Recall that the usual complex Kasteleyn phases adapted to discrete holomorphy for $\ZZ^2$ are given by 
\be\label{kast-eq-conplex-kast-weights}
    \zeta_{b,w}=b-w
\ee
(thinking of $b,w\in\CC$), that is, at each white vertex, the weights of the four adjacent edges starting at the right-going edge and proceeding anticlockwise are $1,i,-1,-i$ respectively. Using the gauge function $\eps(u)=(-i)^{\mathds{1}\{u \text{ has odd vertical coordinate}\}}(-1)^{\mathds{1}\{u  \ \text{is white}\}}$, and the real Kasteleyn phases $\xi$ depicted in Figure \ref{kast-fig-real_weights}, $\zeta_{u,v}:=\eps(u)\eps(v)\xi_{u,v}$ are the Kasteleyn phases of discrete holomorphy of Equation \eqref{kast-eq-conplex-kast-weights}. Remark \ref{kast:rmk:prop} gives us the corollary of Proposition \ref{kast:prop:kast} which we will use in the following sections.

\begin{corollary}\label{kast:corr}
    Let $G$ have the conditions from the above, and be a simply connected subgraph of the upper half of $\ZZ^2$. Then \eqref{kast:eq:propkast} holds with $K$ defined using the Kasteleyn phases $\zeta$ (up to a global gauge factor of modulus $1$). 
\end{corollary}


\subsection{Proof of Kenyon's formula}

We use the following formula for a Pfaffian. Place the vertices of $G^\times$ clockwise around the edge of a disc $D$, in the ordering given above. For a pairing $\pi$ of the vertices, draw curves in $D$ connecting each of the pairs, producing a set of intersecting curves $D_\pi$. Define $k(\pi)$ to be the number of crossings of these curves. The value of $k(\pi)$ depends on how one draws the curves, but $(-1)^{k(\pi)}$ does not. Then 
\be\label{kast:eq:pfaffian}
    \pf K := \sum_{\pi}(-1)^{k(\pi)}\prod_{\substack{xy\in\pi\\x<y}}K_{x,y},
\ee
where the sum is over all pairings of indices (in our case, the indices are vertices of $G^\times$). In particular, if the pairing $\pi$ is given by a dimer configuration $m$ on $G^\times$, $(-1)^{k(m)}$ is well defined. 

Recall that $p$ projects from $G^\times$ to $G$. By extension, $p$ projects a dimer configuration on $G^\times$ to a configuration $p(m):=\om$ on $G$ of disjoint loops (which can be of length 2) and arcs from $\partial$ to $\partial$ (which can be of length 1). We write $\Om$ for the set of all possible such configurations $\om$. A dimer configuration $m$ is then specified by $\om=p(m)$, along with a certain list of indices coming from the two copies for each loop, double edge and arc in $\om$. The first part of the proof, Lemma \ref{kast:lem:sign-formula}, writes $(-1)^{k(m)}$ in terms of these loops, arcs, and indices.\\

Let us make this precise. A configuration $\om$ is a set of disjoint loops $C$ and arcs $A$. Take a loop of length 2 (a doubled edge) in $\om$, located on some edge $e = \{u,v\}$ with $u$ white and $v$ black. Take an index $i_e \in \{1,2\}$. One then has (part of) the corresponding dimer configuration, given by the edges $\{(u,1),(v,1)\}$ and $\{(u,2),(v,2)\}$ if $i_e=2$ and $\{(u,1),(v,2)\}$ and $\{(u,2),(v,1)\}$ if $i_e=1$.

For a loop $C$ in $\om$ of length larger than 2, label its vertices $u_1, v_1, \dots, u_k, v_k$ in counterclockwise order round the loop, where $v_i$ are all black, $u_i$ all white (See Figure \ref{kast-fig-cycles-1}). Let $i_1, j_1, \dots, i_k, j_k\in\{1,2\}$. The dimer configuration corresponding to this list of indices is: for each pair $u_l,v_l$, let the edge $\{(u_l,i_l),(v_l,j_l)\}$ be in the configuration, and for each pair $v_l,u_{l+1}$, let the edge        
 $\{((v_l,3-j_l),(u_{l+1},3-i_{l+1})\}$ be in the configuration. See Figure \ref{kast-fig-cycles-2}. In the figure, the vertices on the upper row are $(u_1,2)$, $(v_1,2)$, $(u_2,2)$, etc. The vertices below those are the same but with second coordinate equal to 1. The two dashed edges are the same edge.

\begin{figure}[h]
    \resizebox{0.6\textwidth}{!}{%
        \includegraphics[]{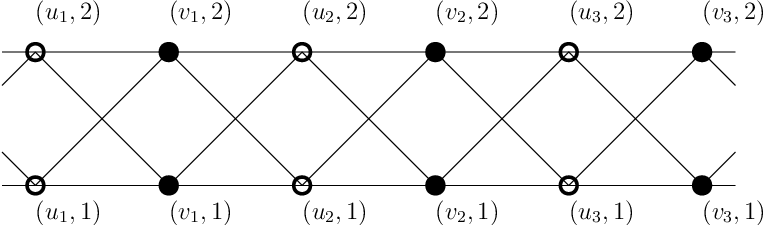}    
    }
    \caption{The vertices of $G^\times$ in a loop, with all the edges of $G^\times$ joining them.}
    \label{kast-fig-cycles-1}
\end{figure}

\begin{figure}[h]
    \resizebox{0.6\textwidth}{!}{%
        \includegraphics[]{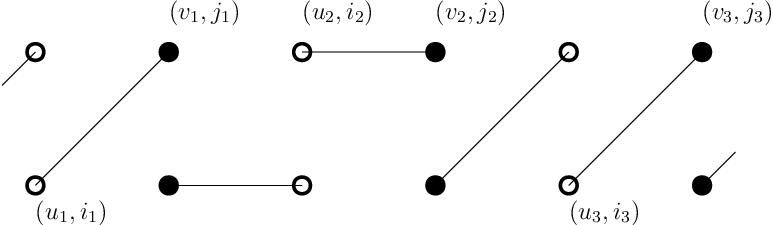}    
    }
    \caption{The vertices and edges of $G^\times$ in a loop.}
    \label{kast-fig-cycles-2}
\end{figure}

For an arc $A$ in $\om$, label its vertices along the arc $b_A, u_1, v_1, \dots, u_k, v_k, w_A$, where $u_i$ and $w_A$ are all white, $v_i$ and $b_A$ all black (see Figure \ref{kast-fig-arcs-1}. Then let $r_1,s_1,\dots,r_k,s_k\in\{1,2\}$. The dimer configuration corresponding to this list of indices is defined as in a loop, plus the two edges $\{b_A,(u_1,3-r_1)\}$ and $\{(v_k,3-s_k),w_A\}$. See Figure \ref{kast-fig-arcs-2}. An arc of length 1 needs no indices, as there is just one (boundary) edge in $G^\times$ that can produce it.

\begin{figure}[h]
    \resizebox{0.75\textwidth}{!}{%
        \includegraphics[]{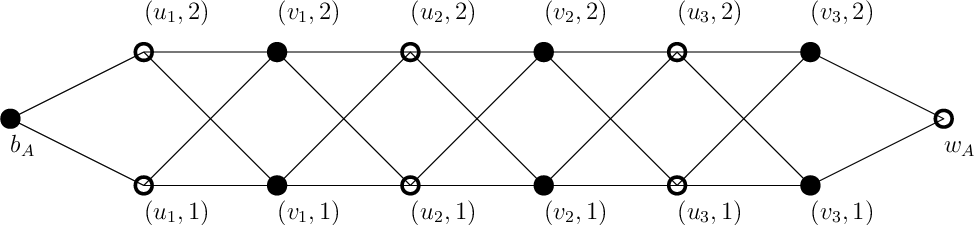}    
    }
    \caption{The vertices in $G^\times$ in an arc, with all the edges of $G^\times$ joining them.}
    \label{kast-fig-arcs-1}
\end{figure}

\begin{figure}[h]
    \resizebox{0.75\textwidth}{!}{%
        \includegraphics[]{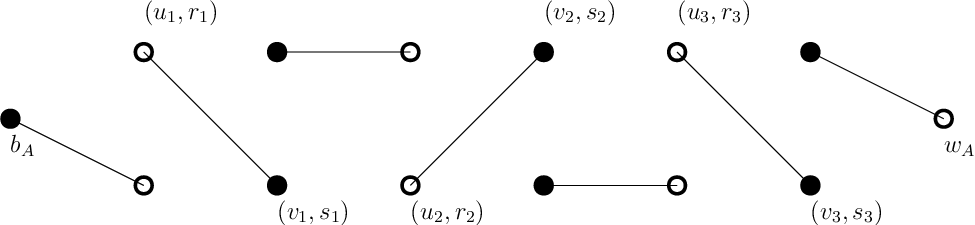}    
    }
    \caption{The vertices and edges in $G^\times$ in an arc.}
    \label{kast-fig-arcs-2}
\end{figure}

Using this notation for a dimer configuration, we have the following lemma.

\begin{lemma}\label{kast:lem:sign-formula}
    Let $m$ be a dimer configuration on $G^\times$, written as $\om=p(m)$ and a list of indices as above for each loop, doubled edge and non-trivial arc. Then
    \bestar
        \begin{aligned}
        (-1)^{k(m)}
        &=
        \prod_{\mathrm{doubled~edges} \ e}(-1)^{i_e} \prod_{\mathrm{cycles} \ C}(-1)^{1+i_1+j_1+\cdots+i_k+j_k}\\
        &\quad \times \prod_{\mathrm{non~trivial~arcs} \ A}(-1)^{r_1+s_1+\cdots+r_k+s_k}(-1)^{\mathds{1}\{b_A>w_A\}},
        \end{aligned}
    \eestar
    where the inequality in the final exponent is the order on the vertices of $G^\times$.
\end{lemma}

\begin{proof}
    Let $m$ be a dimer configuration on $G^\times$, and let $D_m$ be the intersecting curves given by $m$ as described above. The vertices of $G^\times$ are arranged clockwise on the boundary of a disc $D$ according to the order $<$, and each pair of vertices in an edge of $m$ are joined by a curve in $D_m$. Recall $k(m)$ is the number of crossings of these curves. One can permute the pairs of bulk vertices $(v,1),(v,2)$ with each other in $D_m$ and preserve $(-1)^{k(m)}$. For example, if the vertices $(v,1),(v,2),(u,1),(u,2)$ appear consecutively on the boundary of $D_m$, permuting them so they appear in the order $(u,1),(u,2),(v,1),(v,2)$ preserves $(-1)^{k(m)}$. We will use this property several times in the rest of the proof. The curves connecting these vertices now only cross each other and no other curve in $D_m$, as the vertices are consecutive.\\

    Let $e=(u,v)$ a doubled edge in $\omega = p(m)$ with $u$ white and $b$ black. Using the permuting operation described above, one can permute the vertices of $D_m$ so that the vertices of $p^{-1}(e)$ appear consecutively, in the order $(u,1),(u,2),(v,1),(v,2)$. Now the two curves in $D_m$ corresponding to the two edges in $m$ projecting on $e$ do not cross the other curves, and they cross each other once if and only if $i_e=1$, hence the contribution to $(-1)^{k_m}$ is $(-1)^{i_e}$.\\
    
    Let $C=(u_1,v_1,\dots,u_k,v_k)$ be a loop in $\om=p(m)$. Using the permuting operation described above, one can permute the vertices of $D_m$ so that the vertices of $p^{-1}(C)$ appear consecutively, in the order: $(u_1,1),(u_1,2)$, $(u_2,1),(u_2,2)$, \dots, $(u_k,1),(u_k,2)$, $(v_k,1),(v_k,2)$, $(v_{k-1},1),(v_{k-1},2)$, \dots, $(v_1,1),(v_1,2)$. See Figure \ref{kast-fig-no-diags-2}, in which the ordering starts at $(u_1,1)$ and proceeds clockwise. The curves connecting these vertices now only cross each other in $D_m$, and not other curves, as the vertices are consecutive. Write $k_{m,C}$ for the number of these crossings. We claim that

    \be\label{kast:eq:crossings-one-cycle}
        (-1)^{k_{m,C}}=(-1)^{1+i_1+j_1+\cdots+i_k+j_k}.
    \ee

    Let $m_0$ be a dimer configuration such that on the loop $C$, it takes the values $i_l=j_l=2$ for all $l=1,\dots,k$. This configuration has no diagonal edges in $C$; see Figure \ref{kast-fig-no-diags-1}. 

    \begin{figure}[h]
        \resizebox{0.75\textwidth}{!}{%
        \includegraphics[]{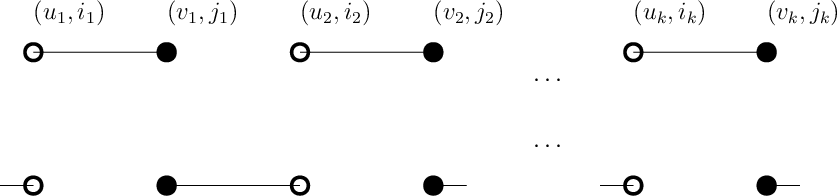}    
        }
        \caption{The dimer configuration $m_0$.}
        \label{kast-fig-no-diags-1}
    \end{figure}

    One sees that $(-1)^{1+i_1+j_1+\cdots+i_k+j_k} = -1$ in this case. Meanwhile, the portion of $D_{m_0}$ corresponding to $C$ can be drawn as:

    \begin{figure}[h]
        \resizebox{0.75\textwidth}{!}{%
        \includegraphics[]{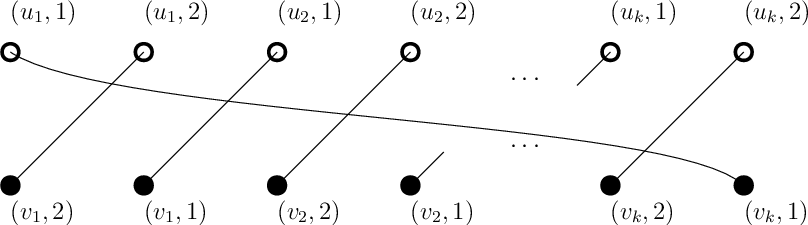}    
        }
        \caption{The vertices and edges of $m_0$ arranged in some part of $D_{m_0}$; the ordering of the vertices here is given by $<$.}
        \label{kast-fig-no-diags-2}
    \end{figure}
    
    and one can see that there are exactly $2k-1$ crossings, hence $k_{m_0,C}=-1$, and \eqref{kast:eq:crossings-one-cycle} holds. Note that in Figure \ref{kast-fig-no-diags-2} the vertices are ordered in $D_{m_0}$ according to $<$, whereas in Figure \ref{kast-fig-no-diags-1} they are arranged as in $G^\times$.

    Now let $m$ be some dimer configuration with some loop $C$ in $p(m)$, and let $v$ be some vertex in $C$. Let $m'$ be the same as $m$, with the edges incident to $(v,1)$, $(v,2)$ swapped - that is, the edges $\{x,(v,1)\}$, $\{(v,2),y\}$ changed to $\{x,(v,2)\}$, $\{(v,1),y)\}$. This changes the sum $i_1+j_1+\cdots+i_k+j_k$ by exactly 1. Meanwhile, the swapping changes the number of crossings $k_{m,C}$ by exactly 1. Hence the equation \eqref{kast:eq:crossings-one-cycle} holds for $m'$ if it holds for $m$. Inductively, \eqref{kast:eq:crossings-one-cycle} holds for all dimer configurations $m$ and all loops $C$ of $m$.\\

    Let us turn now to arcs. Let $m$ be a dimer configuration, and $A$ an arc of $\om=p(m)$. As with loops, we can permute the bulk (non-boundary) vertices of $p^{-1}(A)$ in the graph $D_{m}$ so that they are consecutive, and in the same order as described above with a loop, without changing $(-1)^{k(m)}$. The curves between these bulk vertices in $D_m$ can now only cross each other. The curves in $D_m$ incident with $b_A$ and $w_A$ (the boundary vertices of $A$) may still cross the bulk curves of $A$ (and the curves of other arcs). Assume that $A$ is non trivial and let $k_{m,A}$ be the number of crossings in $D_m$ between curves connecting vertices of $p^{-1}(A)$. We will show that $k_{m,A}$ satisfies
    \be\label{kast:eq:crossings-one-arc}
        (-1)^{k_{m,A}}
        =
        (-1)^{r_1+s_1+\cdots+r_k+s_k}
        (-1)^{\mathds{1}\{b_A>w_A\}}.
    \ee
    Recall that the ordering of the boundary vertices in $D_m$ is the order they appear on the boundary of $G$. Taking $m_0$ analogously to above, with $r_l=s_l=2$ for all $l=1,\dots,k$, we can draw the portion of $D_{m_0}$ corresponding to $A$ as:

    \begin{figure}[h]
        \resizebox{0.85\textwidth}{!}{%
        \includegraphics[]{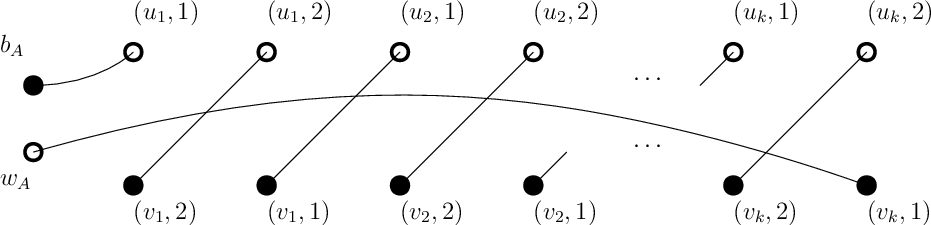}    
        }
        \caption{}
        \label{kast-fig-no-diags-arc}
    \end{figure}

    The case drawn is the case where $w_A<b_A$, and there are $2k-1$ crossings. In the case $w_A>b_A$, the positions of $w_A$ and $b_A$ are exchanged, and there is exactly one more crossing. Hence \eqref{kast:eq:crossings-one-arc} holds. By the same inductive argument as used for loops above, \eqref{kast:eq:crossings-one-arc} holds for all dimer configurations $m$ and all arcs $A$ in $\om=p(m)$.  \\

    It remains to count the crossings between curves corresponding to different arcs (including trivial arcs). We claim that this number is always even. Indeed, in $G$, arcs are planar and non-crossing. The curves for each arc in $D_m$ can be seen as a deformation of the actual arcs in $G$, obtained by identifying the bulk vertices of each arc $A$ (using the fact that we can permute these vertices out of the way as above), and then deforming. Deforming the arcs in this way does not change $(-1)^{k(m)}$, so the number of crossings produced must be even. Taking this result, \eqref{kast:eq:crossings-one-cycle} and \eqref{kast:eq:crossings-one-arc} together, one obtains the lemma.\\
\end{proof}

\begin{proof}[Proof of Proposition \ref{kast:prop:kast}]
    We have that
    \be\label{kast:eq:prop-proof-step1}
        \begin{split}
            \pf K 
            &= 
            \sum_{m}(-1)^{k(m)}\prod_{\substack{xy\in m\\x<y}}K_{x,y}\\
            &= 
            \sum_{\om\in\Om}\sum_{m\in p^{-1}(\om)}(-1)^{k(m)}
             \prod_{\mathrm{doubled~edges} \ e}\prod_{\substack{xy\in e\\x<y}}K_{x,y} \prod_{\mathrm{loops} \ C}\prod_{\substack{xy\in C\\x<y}}K_{x,y}
            \prod_{\mathrm{arcs} \ A}\prod_{\substack{xy\in A\\x<y}}K_{x,y}\\
            &= 
            \sum_{\om\in\Om}\sum_{\mathrm{indices}} \prod_{\mathrm{doubled~edge} \ e}(-1)^{i_e} \prod_{\substack{xy\in e\\x<y}}K_{x,y} \prod_{\mathrm{loops} \ C}(-1)^{1+i_1+\cdots+j_k}\prod_{\substack{xy\in C\\x<y}}K_{x,y} 
             \\
            &\quad \prod_{\mathrm{trivial~arcs} \ A} \prod_{\substack{xy\in A\\x<y}}K_{x,y} \prod_{\substack{\mathrm{non-trivial} \\ \mathrm{arcs} \ A}}(-1)^{r_1+s_1+\cdots+r_k+s_k + \mathds{1}\{b_A>w_A\}} 
            \prod_{\substack{xy\in A\\x<y}}K_{x,y},
        \end{split}
    \ee
    where the sum over ``indices'' is the sum over choices of indices $i$ for doubled edges, $i_1,j_1, \dots,i_k,j_k$ for loops and $r_1,s_1,\dots,r_k,s_k$ for arcs, respectively, and the second equality uses Lemma \ref{kast:lem:sign-formula}. First of all, for a doubled edge $e = \{u,v\}$,
    \bestar
        \sum_{i_e=1,2}(-1)^{i_e} \prod_{\substack{xy\in e\\x<y}}K_{x,y} = K_{(u,1),(v,1)}K_{(u,2),(v,2)}-K_{(u,1),(v,2)}K_{(u,2),(v,1)} = \det(\phi_e) = 1.
    \eestar
    \\
    
    For a trivial arc $A=w_Ab_A$,
    \bestar
        \prod_{\substack{xy\in A\\x<y}}K_{x,y} = (-1)^{\mathds{1}\{w_A > b_A\}}\xi_{w_A,b_A}\nu_{w_A,b_A} = \nu_{w_A,b_A}=1, 
    \eestar
    by the assumption on the Kasteleyn phases.
    \\
    
    Now, for a loop $C$, we show that
    \be\label{kast:eq:trace-cycle-calc}
        \sum_{i_1,j_1,\dots,i_k,j_k}(-1)^{1+i_1+\cdots+j_k}\prod_{\substack{xy\in C\\x<y}}K_{x,y} 
        =
        (-1)^{\tfrac{|C|}{2}}\tr (\phi_C),
    \ee
    where, recall, $\phi_C$ is the monodromy of the connection $\Phi$ around the loop $C$, and by $|C|$ we mean the number of edges in $C$. Recalling that in the ordering of vertices, all white vertices are before all black vertices, the product on the left hand side is 
    \bestar
        \begin{split}
            \prod_{\substack{xy\in C\\x<y}}K_{x,y} 
            &=
            K_{(u_1,i_1),(v_1,j_1)} K_{(u_2,3-i_2),(v_1,3-j_1)} K_{(u_2,i_2),(v_2,j_2)}  \cdots K_{(u_k,i_k),(v_k,j_k)} K_{(u_1,3-i_1),(v_k,3-j_k)}\\
            &=
            \left(\prod_{\substack{xy\in C\\x<y}}\xi_{x,y}\right)
            (\phi_{u_1,v_1})_{i_1,j_1}
            (\phi_{u_2,v_1})_{3-i_2,3-j_1}
            \cdots 
            (\phi_{u_k,v_k})_{i_k,j_k}
            (\phi_{u_1,v_k})_{3-i_1,3-j_k}.
        \end{split}
    \eestar
    We now use that for matrices $\phi\in\SL_2(\CC)$, $\phi_{3-i,3-j}=(-1)^{i+j}(\phi^{-1})_{j,i}$, which one can verify by hand. We replace each factor $(\phi_{u_{l+1},v_l})_{3-i_{l+1},3-j_l}$ with $(-1)^{j_l+i_{l+1}}(\phi_{v_l,u_{l+1}})_{j_l,i_{l+1}}$, and the above becomes:
    \bestar
        \begin{split}
            \left(\prod_{\substack{xy\in C\\x<y}}\xi_{x,y}\right)
            \left((-1)^{i_1+\cdots+j_k}\right)
            (\phi_{u_1,v_1})_{i_1,j_1}
            (\phi_{v_1,u_2})_{j_1,i_2}
            \cdots
            (\phi_{u_k,v_k})_{i_k,j_k}
            (\phi_{v_k,u_1})_{j_k,i_1}.
        \end{split}
    \eestar
    Now the factor $(-1)^{i_1+\cdots+j_k}$ in the above cancels with the same factor on the left hand side of \eqref{kast:eq:trace-cycle-calc}, and the sum over all $i_1,j_1,\dots,i_k,j_k$ produces the trace. The product of factors $\xi_e$ around a loop $C$ is $(-1)^{\tfrac{|C|}{2}+1}$ by the Kasteleyn phases; the $+1$ in this exponent cancels with the remaining factor $-1$ in the left hand side of \eqref{kast:eq:trace-cycle-calc}. Hence \eqref{kast:eq:trace-cycle-calc} holds.\\
    
    It remains to prove the analogous formula for arcs:
    \be\label{kast:eq:trace-arc-calc}
        \sum_{r_1,s_1,\dots,r_k,s_k}(-1)^{r_1+s_1+\cdots+r_k+s_k + \mathds{1}\{b_A>w_A\}}
        \prod_{\substack{xy\in A\\x<y}}K_{x,y} 
        =
        (-1)^{\tfrac{1}{2}(|A|-1)}
        \psi_{w_A}^\T\phi_A\psi_{b_A},
    \ee
    where by $|A|$ we mean the number of edges in $A$, and where, recall, $\phi_A$ is the product of the connection along the bulk edges of $A$, directed towards $b_A$. The proof is very similar to the loop case. 
    We obtain:
    \bestar
        \begin{split}
            \prod_{\substack{xy\in A\\x<y}}K_{x,y} 
            &=
            \left(\prod_{\substack{xy\in A\\x<y}}\xi_{x,y}\right)
            (\psi_{b_A,u_1})_{3-r_1}
            (\phi_{u_1,v_1})_{r_1,s_1} 
            (\phi_{u_2,v_1})_{3-r_2,3-s_1}\\
            & \hspace{3cm}
            \cdots 
            (\phi_{u_k,v_{k-1}})_{3-r_k,3-s_{k-1}}
            (\phi_{u_k,v_k})_{r_k,s_k}
            (\psi_{w_A,v_k})_{3-s_k}\\
            &=
            \left(\prod_{\substack{xy\in A\\x<y}}\xi_{x,y}\right)
            (-1)^{r_1+s_1+\cdots+r_k+s_k}
            (\psi_{w_A,v_k})_{3-s_k}
            (\phi_{v_k,u_k})_{3-s_k,3-r_k}
            (\phi_{u_k,v_{k-1}})_{3-r_k,3-s_{k-1}}\\
            & \hspace{3cm}
            \cdots       
            (\phi_{u_2,v_1})_{3-r_2,3-s_1} 
            (\phi_{v_1,u_1})_{3-s_1,3-r_1}
            (\psi_{b_A,u_1})_{3-r_1}
        \end{split}
    \eestar
    (where in the last line we also reorder the factors for notational convenience). Similarly to the loop case, the second equality is obtained by using the formula $\phi_{u_l,v_l}(r_l,s_l)=(-1)^{r_l+s_l}\phi_{v_l,u_l}(3-s_l,3-r_l)$ for each $l=k,\dots,2$. Now the factor $(-1)^{r_1+s_1+\cdots+r_k+s_k}$ cancels with the same appearing on the left hand side of \eqref{kast:eq:trace-arc-calc}, and the sum over all indices $r_1,s_1,\dots,r_k,s_k$ produces the product $\psi_{w_A}^\T\phi_A\psi_{b_A}$, where $\phi_A=\phi_{v_k,u_k}\phi_{u_k,v_{k-1}}\cdots\phi_{u_2,v_1}\phi_{v_1,u_1}$. 
    
    One then only needs that $\prod_{\substack{xy\in A\\x<y}}\xi_{x,y} = (-1)^{\mathds{1}\{b_A>w_A\}}(-1)^{\tfrac{1}{2}(|A|-1)}$. Recall that we constructed $\xi$ to be a set of Kasteleyn phases such that $\xi_{w,b} = (-1)^{\mathds{1}\{w>b\}}$ on boundary edges $\{w,b\}$ (apart perhaps from some edge $e_0$). Let $C_A$ be the loop in $G$ made up of an arc from $w_A$ to $b_A$, and then the boundary edges from $w_A$ to $b_A$ (not in the direction including $e_0$), which we denote $\partial C_A$. One can show from the above that $\prod_{\substack{xy\in \partial C_A\\x<y}}\xi_{x,y}= (-1)^{\mathds{1}\{b_A>w_A\}}(-1)^{\tfrac{1}{2}(|\partial C_A|+1)}$. Then 
    \bestar
        \begin{split}
            \prod_{\substack{xy\in A\\x<y}}\xi_{x,y}
            &=
            \prod_{\substack{xy\in C_A\\x<y}}\xi_{x,y}\prod_{\substack{xy\in \partial C_A\\x<y}}\xi_{x,y}\\
            &=
            (-1)^{\tfrac{1}{2}|C_A|+1}(-1)^{\mathds{1}\{b_A>w_A\}}(-1)^{\tfrac{1}{2}(|\partial C_A|+1)}\\
            &=
            (-1)^{\mathds{1}\{b_A>w_A\}}(-1)^{\tfrac{1}{2}(|A|-1)}.
        \end{split}
    \eestar

    We now have both \eqref{kast:eq:trace-cycle-calc} and \eqref{kast:eq:trace-arc-calc}, which, upon substituting into \ref{kast:eq:prop-proof-step1}, produces the right hand side of the Proposition \ref{kast:prop:kast}, multiplied by the factor $\prod_C(-1)^{\tfrac{1}{2}|C|}\prod_A(-1)^{\tfrac{1}{2}(|A|-1)}$. This final factor disappears on seeing that $\tfrac{1}{2}(\sum_{C}|C|+\sum_A(|A|-1))$ is half the number of vertices of $V\setminus\partial$, which we assumed to be even.
   
\end{proof}

\section{Folded and shifted dimers}\label{section:l&a}
\subsection{Introduction of the models}
\subsubsection{Temperleyan and piecewise Temperleyan approximations}
We will use as much as possible the notation of \cite{Ken14}. Let $U\subset\CC$ be a simply connected (bounded) set with a smooth boundary. We say that $(\Ue)_{\eps >0}$ is a sequence of graphs approximating $U$ if the following holds:

\begin{definition}[Approximating sequence]\label{def:approx}
    For every $\eps >0$, $\Ue$ is a graph with vertex set $V(\Ue) \subset \bar{U} \cap \eps\ZZ^2$ and edges connecting points at distance $\eps$: if $u,v \in V(\Ue)$, $u \sim v$ in $\Ue$ if and only if $u \sim v$ in $\eps\ZZ^2$. For every $\eps >0$, $\Ue$ is simply connected: if we cover every vertex of $\Ue$ by a square of side $\eps$ centered at the vertex, we obtain a simply connected set. The vertex boundary of $\Ue$ (i.e. the vertices of $\Ue$ which are connected in $\eps \ZZ^2$ to at least one vertex not in $\Ue$) is within $O(\eps)$ of $\partial U$. 
\end{definition}
An example is drawn on Figure \ref{fig:Temperley}: the vertices of the graph $\Ue$ are the black squares enclosed by $U$. Its edges are not drawn, they link nearest neighbours.

From now on, let $U^r$ be a simply connected open set with smooth boundary which is symmetric by reflection along the horizontal axis. For such symmetric sets $U^r$, we denote by $U = U \cap (\RR \times \RR_{>0})$ the restriction of $U^r$ to the strict upper half plane. Let $z \in U$ be fixed. For technical reasons, we assume that there exists $\zp \in \partial U^r$ and $\d >0$ such that the boundary of $U^r$ is horizontal in the neighbourhood $B(\zp, \d)$. We also assume that $U^r$ is contained in the half-plane below this flat horizontal boundary.
\begin{remark}
    This technical hypothesis is necessary for the results of Appendices \ref{appendix:A} and \ref{app:B} to hold. It could be lightened a bit by using the comments after Lemma \ref{lem:Green:boundary}: $\d$ could go to zero with $\eps$ (but we must have $\eps = o(\d)$). 
\end{remark}
Consider an approximating sequence $(\Uer)_{\eps>0}$ of $\Ud$ which is \emph{symmetric} by reflection along the horizontal axis and such that for every $\eps >0$ the vertex boundary of $\Uer$ is horizontal in the neighbourhood $B(\zp, \d)$.

\bigskip

\paragraph{\textbf{Symmetric Temperleyan approximation of $U^r$.}} We define the bipartite Temperleyan domain $\Ged$ associated with $\Uer$ which has two types of black vertices: $B_1(\Ged)$ corresponds to the vertices of $\Uer$ from which we remove the rightmost vertex on the horizontal axis $b_0$, and $B_0(\Ged)$ corresponds to the inner faces of $\Uer$. The graph $\Ged$ has two types of white vertices $W_0(\Ged)$ corresponding to the vertical edges of $\Uer$ and $W_1(\Ged)$ corresponding to the horizontal edges of $\Uer$.  There are no edges between vertices of the same colour ($\Ged$ is bipartite), and there is an edge between $b \in B(\Ged) = B_0(\Ged) \cup B_1(\Ged)$ and $w \in W(\Ged) = W_0(\Ged) \cup W_1(\Ged)$ if and only if the edge $w$ is incident to the face or vertex $b$ in $\Uer$ (see Figure \ref{fig:Temperley}). We also denote by $B_0, B_1$ the vertices and faces of $\eps \ZZ^2$ and by $W_0, W_1$ the vertical and horizontal edges of $\eps \ZZ^2$, so we can write for example $B_0(\Ged) \subset B_0$. By Temperley's bijection \cite{TemFish,KPW}, perfect matchings of $\Ged$ are in bijection with spanning trees of $\Uer$ rooted at $b_0$  (see Kenyon). The sequence $(\Ged)_{\eps >0}$ is a \emph{symmetric Temperleyan approximation} of the open set $\Ud$. Dimer configurations of $\Ge^r$ are in bijection with spanning trees of $\Ue^r$ rooted at $b_0$. 

\bigskip

\paragraph{\textbf{Temperleyan approximation of $U$.}}For every $\eps >0$, let $\Ue$ be the graphs obtained by removing from $\Ue^r$ the vertices in the strict lower half plane and the edges incident to them. The sequence $(\Ge)_{\eps > 0} = (\Ge^1)_{\eps >0}$ obtained by restricting $\Ged$ to the upper half plane (removing all vertices of $\Ged$ in $\RR \times \RR_{<0}$, and the edges incident to them) is a \emph{Temperleyan approximation} of the open set $U$. Dimer configurations of $\Ge = \Ge^1$ are in bijection with spanning trees of $\Ue$ rooted at $b_0$. 

\bigskip

\paragraph{\textbf{Piecewise Temperleyan approximation of $U$.}} The sequence $(\Ge^2)_{\eps >0}$ obtained by restricting $\Ge^r$ to the \emph{strict} upper half plane (i.e. removing all vertices of $\Ge^r$ in $\RR \times \RR_{\leq 0}$ and the edges incident to them) is a \emph{piecewise Temperleyan approximation} of the open set $U$. Indeed, it has exactly two convex white corner, $v_l^*$ and $v_r^*$ respectively on the left and right of the horizontal axis, see the definition in Section 5.1 of \cite{Russ}.

\subsubsection{Loops and arcs on a Temperleyan approximation}


We say that a subset of edges $\omega \subset E(\Ge)$ is a loops and arcs configuration if it is made of non-intersecting loops in the bulk, arcs connecting points on the horizontal axis and doubled edges, such that every vertex of $E(\Ge)$ lies on a loop, arc or doubled edge. We denote by $\Omega(\Ge)$ the set of loops and arcs configurations on $\Ge$. There are two natural ways to obtain a random loops and arcs configurations:
\begin{itemize}
\item
Folding a random uniform dimer configuration on $\Ged$ along the horizontal axis
\item
Superimposing a random uniform dimer configuration on $\Ge^1$ with an independent random uniform dimer configuration on $\Ge^2$  
\end{itemize}
In both cases, we obtain a random loops and arcs configuration $\omega \subset E(\Ge)$ of $\Ge$ (see Figure \ref{fig:Temperley}) which we call a \emph{folded} or \emph{shifted} dimer configuration. We denote by $\mu^1_{\eps}, \mu^2_{\eps}$ their respective probability laws on $\Omega(\Ge)$ and by $\Ee^1, \Ee^2$ the expectation with respect to these laws. 

Given $\om \in \Om(\Ge)$, we say that an arc $A$ of $\om$ \emph{encloses} a point $z \in U$ if any continuous path in $U$ from $z$ to $\zp$ crosses the arc $A$. For every $\eps >0$, define two integer-valued random variables: $\nez$ is the number of arcs enclosing the fixed point $z \in U$, and $\oez$ is the indicator that there is an odd number of arcs enclosing the point $z$ (that is the parity of $\nez$). Our main result describes the joint law of $\nez$ and $\oez$ in the scaling limit for the folded and shifted models.

\subsubsection{Discrete and continuous Green function.}
In this paragraph, we recall the notation and results of \cite{Ken14} on continuous Green functions. Let $g : u,v \in U^r \to g(u,v) \in \RR$ be the continuous Green function with Dirichlet boundary conditions on $U^r$. Let $\tilde{g}: u,v \in U^r \to g(u,v) \in \CC$ be the analytic function of $v$ whose real part is the Dirichlet Green function. Define its Wirtinger derivatives 
\be\label{eq:def:Fpm}
    \forall u, v \in U^r,
    \left\{
        \begin{array}{ll}
            F_-(u,v) & = \frac{\partial \tilde{g}(u,v)}{\partial \overline{u}}\\
            F_+(u,v) & = \frac{\partial \tilde{g}(u,v)}{\partial u} \\
        \end{array}
    \right..
\ee
The holomorphic derivative $F_+$ is an an analytic function of $(u,v)$ while the anti-holomorphic derivative $F_-$ is an analytic function of $(\overline{u},v)$. Lemma \ref{lem:Green:boundary} (which adapts Theorems 13 and 14 of \cite{Ken00} to our setting) of the appendix gives the asymptotic of the inverse Kasteleyn matrix on $\Ge^r$ in terms of $F_+$ and $F_-$, while Lemma \ref{lem:russ} (which adapts Theorem 6.1 of \cite{Russ} to our setting) of the appendix gives the asymptotic of the inverse Kasteleyn matrix on $\Ge^1$ and $\Ge^2$ in terms of $F_+$ and $F_-$

\begin{remark}\label{rem:Schwarz}
Recall that the boundary of $U^r$ is horizontal in the ball $B(\zp, \d/2)$ and that $U^r$ is contained in the half-plane below this horizontal boundary. Denote by $s$ the reflection along this horizontal boundary. Since the Green function $g$ has Dirichlet boundary conditions, i.e. it is $0$ on the boundary of $U^r$, the Schwarz reflection principle (see for example Section 6.5 of \cite{Ahlfors1966}) implies that $i\tilde{g}(u,v)$ as a function of $v$ extends to an analytic (except at $u$ and $s(u)$) function on $U^r \cup s(U^r)$. Hence $F_+$ and $F_-$ also extend analytically (as functions of $v$) to this domain, in particular they are well-defined at $\zp$ and they remain bounded near $\zp$ (and so do their derivatives). In the proof of Lemma \ref{lem:Green:boundary}, we will detail a discrete version of this argument developed by Kenyon in the proof of Theorem 14 of \cite{Ken00}.
\end{remark}

Let $\gamma$ be a path from $z$ to $\zp$, $\gamma \subset U$ except for the endpoint. Our main theorem expresses $n$-th moments of random variables in terms of $n$-fold integrals along $\gamma$ of $F_+$ and $F_-$. In particular, note that $F_-$ is an analytic function of $(\overline{u},v)$, $F_+$ is an analytic function of $(u,v)$, $\overline{F_-}$ is an analytic function of $(u,\overline{v})$ and $\overline{F_+}$ is an analytic function of $(\overline{u},\overline{v})$. Hence, to have $n$-fold integrals of analytic functions along $\gamma$, for $\sigma\in\{\pm1\}^n$, we will be looking at terms of the form 
\be\label{eq:analytic:naive}
    \idotsint_{\gamma} \prod_{i=0}^{n-1}F^{(\sigma_i)}_{-\sigma_i\sigma_{i+1}}(\overline{z_{i+1}},z_i) \prod_{i=0}^{n-1} dz^{(\sigma_i)}_i,
\ee
where $dz_{k_i}^{(1)}=dz_{k_i}$ and $dz_{k_i}^{(-1)}=d\ol{z_{k_i}}$, and similar for $F_+$ and $F_-$, and where the subscript $-\sigma_i\sigma_{i+1}$ should be read as $\pm$ when $-\sigma_i\sigma_{i+1}=\pm1$, respectively. This is well defined since the integral of an analytic function does not depend on the path.

\subsection{Main result and examples.}
For all $n \in \NN$, let
\bestar
    c_n(z,\zp) 
    = 
    2i^n\sum_{\sigma\in\{\pm1\}^n}
    \idotsint_{\gamma} \prod_{i=0}^{n-1}\sigma_i F^{(\sigma_i)}_{-\sigma_i\sigma_{i+1}}(\overline{z_{i+1}},z_i) \prod_{i=0}^{n-1} dz^{(\sigma_i)}_i.
\eestar
For example,
\be\label{eq:ex:c}
    \left\{
    \begin{array}{ll}
    c_1(z,\zp) &= -4\Im\left\{\int_{\gamma}F_-(\overline{z},z)dz\right\}\\
    c_2(z,\zp) &= 4\Re\left\{\int\int_{\gamma}F_+(\overline{z_1},z_0)\overline{F_+(\overline{z_0},z_1)}dz_0\overline{dz_1} - \int\int_{\gamma}F_-(\overline{z_1},z_0)F_-(\overline{z_0},z_1)dz_0dz_1\right\}
    \end{array}
    \right.
\ee
Note that $c_n(z,\zp)=c_n(z,\zp,U)$ is a conformally invariant quantity, in the sense that if $f:U\to U'$ is a conformal map, then $c_n(z,\zp,U)=c_n(f(z),f(\zp),f(U))$. This follows from two facts: that each integrand is analytic (resp. antianalytic) in $z_i$ when $\sigma_i=1$ (resp. $\sigma_i=-1$), and that the functions $F_{\pm}=F_{\pm,U}$ satisfy $F_{\pm,U}(u,v) = (f'(u))^{(\pm)} F_{\pm,f(U)}(f(u),f(v))$, where the superscript $(\pm)$ again denotes whether we take the complex conjugate or not. See Proposition 15 of \cite{Ken00}, as well as the discussion after Proposition 20 of the same paper.

Note also that $c_n(z,\zp)$ is always real, since the terms in the sum corresponding to $\sigma$ and $-\sigma$ are complex conjugates. Moreover, since $F_+$ and $F_-$ are analytic on $U$ and remain bounded near $\zp$, the $c_n(z,\zp)$ grow at most exponentially in $n$: there exists an absolute constant $C$ (depending only on $\gamma$ and $U$) such that 
\be\label{eq:ck:bounded}
    \forall n \in \ZZ_{\geq 0},~c_n(z,\zp) \leq C^n.\\
\ee

To state our main result, we need to introduce the \emph{complete Bell polynomials}, defined by
\bestar
    B_n(X_1,\dots,X_n) = n! \sum_{j_1 + 2j_2 + \dots + nj_n = n} \prod_{k=1}^n \frac{X_k^{j_k}}{(k!)^{j_k}(j_k)!}.
\eestar
For example $B_0 = 1, B_1 = X_1, B_2 = X_1^2 + X_2$.\par
Recall that $\nez$ and $\oez$ are respectively the number of arcs enclosing $z$, and the parity of the number of arcs enclosing $z$.
\begin{theorem}\label{thm:asymp:ne:oe}
    Let $\omega \in \Om(\Ge)$ be a folded or shifted dimer configuration, with law $\mu_{\eps} = \mu_{\eps}^1$ or $\mu_{\eps} = \mu_{\eps}^2$, and expectation with respect to this law denoted by $\Ee$. For all $n \in \ZZ_{\geq 0}$, for all $\sigma \in \{0,1\}$,
    \bestar
        \Ee\left[\binom{\frac{\nez-\oez}{2}}{n}\oez^{\sigma}\right] \overset{\eps \to 0}{\longrightarrow} \frac{(-1)^{n+\sigma}}{(2n+\sigma)!}B_{2n+\sigma}\left(\left((-2)^{k-1}(k-1)!c_k(z,\zp)\right)_{1 \leq k \leq 2n+\sigma}\right).
    \eestar
\end{theorem}
The proof of this theorem given in Section \ref{sec:proof-main-thm}. The theorem gives the following corollary:
\begin{corollary}\label{cor:joint:moments}
    The asymptotics of all joint moments of $(\nez,\oez)$ can be computed explicitly: more precisely, for all $n \in \ZZ_{\geq 0}$ and $\sigma \in \{0,1\}$, $\Ee[\nez^n\oez^\sigma]$ converges towards a polynomial in the $\left(c_k(z,\zp)\right)_{1 \leq k \leq 2n+\sigma}$. For example,   
    \be\label{eq:cor:first:two}
        \left\{
        \begin{array}{ll}
            \Ee[\oez] & \overset{\eps \to 0}{\longrightarrow} -c_1(z,\zp)\\
            \Ee[\nez] & \overset{\eps \to 0}{\longrightarrow} 2c_2(z,\zp) - c_1(z,\zp)^2 - c_1(z,\zp)\\
            \mathrm{var}[\nez] & \overset{\eps \to 0}{\longrightarrow}
            -\tfrac{2}{3}c_1^4 - \tfrac{4}{3}c_1^3 - 3c_1^2 - c_1
            +\tfrac{32}{3}c_1c_3 - 4c_2^2 + 4c_2
            +\tfrac{16}{3}c_3 - 16c_4,
        \end{array}
        \right.
    \ee
    where in the last line we write $c_i=c_i(z,\zp)$ for brevity. The parity of the number of arcs enclosing a given point $\oez$ and the number of arcs surrounding a given point $\nez$ converge in law towards random variables $o(z), n(z)$ which are conformally invariant.
\end{corollary}
We explain how Theorem \ref{thm:asymp:ne:oe} implies Corollary \ref{cor:joint:moments}.
\begin{proof}
    We first observe that since $\oez \in \{0,1\}$, $\oez^k = \oez$ for all $k \geq 1$, it is enough to compute $\Ee[\nez^k]$ and $\Ee[\nez^k\oez]$ for all $k \in \ZZ_{\geq 0}$. Theorem \ref{thm:asymp:ne:oe} applied with $n = 0, \sigma = 1$ gives
    \bestar
        \Ee[\oez] \overset{\eps \to 0}{\longrightarrow} -B_1(c_1(z,\zp)) = -c_1(z,\zp).
    \eestar
    Theorem \ref{thm:asymp:ne:oe} applied with $n=1,\sigma = 0$ gives
    \bestar
        \Ee[\nez-\oez] = -B_2(c_1(z,\zp), -2c_2(z,\zp)) +o_{\eps \to 0}(1) = -c_1(z,\zp)^2 +2c_2(z,\zp) + o_{\eps \to 0}(1).  
    \eestar
    By induction, it is possible to obtain all joints moments of $(\nez,\oez)$ in the same way. Assume that for some $n$, for all $0 \leq k \leq n$ and $\sigma \in \{0,1\}$, $\Ee[\nez^k\oez^\sigma]$ is a polynomial in the $c_i(z,\zp)$ for $1 \leq i \leq 2k+\sigma$. Developing the binomial coefficient, we can write
    \bestar
        \begin{aligned}
            \Ee\left[\binom{\frac{\nez-\oez}{2}}{n+1}\right]
            &= \frac{1}{2^{n+1}(n+1)!}\left(\Ee[\nez^{n+1}] + \sum_{\underset{\sigma \in \{0,1\}}{0 \leq k \leq n}} a_{k\sigma}\Ee[\nez^k\oez^{\sigma}]\right)
        \end{aligned}
    \eestar
    where the $a_{k\sigma}$ are explicit coefficients. By induction, the $\Ee[\nez^k\oez^{\sigma}]$ are explicit polynomials in the $(c_{k+\sigma}(z,\zp))_{1\leq 2k+\sigma \leq 2n+1}$ (up to $o_{\eps \to 0}(1)$, and by Theorem \ref{thm:asymp:ne:oe} applied at $n+1, \sigma = 0$, the left-hand side is an explicit polynomial in the $(c_{k+\sigma}(z,\zp))_{1\leq 2n+2}$ (up to $o_{\eps \to 0}(1)$), so $\Ee[\nez^{n+1}]$ is also an explicit polynomial in the $(c_{k+\sigma}(z,\zp))_{1\leq 2n+2}$ up to $o_{\eps \to 0}(1)$.\par
    Similarly, Theorem \ref{thm:asymp:ne:oe} with $n+1, \sigma = 1$ enables to express $\Ee[\nez^{n+1}\oez]$ in terms of the $\left(c_k(z,\zp)\right)_{1 \leq k \leq 2n+3}$, which concludes the induction step.
    
    \bigskip

    We obtained that the moments of $\nez$ converge: for all $n$, $\Ee[\nez^n] \overset{\eps \to 0}{\longrightarrow}m_n$. To deduce the convergence in law of the random variable $\nez$, we only have to check the conditions of the moment problem, that is find a bound on $m_n$. Bounding crudely the binomial coefficient from below, we find that
    \bestar
        \begin{aligned}
        \Ee\left[\binom{\frac{\nez-\oez}{2}}n\right] \geq \Ee\left[\binom{\frac{\nez-1}{2}}{n}\right] &\geq \Ee\left[\mathds{1}_{\nez \geq 2n}\frac{(\nez-2n)^n}{2^n n!}\right]\\
        &\geq \Ee\left[\mathds{1}_{\nez \geq 4n}\frac{\nez^n}{4^n n!}\right] \geq \frac{1}{4^n n!}(\Ee[\nez^n]-(4n)^n).
        \end{aligned}
    \eestar
    and taking the limit in Theorem \ref{thm:asymp:ne:oe} with $n, \sigma = 0$ gives
    \bestar
        m_n \leq (4n)^n + 4^n n! \frac{\left|B_{2n}\left(\left((-2)^{k-1}(k-1)!c_k(z,\zp)\right)_{1 \leq k \leq 2n}\right)\right|}{(2n)!}.
    \eestar
    Since the $c_k(z,\zp)$ grow at most exponentially by Equation \eqref{eq:ck:bounded}, using the explicit expression of the Bell polynomial, we obtain
    \bestar
        \begin{aligned}
            \frac{\left|B_{2n}\left(\left((-2)^{k-1}(k-1)!c_k(z,\zp)\right)_{1 \leq k \leq 2n}\right)\right|}{(2n)!} 
            &\leq \sum_{j_1 + 2j_2 + \dots + 2nj_{2n} = 2n} \prod_{i=1}^{2n} \frac{((2C)^k k!)^{j_k}}{(k!)^{j_k}(j_k)!}\\
            &= (2C)^{2n} \sum_{j_1 + 2j_2 + \dots + 2nj_{2n} = 2n} \frac{1}{(j_k)!}\\
            &\leq (2C)^{2n} \left|\left\{(j_1,\dots, j_{2n})~;~j_1 + 2j_2 + \dots + 2nj_{2n} = 2n\right\}\right|\\
            &\leq (2C)^{2n} \frac{(2n)^{2n}}{(2n)!},
        \end{aligned}
    \eestar
    since there are at most $n$ choices for $j_1$, $n/2$ choices for $j_2$ etc. Using Stierling's formula, we conclude that $m_n \leq C^n n^n$ for another absolute constant $C$, hence Carleman's condition is verified and $\nez$ converges towards a random variable $n(z)$ which is uniquely determined by its moments $m_n$. Since these moments are conformally invariant (they are expressed as integrals of conformal quantities), the law of $n(z)$ is also conformally invariant. 
\end{proof}

In the particular case $U^r=[-\infty, +\infty]\times[-\pi/2, \pi/2]$ where the Green function has a simple explicit expression, the computations turn out to be particularly simple.
\begin{remark}
    The infinite strip does not fill the technical hypothesis of our theorem since it is not a bounded open set, although it fills the other conditions: it is symmetric and has a horizontal boundary. Nonetheless, since the formula for the moments depends only on conformally invariant quantities, the moments can be computed in any simply connected open set, and the computations transferred to any other domain by a Riemann map. Besides, if we take $\Uer = \eps \Z^2 \cap ([-1/\eps, 1/\eps] \times [-\pi/2,\pi/2])$, the proof of Theorem \ref{thm:asymp:ne:oe} still holds, at least for the folded dimer model:  for the proof to extend to this folded case on $\Uer$, the only part of the proof that has to be modified is Lemma \ref{lem:Green:boundary} of the Appendix. The proof of this lemma still works because, to show that $K^{-1}(b,w)$ is close to the partial derivative of the Green function $g$, we only need that $g_{\RR^2}$ (the full-plane Green function) or $g_{\HH}$ (the half-plane Green function) and $g$ are $O(\eps)$ on the boundary of $\Ged$ to apply the Harnack lemma. This is also true here.
\end{remark}
\begin{example}[The infinite strip]\label{ex:strip}
    Let $U^r=[-\infty, +\infty]\times[-\pi/2, \pi/2]$. We give a particularly simple expression for $\Ee[\oez]$ and $\Ee[\nez]$. To apply Corollary \ref{cor:joint:moments}, we need to identify $F_-$ and $F_+$. The conformal map from this strip to the half plane is 
    \bestar
        \phi: U^r \to \HH, u \to \exp(u+i\pi/2) = i\exp(u).
    \eestar
    On the half plane, $\tilde{g}_{\HH}(u,v) = -\frac{1}{2\pi} \log \frac{u-v}{\bar{u}-v}$ with $\log$ denoting the principal value of the complex logarithm (see the footnote in Section 6.3 of \cite{Ken14}). Hence, on the strip
    \bestar
        \tilde{g}(u,v) = \tilde{g}_{\HH}(\phi(u),\phi(v)) = -\frac{1}{2\pi} \log \left(-\frac{e^u-e^v}{e^{\bar{u}}+e^v}\right).
    \eestar
    By definition, for $u,v \in U^r$,
    \bestar
        F_+(u,v) = -\frac{1}{2\pi} \frac{e^u}{e^u-e^v},\qquad F_-(u,v) = \frac{1}{2\pi} \frac{e^{\bar{u}}}{e^{\bar{u}}+e^v}.
    \eestar
    Let $z = x+iy$ (with $x \in (0, \pi/2)$), $\zp = x+i\pi/2$. For $\gamma$, we choose the vertical straight line from $z$ to $\zp$. We first compute the probability that there is an odd number of arcs enclosing $z$: for all $z \in U^r$, $F_-(\bar{z},z) = \frac{1}{4\pi}$ so 
    \bestar
        c_1(z,\zp) = -4\Im\left\{ \int_{x+iy}^{x+i\pi/2} \frac{1}{4\pi}dz\right\} = -\frac{1}{2}+\frac{y}{\pi}
    \eestar
    and by Corollary \ref{cor:joint:moments},
    \bestar
        \Ee[\oez]  \overset{\eps \to 0}{\longrightarrow} -c_1(z,\zp) = \frac{1}{2}-\frac{y}{\pi}.
    \eestar
    To obtain the expected number of arcs in the limit, we first need to compute
    \bestar
        c_2(z,\zp) = 4\Re\left\{\int\int_{\gamma}F_+(\overline{z_1},z_0)\overline{F_+(\overline{z_0},z_1)}dz_0\overline{dz_1} - \int\int_{\gamma}F_-(\overline{z_1},z_0)F_-(\overline{z_0},z_1)dz_0dz_1\right\}.
    \eestar
    For all $z_0 = x+iy_0$, $z_1 = x + iy_1$,
    \bestar
        \left\{
        \begin{array}{llll}
             F_-(\bar{z_1}, z_0) &= \frac{1}{2\pi}\frac{e^{z_1}}{e^{z_1}+e^{z_0}} &= \frac{1}{2\pi}\frac{e^{iy_1}}{e^{iy_1}+e^{iy_0}} &= \frac{1}{4\pi}\left(1+i\tan\left(\frac{y_1-y_0}{2}\right)\right)  \\
             F_+(\bar{z_1}, z_0) &= -\frac{1}{2\pi}\frac{e^{\bar{z_1}}}{e^{\bar{z_1}}-e^{z_0}} &= -\frac{1}{2\pi}\frac{e^{-iy_1}}{e^{-y_1}-e^{iy_0}} &= -\frac{1}{4\pi}\left(1+i\cot\left(\frac{y_1+y_0}{2}\right)\right).
        \end{array}
        \right.
    \eestar
    After some algebra, we get
    \bestar
        c_2(z,\zp) = -\frac{1}{\pi^2}\ln(\sin(y)),
    \eestar
    so Corollary \ref{cor:joint:moments} implies
    \be \label{eq:mean}
        \Ee[\nez] \overset{\eps \to 0}{\longrightarrow} -\frac{2}{\pi^2}\ln(\sin(y)) + \left(\frac{1}{4}-\frac{y^2}{\pi^2}\right).
    \ee
    \end{example}
    \begin{remark}
        This is coherent with what we expect of the continuous model: when $z$ approaches the real axis, the number of arcs enclosing $z$ explodes and it is odd or even with probability $1/2$. When $z$ approaches the boundary of the strip, the probability that there is at least one arc enclosing $z$ goes to $0$, and so does the probability that there is an odd number of arcs. Besides, the number of arcs above a point is always positive and explodes when this point approaches the real axis.
    \end{remark}

\section{Proof of the main theorem: the folded case}\label{sec:proof-main-thm}
In this section, we prove Theorem \ref{thm:asymp:ne:oe} when $\mu_{\eps} = \mu^1_{\eps}$, that is in the case of folded dimers. In the spirit of \cite{Ken14}, the proof consists of choosing a connection on the graph $\Ge$ and analysing its properties. 
\begin{proof}
    Let $\gamma$ be a path in $U$ from $z$ to $\zp$. The associated \emph{zipper} $E_{\eps}(\gamma)$ is the set of oriented edges of $\Ge$ crossing the path $\gamma$ from left to right. An example of a zipper as a specific polygonal path is drawn in Figure \ref{fig:zipper}.\par
    \begin{figure}
        \begin{overpic}[abs,unit=1mm,scale=0.5]{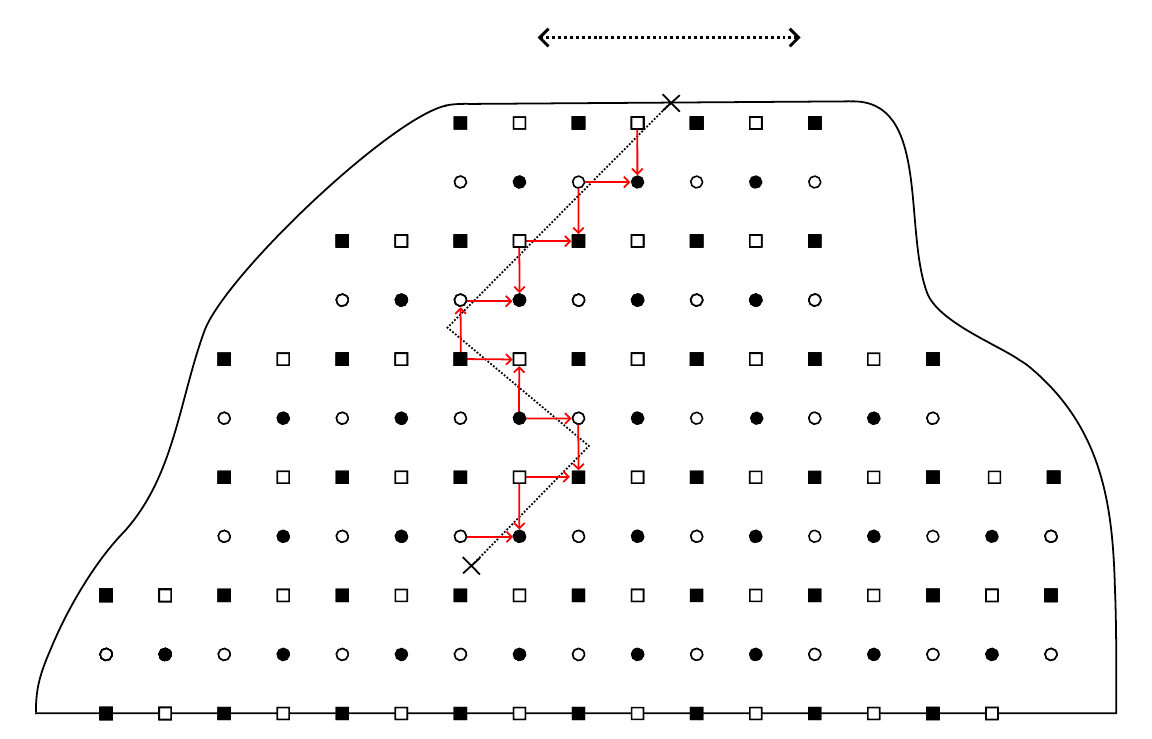}
            \put(25,52){$U$}
            \put(54,60){$\d$}
            \put(40.5,12.5){$z$}
            \put(55,56){$\zp$}
        \end{overpic}
        \caption{A path $\gamma$ from $z$ tp $\zp$ drawn as a dotted line. The associated zipper $E_{\eps}(\gamma)$ is drawn in red (arrows indicate directed edges).}
        \label{fig:zipper}
    \end{figure}
    Let $\eps >0$ be fixed. As in the preceding section, for every $\a >0$ we can define a connection on $\Ge$ associated to the zipper by setting, for any directed edge $e \in E_{\eps}(\gamma)$, $\phi_e = \begin{pmatrix} 1&\a \\ 0&1\end{pmatrix}$ and $\phi_{e^{-1}} = \phi_e^{-1}$, and $\phi_e = I_2$ on all other bulk edges. We also define $\psi_e = \begin{pmatrix} 1 \\ 1\end{pmatrix}$ for all edges $e$ linking a boundary point with a bulk point. Denote by $K_{\a}$ the Kasteleyn matrix associated with this connection and the usual Kasteleyn phases $\zeta$ of discrete holomorphy given by Equation \eqref{kast-eq-conplex-kast-weights}. We drop the dependency in $\eps$ of the Kasteleyn matrix to lighten the notation. The Kasteleyn matrix associated with the trivial connection is $K = K_0$. By Corollary \ref{kast:corr},
        \be\label{eq:kast:Ka}
            \pf K_{\a} = \sum_{\om\in\Om} 
            \prod_{\mathrm{loops} \ C}\tr(\phi_C)
            \prod_{\mathrm{arcs} \ A}\psi_{b_A}^\T\phi_A\psi_{w_A}.
        \ee
    From here, the proof follows three main steps:
    \begin{itemize}
        \item Show that the right-hand side of Equation \eqref{eq:kast:Ka} can be expressed as a power series in $\a$ involving the moments of $\nez$ and $\oez$.
        \item Show that the left hand side can be expanded as a power series in $\a$ with coefficients $\tr((K^{-1}S)^n)$ where $K^{-1}$ is a discrete holomorphy matrix (see Appendix \ref{appendix:A} for more details on the discrete holomorphy theory on the square lattice), and $S$ is a matrix encoding the zipper.
        \item Use the asymptotic expression of $K^{-1}$ in terms of the continuous Green function obtained by Kenyon in \cite{Ken00} (see Appendix \ref{appendix:A}) to show that $\tr((K^{-1}S)^n)$ converges to an $n$-fold holomorphic integral.   
    \end{itemize}
    \paragraph{\textbf{First step: the right-hand side of Equation \eqref{eq:kast:Ka}}}
    Let $\a >0$. Given $\om \in \Omega(\Ge)$, we denote by $c_{\eps}$ its total number of loops and $n_{\eps}$ its total number of arcs (which is constant and equal to the number of white boundary vertices). We can write:
        \bestar
            \pf K = \sum_{\om\in\Om} 2^{c_{\eps}}2^{n_{\eps}},
        \eestar
    which is the partition function of the loops and arcs model. Now let $\a \geq 0$. The possible monodromies for loops are $I_2$, $\begin{pmatrix} 1&\a \\ 0&1\end{pmatrix}$ and its inverse $\begin{pmatrix} 1&\ {-\a} \\ 0&1\end{pmatrix}$, which all have trace $2$. Hence each loop contributes a factor $2$ to the right-hand side of the Kasteleyn formula. For an arc $A$ with endpoints $w_A$ and $b_A$, we say that the arc is clockwise if $w_A$ is left of $b_A$. By definition of the zipper $E(\gamma)$, there are three possibilities for the monodromy along an arc. If $A$ does not enclose $z$, then $\phi_A = I_2$ and $\psi_{b_A}^\T\phi_A\psi_{w_A}= 2$. If $A$ is clockwise and encloses $z$, since $\phi_A$ is computed by orienting the edges towards $w_A$, $\phi_A = \begin{pmatrix} 1&\a \\ 0&1\end{pmatrix}$ and $\psi_{b_A}^\T\phi_A\psi_{w_A}= 2 + \a$. If $A$ is counterclockwise and encloses $z$, we get $\psi_{b_A}^\T\phi_A\psi_{w_A}= 2 - \a$.  Given $\om \in \Omega(\Ge)$, we denote by $r_{\eps}(z)$ (resp $l_{\eps}(z)$) its number of clockwise (resp counterclockwise) arcs enclosing $z$. These quantities are random variables whose laws depends on $\eps$. For fixed $\eps$, 
        \be\label{eq:pfaffian:developed}
            \begin{aligned}
                \pf K_{2\a} 
                &= \sum_{\om\in\Om} 2^{c_{\eps}}2^{n_{\eps} - r_{\eps}(z)-l_{\eps}(z)}(2+2\a)^{r_{\eps}(z)}(2-2\a)^{l_{\eps}(z)} \\
                &= \sum_{\om\in\Om} 2^{c_{\eps}+n_{\eps}}(1+\a)^{r_{\eps}(z)}(1-\a)^{l_{\eps}(z)}.\\
            \end{aligned}
        \ee
    Recall the definition of $\nez$ and $\oez$. We claim that 
    \bestar
        \forall \eps \geq 0,~
        \left\{
        \begin{array}{ll}
             \nez &=l_{\eps}(z) + r_{\eps}(z)\\
             \oez &=l_{\eps}(z) - r_{\eps}(z)
        \end{array}
        \right.
        .
    \eestar
    The first equality is clear: an arc is either oriented clockwise or counterclockwise. We explain why the second one holds. For $\om \in \Omega(\Ge)$, the arcs enclosing $z$ are alternating in orientation. To see this, assume that two arcs $A_1$ and $A_2$ both enclosing $z$ have the same orientation, say clockwise $w_{A_1} < w_{A_2} < b_{A_2} < b_{A_1}$, and that there is no other arc enclosing $z$ between them. This cannot happen because there is an odd number of boundary vertices between $w_{A_1}$ and $w_{A_2}$ so they cannot be matched together. For the same reason, since the leftmost corner on the boundary is black while the rightmost is white by definition of the Temperleyan domain, there is always $0$ or $1$ more counterclockwise than clockwise arc which encloses $z$. Hence, $\oez := l_{\eps}(z) - r_{\eps}(z) = \mathds{1}\{\nez \text{ is odd}\} \in \{0,1\}$, and 
    \bestar
        \pf K_{2\a} = \sum_{\om\in\Om} 2^{c_{\eps}+n_{\eps}} (1-\a^2)^{r_{\eps}(z)}(1-\a)^{\oez} = \sum_{\om\in\Om} 2^{c_{\eps}+n_{\eps}} (1-\a^2)^{r_{\eps}(z)}(1-\oez\a). 
    \eestar
    Dividing by $\pf K$, we get
    \be\label{eq:development:powers:a}
        \begin{aligned}
            \frac{\pf K_{2\a}}{\pf K} &= \Ee[(1-\a^2)^{r_{\eps}(z)}(1-\oez\a)] = \Ee\left[\sum_{k =0}^{\infty}(-1)^k\binom{r_{\eps}(z)}{k}\a^{2k}(1-\oez\a)\right]\\
            &= \sum_{k =0}^{\infty}(-1)^k\a^{2k}\Ee\left[\binom{r_{\eps}(z)}{k}\right] - \sum_{k =0}^{\infty}(-1)^k\a^{2k+1}\Ee\left[\binom{r_{\eps}(z)}{k}\oez\right]\\
            &= \sum_{k =0}^{\infty}(-1)^k\a^{2k}\Ee\left[\binom{\frac{\nez-\oez}{2}}{k}\right] - \sum_{k =0}^{\infty}(-1)^k\a^{2k+1}\Ee\left[\binom{\frac{\nez-\oez}{2}}{k}\oez\right].
        \end{aligned}
    \ee
    For every fixed $\eps >0$ and every $\a >0$ small enough, the sums converge because the number of (counterclockwise) arcs surrounding any given point is bounded. 

    \bigskip
    
    \paragraph{\textbf{Second step: the left-hand side of Equation \eqref{eq:kast:Ka}.}}
    To lighten the notation, we write $E_{\eps} := E_{\eps}(\gamma)$. By definition, the zipper is a set of directed edges: we also say that an undirected edge $wb \in E(\Ge)$ belongs to the zipper and we write $wb \in E_{\eps}$ if $(w,b) \in E_{\eps}$ or $(b,w) \in E_{\eps}$. We also say that a (white or black) vertex $x \in \Ge$ belongs to the zipper and we write $x \in E_{\eps}$ if there exists $y \in \Ge$ such that $xy \in E_{\eps}$. If $wb$ is an undirected edge in $\Ge$, we define its sign by $\sgn(E_{\eps})_{w,b} = \mathds{1}\{(w,b) \in E_{\eps}\} - \mathds{1}\{(b,w) \in E_{\eps}\}$. We write the vertices of $\Ged$ as $(x,i)$ with $x \in \Ge$, $i \in \{1,2\}$ where $i$ indicates whether $x$ is in the upper or lower half plane. We use the convention that $i=1$ for the vertices on the horizontal axis. We note that $K_{2\a}$ is skew-symmetric, $(K_{2\a})_{(w,i),(w',j)} = (K_{2\a})_{(b,i),(b',j)} = 0$ for any $(w,i),(w',j) \in W(\Ged)$ and $(b,i),(b',j) \in \Ged$, and that for two vertices $(w,i) \in W(\Ged),(b,j) \in B(\Ged)$, 
        \bestar
            \begin{aligned}
                (K_{2\a})_{(w,i),(b,j)} 
                &= \zeta_{w,b}\mathds{1}\{w \sim b\}(\mathds{1}\{i=j\} + 2\a \mathds{1}\{i=2, j=1\}\sgn(E_{\eps})_{w,b}) \\
                &= K_{(w,i),(b,j)} + 2\a S_{(w,i),(b,j)},
            \end{aligned}
        \eestar
    where 
    \bestar
    S_{(w,i),(b,j)} = \mathds{1}\{w \sim b,i=2, j=1\}\zeta_{w,b}\sgn(E_{\eps})_{w,b}
    \eestar
    is defined by the previous equation (and $S$ is also skew-symmetric, zero on $W(\Ged) \times W(\Ged)$ and on $B(\Ged) \times B(\Ged)$). The key observation is to interpret $K_{2\a} = K + 2\a S$ as a perturbation of the usual discrete holomorphy matrix $K$ if we “unfold" the graph $\Ged$ i.e. if we identify $(x,2) \in \Ged$ with $\bar{x} \in \eps \ZZ^2$. This is almost true, but when the graph is unfolded the vertical edges on the lower half plane have the opposite orientation, see Figure \ref{fig:Kasteleyn}. 
    \begin{figure}
    \centering        
        \begin{overpic}[abs,unit=1mm,scale=0.5]{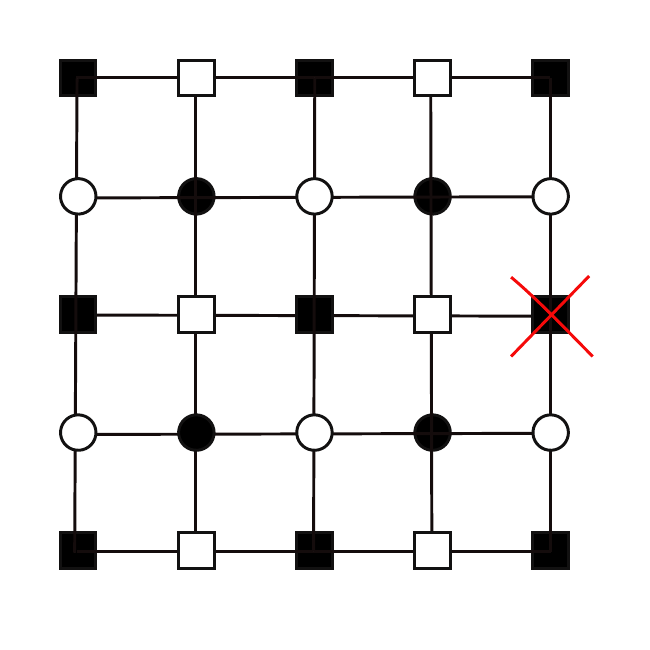}
            \put(9,5){-1}\put(21,5){1}\put(30,5){-1}\put(41,5){1}
            \put(11,15){1}\put(20,15){-1}\put(31,15){1}\put(40,15){-1}
            \put(10,25.5){-1}\put(21,25.5){1}\put(30,25.5){-1}\put(41,25.5){1}
            \put(11,35.5){1}\put(20,35.5){-1}\put(31,35.5){1}\put(40,35.5){-1}
            \put(10,45.5){-1}\put(21,45.5){1}\put(30,45.5){-1}\put(41,45.5){1}
            \put(5, 12){i}\put(13.5, 12){-i}\put(25, 12){i}\put(33.5, 12){-i}\put(45, 12){i}
            \put(3.5, 22){-i}\put(15, 22){i}\put(23.5, 22){-i}\put(35, 22){i}\put(43.5, 22){-i}
            \put(3.5, 32){-i}\put(15, 32){i}\put(23.5, 32){-i}\put(35, 32){i}\put(43.5, 32){-i}
            \put(5, 42){i}\put(13.5, 42){-i}\put(25, 42){i}\put(33.5, 42){-i}\put(45, 42){i}
            \put(-3,27.5){$\substack{(b,1)\\=(b,2)}$}
            \put(12,51){$(w,1)$}
            \put(12,3){$(w,2)$}
        \end{overpic}
        \begin{overpic}[abs,unit=1mm,scale=0.5]{Kasteleyn_orientation.pdf}
            \put(9,5){-1}\put(21,5){1}\put(30,5){-1}\put(41,5){1}
            \put(11,15){1}\put(20,15){-1}\put(31,15){1}\put(40,15){-1}
            \put(10,25.5){-1}\put(21,25.5){1}\put(30,25.5){-1}\put(41,25.5){1}
            \put(11,35.5){1}\put(20,35.5){-1}\put(31,35.5){1}\put(40,35.5){-1}
            \put(10,45.5){-1}\put(21,45.5){1}\put(30,45.5){-1}\put(41,45.5){1}
            \put(5, 22){i}\put(13.5, 22){-i}\put(25, 22){i}\put(33.5, 22){-i}\put(45, 22){i}
            \put(3.5, 12){-i}\put(15, 12){i}\put(23.5, 12){-i}\put(35, 12){i}\put(43.5, 12){-i}
            \put(3.5, 32){-i}\put(15, 32){i}\put(23.5, 32){-i}\put(35, 32){i}\put(43.5, 32){-i}
            \put(5, 42){i}\put(13.5, 42){-i}\put(25, 42){i}\put(33.5, 42){-i}\put(45, 42){i}
            \put(-3,27.5){$\substack{(b,1)\\=(b,2)}$}
            \put(12,51){$(w,1)$}
            \put(12,3){$(w,2)$}
        \end{overpic}
        \caption{The phases $\zeta$ associated to $K$ (on the left) and $\tilde{\zeta}$ associated to $\tK$ (on the right). The orientations differ only for vertical edges in the lower half part of the pictures. We marked a point $b$ on the reflection axis and a point $(w,1)$ and its reflection $(w,2)$ identified with the complex conjugate. The vertex crossed in red is $b_0$ removed by the Temperleyan boundary conditions.}
        \label{fig:Kasteleyn}
    \end{figure}
    This is fixed by performing a gauge equivalence (on the weights this time, and not on the Kasteleyn phases as in Remark \ref{kast:rmk:prop}): we multiply all columns and rows by 
    \be\label{eq:gauge:folded}
        \eps_{(x,i)} = 
        \left\{
        \begin{array}{ll}
             -1& \text{ if } i=2 \text{ and } x \in W_0 \cup B_0\\
             1& \text{ otherwise.}
        \end{array}
        \right.
    \ee
    Hence for all $(w,i), (b,j)$,
    \bestar
        \tilde{\zeta}_{(w,i),(b,j)} = \eps_{(w,i)}\eps_{(b,j)}\zeta_{w,b} \quad , \quad (\tilde{K}_{2\a})_{(w,i),(b,j)} = \eps_{(w,i)}\eps_{(b,j)}(K_{2\a})_{(w,i),(b,j)}.
    \eestar
    The matrix $\tilde{K}$ is the usual discrete holomorphy operator on the Temperleyan graph $\Ge^r$, see Figure \ref{fig:Kasteleyn}. In particular, it satisfies the asymptotic estimates of Appendix \ref{appendix:A}. This depends crucially on our definition of $\psi$ on the boundary. We have
    \be\label{eq:Ka=K+aS}
        (\tilde{K}_{2\a})_{(w,i),(b,j)} = \tilde{K}_{(w,i),(b,j)} + 2\a \tilde{S}_{(w,i),(b,j)}
    \ee
    with 
    \be\label{eq:def:tS}
    \tilde{S}_{(w,i),(b,j)} = \mathds{1}\{i=2, j=1, w \sim b\}\zeta_{w,b}\sgn(E_{\eps})_{w,b}(-1)^{w \in W_0}.
    \ee
    By multi-linearity of the determinant,
    \bestar
        \det(\tK_{\a}) = \bigg(\prod_{(x,i) \in \Ged} \eps(x,i)\bigg)^2 \det(\tK) = \det(\tK).
    \eestar
    Since the discrete holomorphy matrix $\tK$ is invertible (see Appendix \ref{appendix:A}), Equation \eqref{eq:Ka=K+aS} implies
    \be\label{eq:Pf^2}
        \left(\frac{\pf K_{2\a}}{\pf K}\right)^2 = \left(\frac{\pf \tK_{2\a}}{\pf \tK}\right)^2 
        = \det(I+2\a \tS\tK^{-1}).
    \ee    
    By using the identity (which holds for any finite matrix $M$ and $\a < \rho(M)^{-1}$ the inverse spectral radius)
    \bestar
        \log \big(\det (I+\a M)\big) = \sum_{k=1}^{\infty} \frac{(-1)^{k-1}}{k}\a^k\tr(M^k)
    \eestar
    and taking the logarithm and square root in Equation \eqref{eq:Pf^2} (using that the Pfaffians are positive for $\a < 1$ due to Equation \eqref{eq:pfaffian:developed}), we get that
    \be\label{eq:power-expansion:RHS}
         \forall \alpha \in (0, \rho(\tS\tK^{-1})^{-1}),~\frac{\pf K_{2\a}}{\pf K}
        = \exp\left(\frac{1}{2}\sum_{k=1}^{\infty} \frac{(-1)^{k-1}}{k}(2\a)^k\tr\big((\tS\tK^{-1})^k\big)\right).
    \ee
    Combining this with Equation \eqref{eq:development:powers:a}, we obtain for all $\alpha \in (0, \rho(\tS\tK^{-1})^{-1})$,
    \be\label{eq:second:step}
        \begin{aligned}
        &\sum_{k =0}^{\infty}(-1)^k\a^{2k}\Ee\left[\binom{\frac{\nez-\oez}{2}}{k}\right] - \sum_{k =0}^{\infty}(-1)^k\a^{2k+1}\Ee\left[\binom{\frac{\nez-\oez}{2}}{k}\oez\right]\\
        &\quad = \exp \left(\frac{1}{2}\sum_{k=1}^{\infty} \frac{(-1)^{k-1}}{k}(2\a)^k\tr\big((\tS\tK^{-1})^k\big)\right).
        \end{aligned}
    \ee
    This concludes the second step of the proof: we can see from this equation that the moments of $\nez$ and $\oez$ can be computed in terms of the $\tr\big((\tS\tK^{-1})^k\big)$ by algebraic manipulations (this will be detailed at the end of the proof).

    \bigskip
    
    \paragraph{\textbf{Third step: the asymptotic of the trace terms.}}
    The next step of the proof is to compute $\tr\big((\tS\tK^{-1})^n\big)$ for all $n$. Our asymptotic will imply in particular that the spectral radius $\rho(\tS\tK^{-1})$ is uniformly bounded. Let $n \geq 1$ be fixed. Using the convention that $(w_n,i_n) = (w_0,i_0)$, we can write
    \be\label{eq:product:to:develop}
        \tr\big((\tS\tK^{-1})^n\big) = 2\sum_{(w_0,i_0), \dots, (w_{n-1},i_{n-1})}\prod_{k=0}^{n-1} (\tS\tK^{-1})_{(w_k,i_k),(w_{k+1},i_{k+1})}
    \ee
    since $\tS\tK^{-1}$ is symmetric (explaining the factor $2$, coming from the black vertices) and zero from black to white and white to black. The factors in the product are computed as follows: for all $u, v \in W(\Ged)$, for all $i, j \in \{1,2\}$, using Equation \eqref{eq:def:tS},
    \bestar
        \begin{aligned}
            (\tS\tK^{-1})_{(u,i),(v,j)} &= \sum_{(b,k) \in B(\Ged)} \tS_{(u,i),(b,k)}\tK^{-1}_{(b,k),(v,j)}\\
            &= \mathds{1}\{i=2\}\sum_{b: \{ub\} \in E_{\eps}} \zeta_{u,b}\sgn(E_{\eps})_{u,b}(-1)^{\mathds{1}\{u \in W_0\}}\tK^{-1}_{(b,1),(v,j)}.
        \end{aligned}
    \eestar
    Developing the product in Equation \eqref{eq:product:to:develop} yields
    \be\label{eq:trace:n}
        \begin{aligned}
            \tr\big((\tS\tK^{-1})^n\big) &= 2 \sum_{w_0, \dots, w_{n-1}} \sum_{\underset{\{w_kb_k\} \in E_{\eps}}{b_0, \dots, b_{n-1}:}} \prod_{k = 0}^{n-1}\zeta_{w_k,b_k}\sgn(E_{\eps})_{w_k,b_k}(-1)^{\mathds{1}\{w_k \in W_0\}} \tK^{-1}_{(b_k,1),(w_{k+1},2)}\\
            &=2 \sum_{\underset{e_k = \{w_kb_k\}}{e_0,\dots,e_{n-1} \in E_{\eps}}} \prod_{k=0}^{n-1}\zeta_{w_k,b_k}\sgn(E_{\eps})_{w_k,b_k}(-1)^{\mathds{1}\{w_k \in W_0\}} \tK^{-1}_{(b_k,1),(w_{k+1},2)}.
        \end{aligned}
    \ee
    From here, the rest of the proof consists of replacing $\tK^{-1}$ by its asymptotic expression from Lemma \ref{lem:Green:boundary} and identifying an $n$-fold Riemann sum that converges towards an $n$-fold holomorphic integral. Since the asymptotic expression of $\tK^{-1}_{(b_k,1),(w_{k+1},2)}$ varies according to the type of $w_{k+1}$ ($W_0$ or $W_1$) and the type of $b_k$ ($B_0$ or $B_1$), we must split the computations into many pieces.\par
    We choose a specific path $\gamma$ in $U$ (except for its endpoint $\zp \in \partial U$) for which the computations are tractable. More precisely, we choose $\gamma$ from $z$ to $\zp$ polygonal with slope $\pm 1$ consisting of a finite number of line segments, and we denote by $E_{\eps} = E_{\eps}(\gamma)$ the associated zipper. For each line segment of $\gamma$ in one of the directions NE, NW, SW, SE, the corresponding portion of the zipper $E_{\eps}(\gamma)$ consists of a zig-zag path of edges, alternately horizontal and vertical. Starting from here, $C = C(U,\gamma) > 0$ is a positive constant depending only on $\gamma$ and the open set $U$ that is allowed to change from line to line. The number of edges in the zipper is $|E_{\eps}(\gamma)| = |\gamma|\eps^{-1} + O(1) \leq C\eps^{-1}$.\par
    The edges of the zipper can be grouped in $N_{\eps} = |\gamma|\eps^{-1}+O(1) \leq C\eps^{-1}$ packets of four edges (except the edges near the changes of direction of $\gamma$, of which there are $\leq C$). We have
    \bestar
        E_{\eps} = \bigg(\bigsqcup_{1 \leq k \leq N_{\eps}}E_{\eps}(k)\bigg) \sqcup F_{\eps},
    \eestar
    where for $1 \leq k \leq N_{\eps}$, the $k^{th}$ packet $E_{\eps}(k)$ is a zig-zag path between vertices $(w_k^0,b_k^0,w_k^1,b_k^1,w^{k+1}_0)$  with $w_k^0 \in W_0, w_k^1 \in W_1, b_k^0 \in B_0, b_k^1 \in B_1$ as in Figure \ref{fig:zipper:local}, and where $|F_{\eps}| \leq C$. 
    \begin{figure}
        \begin{overpic}[abs,unit=1mm,scale=1]{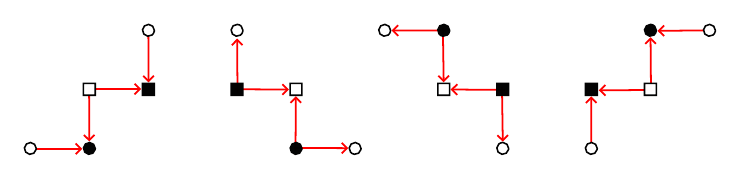}
            \put(7,30){$dz = \eps(1+i)$}
            \put(37,30){$dz = \eps(-1+i)$}
            \put(67,30){$dz = \eps(1-i)$}
            \put(100,30){$dz = \eps(-1-i)$}
            \put(8,5){$1$}
            \put(15,7){$-i$}
            \put(18,15){$1$}
            \put(24.5,18){$-i$}
            \put(35,18){$-i$}
            \put(42,15){$-1$}
            \put(44.5,7){$-i$}
            \put(52,5){$-1$}
            \put(69.5,25){$1$}
            \put(73,18){$i$}
            \put(79.5,15){$1$}
            \put(83,7){$i$}
            \put(101,7){$i$}
            \put(102,15){$-1$}
            \put(112,17){$i$}
            \put(112,25){$-1$}
            \put(2,-1){$w_k^0$}
            \put(15.5,23){$w_{k+1}^0$}
            \put(58,-1){$w_k^0$}
            \put(41.5,23){$w_{k+1}^0$}
            \put(81,-1){$w_{k+1}^0$}
            \put(58.5,23){$w_k^0$}
            \put(95,-1){$w_{k+1}^0$}
            \put(121.5,23){$w_k^0$}
        \end{overpic}
        \caption{The four possible directions for a zig-zag portion of the zipper $E_{\eps}(k)$. The Kasteleyn phases of the (undirected) edges are indicated in black.}
        \label{fig:zipper:local}
    \end{figure}
    More precisely,
    \bestar
        E_{\eps}(k) = \{w_k^0b_k^0, b_k^0w_k^1, w^k_1b_k^1, b_k^1w_{k+1}^0\}.
    \eestar
    If we write for all $0 \leq k \leq N_{\eps}$, $z_k = w_k^0$ and $dz_k =dx_k+idy_k \in \{\eps(\pm 1 \pm i)\}$ the displacement of the corresponding portion of the zipper (NE,NW,SE or SW), $b_k^0 = w_k^0 + (dx_k/2,0)$, $w_k^1 = b_k^0 +(0,dy_k/2)$, $b_k^1 = w_k^1 + (dx_k/2)$,$w_{k+1}^0 = b_k^1 + (0,dy_k/2)$. Observe that $z_{k+1} = w_{k+1}^0 = z_k^0 + dz_k$ except for at most $C(\gamma)$ values of $k$ (corresponding to the last $w_k^0$ before each change of direction of $\gamma$; recall that $\gamma$ is made of a finite number of straight segments) for which we still have $|z_{k+1}-z_k| \leq 2\sqrt{2}\eps$.
    \par
    We check the technical hypothesis of Lemma \ref{lem:Green:boundary}. Recall that by assumption, the boundary $\partial U$ is flat in a $\d$-neighbourhood of the endpoint $\zp$ of $\gamma$, and that except for its endpoint $\gamma$ is a path in the open set $U$. This implies that there exists $\d' = \d'(\gamma) \leq \d$ such that for all $w, b$ belonging to some edges of $E_{\eps}$, $|b-\overline{w}| \geq \delta'/2$ and both $b$ and $\overline{w}$ are either in $B(\zp, \delta'/2)$ or at distance at least $\delta'/2$ of the boundary of $U^r$.\par
    On the one hand, by Lemma \ref{lem:Green:boundary} and in particular the weaker consequence Equation \eqref{eq:K:bounded}, there exists $C = C(\gamma) >0$ such that 
    \be\label{eq:unif:control}
        \forall b,w \in E_{\eps},~ |\tK^{-1}_{(b,1),(w,2)}| \leq C\eps.
    \ee
    On the other hand, upon increasing $C$ we have $|F_{\eps}|\leq C$ and $|E_{\eps}| \leq C\eps^{-1}$, so we can neglect the edges of $F_{\eps}$ in Equation \eqref{eq:trace:n} and group the remaining edges by packets:
    \be\label{eq:trace-working-1}
        \begin{aligned}
            \tr\big((\tS\tK^{-1})^n\big) 
            &\approx 2\sum_{\underset{e_k = \{w_kb_k\}}{e_0,\dots,e_{n-1} \in E_{\eps} \setminus F_{\eps}}} \prod_{k=0}^{n-1}\zeta_{w_k,b_k}\sgn(E_{\eps})_{w_k,b_k}(-1)^{\mathds{1}\{w_k \in W_0\}} \tK^{-1}_{(b_k,1),(w_{k+1},2)}\\
            &\approx 2 \sum_{1 \leq k_1,\dots,k_n \leq N_{\eps}} \sum_{\underset{\forall 1\leq i \leq n}{e_i = \{w_ib_i\} \in E_{\eps}(k_i)}}\prod_{i=0}^{n-1}\zeta_{w_i,b_i}\sgn(E_{\eps})_{w_i,b_i}(-1)^{\mathds{1}\{w_i \in W_0\}} \tK^{-1}_{(b_i,1),(w_{i+1},2)},
        \end{aligned}
    \ee
    where the $\approx$ really means
    \be\label{eq:neglect}
        \begin{aligned}
        &\left|\tr\big((\tS\tK^{-1})^n\big) -  2 \sum_{1 \leq k_1,\dots,k_n \leq N_{\eps}} \sum_{\underset{\forall 1\leq i \leq n}{e_i = w_ib_i \in E_{\eps}(k_i)}}\prod_{i=0}^{n-1}\zeta_{w_i,b_i}\sgn(E_{\eps})_{w_i,b_i}(-1)^{\mathds{1}\{w_i \in W_0\}} \tK^{-1}_{(b_i,1),(w_{i+1},2)}\right|\\
        &\qquad \leq C^{2n}\eps.\\
        \end{aligned}
    \ee

From now on, we fix $1 \leq k_1,\cdots,k_n \leq N_{\eps}$ and take the corresponding term $s_{k_1,\cdots,k_n}$ in the sum on the right-hand side of Equation \eqref{eq:trace-working-1}
    \bestar
        s_{k_1,\dots,k_n} 
        =
        \sum_{\underset{\forall 1\leq i \leq n}{e_i = w_ib_i \in E_{\eps}(k_i)}}
        \prod_{i=0}^{n-1}
        \zeta_{w_i,b_i}\sgn(E_{\eps})_{w_i,b_i}(-1)^{\mathds{1}\{w_i \in W_0\}} 
        \tK^{-1}_{(b_i,1),(w_{i+1},2)},
    \eestar
and show that it is equal (up to an error of order $o(\eps)$) to 
    \bestar
        \tilde{s}_{k_1,\dots,k_n} 
        = 
        \sum_{\sigma\in \{\pm1\}^n}
        \prod_{i=0}^{n-1} i\sigma_i 
        F^{(\sigma_i)}_{-\sigma_i\sigma_{i+1}}(\overline{z_{k_{i+1}}},z_{k_i}) dz^{(\sigma_i)}_{k_i}.
    \eestar

Write $r_i=1$ if $w_i\in W_0$, and $r_i=-1$ if $w_i\in W_1$, and $s_i$ similar for $b_i$. Then we have, firstly: $(-1)^{\mathds{1}\{w_i \in W_0\}} = -r_i$; secondly $\tK^{-1}_{(b_i,1),(w_{i+1},2)} \leq C\eps$ by Equation \eqref{eq:unif:control}; thirdly Lemma \ref{lem:Green:boundary} gives us that there exists $\phi_{\d}(\eps) \overset{\eps \to 0}{\longrightarrow} 0$ such that
\bestar
    \left|\tK^{-1}_{(b_i,1),(w_{i+1},2)} 
        -
        \eps\tfrac{1}{2}\left[F_+ +r_{i+1}F_- +s_i\ol{F_-} +r_{i+1}s_i\ol{F_+}\right]\right|
        \le \eps \phi_{\d}(\eps)
\eestar
(where we omit the argument $(\ol{z_{k_i+1}},z_{k_i})$ for each of the functions $F_+, F_-, \ol{F_-}, \ol{F_+}$). Actually, we could take $\phi_{\d}(\eps) = C(\d)\eps$ due to Lemma \ref{lem:Green:boundary}, but we stay at a higher level of generality for further reference. Finally one can verify that 
    \bestar
        \zeta_{w_i,b_i}\sgn(E_{\eps})_{w_i,b_i}
        =
        \eps^{-1}(-i)[\tfrac{1}{2}(1-r_is_i)dx_{k_i} + \tfrac{1}{2}(1+r_is_i)idy_{k_i}],
    \eestar
indeed, use that $\tfrac{1}{2}(1-r_is_i)=\mathds{1}\{r_i\neq s_i\}$ and $\tfrac{1}{2}(1+r_is_i)=\mathds{1}\{r_i= s_i\}$, and the values of $\zeta_{w_i,b_i}\sgn(E_{\eps})_{w_i,b_i}$ in Figure \ref{fig:zipper:local}.
It follows that 
    \bestar
        \zeta_{w_i,b_i}\sgn(E_{\eps})_{w_i,b_i}(-1)^{\mathds{1}\{w_i \in W_0\}}
        =
        \eps^{-1}\tfrac{1}{2}i[r_idz_{k_i} - s_id\ol{z_{k_i}}],
    \eestar
and so we have, up to an error $4^nC^{n-1}\eps^n\phi_{\d}(\eps)$,
    \bestar
        s_{k_1,\dots,k_n} 
        =
        \sum_{r,s\in\{\pm1\}^n}
        \prod_{i=0}^{n-1}
        \tfrac{1}{4}i[r_i dz_{k_i} - s_i d\ol{z}_{k_i}]
        \left[F_+ +r_{i+1}F_- +s_i\ol{F_-} +r_{i+1}s_i\ol{F_+}\right].
    \eestar
We multiply out this product. Recall that $dz_{k_i}^{(1)}=dz_{k_i}$ and $dz_{k_i}^{(-1)}=d\ol{z_{k_i}}$, and similar for $F_+$ and $F_-$. The above is equal to
    \bestar
    \begin{split}
        &\sum_{r,s\in\{\pm1\}^n}
        \sum_{\kappa,\theta\in\{0,1\}^n}
        \sum_{\substack{\tau\in\{\pm1\}^n}}
        \prod_{i=0}^{n-1}
        \tfrac{1}{4}i\left[r_i^{1-\theta_i}(-s_i)^{\theta_i}dz_{k_i}^{((-1)^{\theta_i})}\right]
        \left[
            r_{i+1}^{\tfrac{1}{2}(1-\tau_i)+\kappa_i} s_i^{\kappa_i} F_{\tau_i}^{((-1)^{\kappa_i})}
        \right]\\
        &\hspace{1cm}=
        \sum_{\substack{r,s,\tau\in\{\pm1\}^n \\ \kappa,\theta\in\{0,1\}^n}}
        \prod_{i=0}^{n-1}
        \tfrac{1}{4}i
        (-1)^{\theta_i} r_i^{1-\theta_i} r_{i+1}^{\tfrac{1}{2}(1-\tau_i)+\kappa_i} s_i^{\theta_i+\kappa_i}
        F_{\tau_i}^{((-1)^{\kappa_i})} dz_{k_i}^{((-1)^{\kappa_i})}.
    \end{split}
    \eestar
If for some $i=0,\dots,n-1$, we have $\kappa_i+\theta_i=1$, then summing over $s_i=\pm1$ gives zero. So the only terms that remain have $\theta_i=\kappa_i$ for all $i=0,\dots,n-1$, and then the product is independent of $s$, so the above is
    \bestar
    \begin{split}
        &\sum_{\substack{r,\tau\in\{\pm1\}^n \\ \kappa\in\{0,1\}^n}}
        \prod_{i=0}^{n-1}
        \tfrac{1}{2}i
        (-1)^{\kappa_i} r_i^{1-\kappa_i} r_{i+1}^{\tfrac{1}{2}(1-\tau_i)+\kappa_i}
        F_{\tau_i}^{((-1)^{\kappa_i})} dz_{k_i}^{((-1)^{\kappa_i})}\\
        &\hspace{1cm}=
        \sum_{\substack{r,\tau\in\{\pm1\}^n \\ \kappa\in\{0,1\}^n}}
        \prod_{i=0}^{n-1}
        \tfrac{1}{2}i
        (-1)^{\kappa_i} r_i^{1-\kappa_{i+1}+\kappa_i + \tfrac{1}{2}(1-\tau_i)}
        F_{\tau_i}^{((-1)^{\kappa_i})} dz_{k_i}^{((-1)^{\kappa_i})},
    \end{split}
    \eestar
after reindexing $r_{i+1}=r_i$. Now similarly to the above, if for some $i=0,\dots,n-1$, we have that $1-\kappa_{i+1}+\kappa_i +\tfrac{1}{2}(1-\tau_i)$ is odd, then the sum over $r_i=\pm1$ gives zero. Hence the only terms that remain have, for all $i=0,\dots,n-1$, that $1-\kappa_{i+1}+\kappa_i +\tfrac{1}{2}(1-\tau_i)$ is even, which is the same as $\mathds{1}\{\kappa_i\neq\kappa_{i+1}\} = \mathds{1}\{\tau=1\}$, which is the same as $\tau_i=-(-1)^{\kappa_i}(-1)^{\kappa_{i+1}}$. The sum is then independent of $r$, and we obtain
    \bestar
        s_{k_1,\dots,k_n} 
        =
        \sum_{\substack{\kappa\in\{0,1\}^n}}
        \prod_{i=0}^{n-1}
        i
        (-1)^{\kappa_i} 
        dz_{k_i}^{((-1)^{\kappa_i})} F_{-(-1)^{\kappa_i}(-1)^{\kappa_{i+1}}}^{((-1)^{\kappa_i})}.
    \eestar
Finally, reparameterizing as $\sigma_i=(-1)^{\kappa_i}$ gives 
    \bestar
        \sum_{\sigma\in \{\pm1\}^n}
        \prod_{i=0}^{n-1} i\sigma_i 
        F^{(\sigma_i)}_{-\sigma_i\sigma_{i+1}}(\overline{z_{k_{i+1}}},z_{k_i}) dz^{(\sigma_i)}_{k_i},
    \eestar
as desired. Hence, Equation \eqref{eq:trace-working-1} becomes
    \be\label{eq:asymp:tr}
            \left|
            \tr((\tS\tK^{-1})^n) 
            - 
            2 \sum_{1 \leq k_1,\cdots,k_n \leq N_{\eps}} 
            \sum_{\sigma\in \{\pm1\}^n}
            \prod_{i=0}^{n-1} i\sigma_i 
            F^{(\sigma_i)}_{-\sigma_i\sigma_{i+1}}(\overline{z_{k_{i+1}}},z_{k_i}) dz^{(\sigma_i)}_{k_i}
            \right|\\
            \leq C^{n}\phi_{\d}(\eps).
    \ee
    for some new constant $C >0$. Recall that for all $1 \leq k \leq N_{\eps}$, $z_k = w_k^0$, $|dz_k| = \sqrt{2}\eps$ and $z_{k+1} = z_k + dz_k$ (except for at most $C$ values of $k$ for which we still have $|z_{k+1}-z_k|\leq 2\sqrt{2}\eps$). We obtain the sum of $2^n$ (one for each $\sigma \in \{\pm1\}^n$) $n$-fold Riemann sums approximating $n$-fold integrals along the path $\gamma$ of an analytic function when $\eps \to 0$. For each of these terms, the error can be bounded in terms of bounds on the derivatives of $F_+$ and $F_-$ around $\gamma$ (by Remark \ref{rem:Schwarz}, these bounds also work near $\zp$): we can find a constant $C$ depending only on $F_+, F_-$ and their derivatives (i.e. depending only on $U$ and $\gamma$) such that
    \bestar
        \begin{aligned}
        \left|2 \sum_{1 \leq k_1,\cdots,k_n \leq N_{\eps}} s_{k_1,\cdots,k_n} - c_n(z,\zp)\right| \leq 2^nC^n \eps.
        \end{aligned}
    \eestar
    Upon increasing $C$, the triangle inequality gives
    \be\label{eq:compare:coef}
        \left|\tr((\tS\tK^{-1})^n) - c_n(z,\zp)\right| \leq C^n \phi_{\d}(\eps).
    \ee
    for some new constant $C >0$. We are ready to conclude: since the $c_k(z,\zp)$ grow at most exponentially in $k$ (see Equation \eqref{eq:ck:bounded}), the power series 
    \bestar
        S_{\eps}(\a) = \frac{1}{2}\sum_{k=1}^{\infty} \frac{(-1)^{k-1}}{k}(2\a)^k \tr((\tS\tK^{-1})^k) 
    \eestar
    has radius of convergence at least $C^{-1}$ for some absolute constant,
    and Equation \eqref{eq:compare:coef} implies that for all $k$,
    \be\label{eq:coefs}
        \tr((\tS\tK^{-1})^k) = c_k(z,\zp) + o(\eps).
    \ee
    Since this power series has no constant coefficient, its exponential $\exp(S_{\eps}(\a))$ is also a power series with radius of convergence at least $C^{-1}$. The complete Bell polynomials enable us to express the coefficients of the exponential of a power series: for all $k \geq 0$, for all power series with no constant term,
\bestar
    \exp\left(\sum_{k=1}^{\infty}a_k \frac{X^k}{k!}\right) = \sum_{n=0}^{\infty}B_n(a_1,\dots,a_n)\frac{X^n}{n!}.
\eestar
    Using Equation \eqref{eq:coefs}, we obtain that the $n$-th coefficient of the power series $\exp(S_{\eps}(\a))$ is
    \bestar
        \begin{aligned}
        [\exp(S_{\eps}(\a))](n) &=
         \frac{1}{n!}B_n\left(\left(\frac{(-2)^{k-1}k!}{k}\tr((\tS\tK^{-1})^k)\right)_{1 \leq k \leq n}\right)\\
        &= \frac{1}{n!}B_n\left(\left((-2)^{k-1}(k-1)!c_k(z,\zp)\right)_{1 \leq k \leq n}\right) + o_{\eps \to 0}(1).
        \end{aligned}
    \eestar
    This concludes by uniqueness of the coefficients in Equation \eqref{eq:second:step}.
    \end{proof}


\section{Proof of the main theorem: the shifted case}\label{sec:shifted}
In this section, we prove Theorem \ref{thm:asymp:ne:oe} when $\mu_{\eps} = \mu^2_{\eps}$, that is in the shifted case. The proof is very similar in spirit to the proof of the last section, but we have to use the estimates of the inverse Kasteleyn matrix in piecewise Temperleyan domains obtained by Russkikh in \cite{Russ} (see Lemma \ref{lem:russ:boundary} of the Appendix) instead of those of Lemma \ref{lem:Green:boundary} which only work for Temperleyan domains. We follow the same steps and indicate the main changes.

\begin{proof}
As in the preceding section, let $\gamma$ be a path in $U$ from $z$ to $\zp$ and $E_{\eps} = E_{\eps}(\gamma)$ be the associated zipper, see Figure \ref{fig:zipper}. Let $\eps >0$ be fixed. For every $\a \geq 0$, we define a connection on $\Ge$ by setting for any directed edge $e \in E_{\eps}(\gamma)$, $\phi_e = \begin{pmatrix}
    1+\a/2 & -\a/2 \\ \a/2 & 1-\a/2\end{pmatrix}$, $\phi_{e^{-1}} = \phi_{e}^{-1}$, and $\phi_e = I_2$ on all other bulk edges. We also define $\psi_e = \begin{pmatrix} 1 \\ 0\end{pmatrix}$ for all edges $e$ linking a boundary point with a bulk point. We denote by $K_{\alpha}$ the Kasteleyn matrix associated with this connection and the Kasteleyn phases of discrete holomorphy given by Equation \eqref{kast-eq-conplex-kast-weights}, and by $K = K_0$. Corollary \ref{kast:corr} writes. 
    \be\label{eq:kast:Ka:shifted}
            \pf K_{\a} = \sum_{\om\in\Om} 
            \prod_{\mathrm{loops} \ C}\tr(\phi_C)
            \prod_{\mathrm{arcs} \ A}\psi_{b_A}^\T\phi_A\psi_{w_A}.
    \ee
    The proof now follows the same three steps as the preceding section. 
    
    \bigskip
    
    \paragraph{\textbf{First step: the right-hand side of Equation \eqref{eq:kast:Ka:shifted}.}} Loops get a weight $2$, arcs enclosing $z$ get a weight $1\pm \a/2$ according to their orientation, so we obtain, using the same notations as in the preceding proof,
\bestar
    \pf K_{2\alpha} = \sum_{\om \in \Om}2^{c_{\eps}}(1-\a^2)^{r_{\eps}(z)}(1-o_{\eps}(z)); \quad  \quad \pf K = \sum_{\om \in \Om}2^{c_{\eps}}
\eestar
and dividing by $\pf K$ we obtain 
\be\label{eq:development:powers:a:bis}
    \frac{\pf K_{2\a}}{\pf K}
            = \sum_{k =0}^{\infty}(-1)^k\a^{2k}\Ee\left[\binom{\frac{\nez-\oez}{2}}{k}\right] - \sum_{k =0}^{\infty}(-1)^k\a^{2k+1}\Ee\left[\binom{\frac{\nez-\oez}{2}}{k}\oez\right]
    \ee
which is exactly Equation \eqref{eq:development:powers:a}, and concludes the first step of the proof.

\bigskip

\paragraph{\textbf{Second step: the left-hand side of Equation \eqref{eq:kast:Ka:shifted}.}} We use the same notations as in the preceding proof: $E_{\eps}$ for the zipper, $\sgn(E_{\eps})_{w,b}$ for the direction of an edge of the zipper. For $i \in \{1,2\}$, we write the vertices of $\Ge^i$ as $(x,i)$ with $x \in \Ge$ if $i=1$, $x \in \Ge \setminus (\RR \times \{0\})$ if $i=2$. Contrary to the last section, the matrix $K$ is a \emph{block matrix} since $\psi = \begin{pmatrix}1\\0\end{pmatrix}$: $K_{(w,i),(b,j)} = 0$ as soon as $i \neq j$. On the $i=j=1$ block, it is the Kasteleyn matrix on the Temperleyan domain $\Ge^1$: it has Dirichlet boundary conditions everywhere. We write $K_{(w,1),(b,1)} = K^1_{w,b}$. On the $i=j=2$ block, it is the Kasteleyn matrix on the \emph{piecewise} Temperleyan domain $\Ge^2$: it has Dirichlet boundary conditions everywhere, except on the horizontal axis where it has Neumann boundary conditions (see \cite{Russ} for a general discussion of mixed boundary conditions for Kasteleyn matrices). We write $K_{(w,2),(b,2)} = K^2_{w,b}$. The asymptotic of $(K^1)^{-1}$ and $(K^2)^{-1}$ when $\eps \to 0$ were obtained in \cite{Russ} and are recalled in Lemma \ref{lem:russ:boundary}. Note that contrary to the last section, the vertical orientation is correct for $K^1$ and $K^2$, hence we do not need any gauge change. We can write
\bestar
    (K_{2\a})_{(w,i),(b,j)} = K_{(w,i),(b,j)} + \a S_{(w,i),(b,j)}
\eestar
with 
\be\label{eq:def:S:shifted}
    S_{(w,i),(b,j)} = \mathds{1}\{w \sim b\}\zeta_{w,b}\sgn(E_{\eps})_{w,b}(-1)^{i-1}.
\ee
In other words, $K_{2\a} = K + \a S$ and as in the preceding section Equation \eqref{eq:development:powers:a:bis} becomes, for all $\a \in (0, \rho(SK^{-1})^{-1})$,
\be\label{eq:second:step:shifted}
        \begin{aligned}
        &\sum_{k =0}^{\infty}(-1)^k\a^{2k}\Ee\left[\binom{\frac{\nez-\oez}{2}}{k}\right] - \sum_{k =0}^{\infty}(-1)^k\a^{2k+1}\Ee\left[\binom{\frac{\nez-\oez}{2}}{k}\oez\right]\\
        &\quad = \exp \left(\frac{1}{2}\sum_{k=1}^{\infty} \frac{(-1)^{k-1}}{k}\a^k\tr\big((SK^{-1})^k\big)\right),
        \end{aligned}
    \ee
    which concludes the second step.

    \bigskip

    \paragraph{\textbf{Third step: the asymptotic of the trace term.}} We only need to prove, as in the preceding section (up to a factor $2^n$) that
    \be\label{eq:asymp:tr}
            \left|
            \tr((SK^{-1})^n) 
            - 
            2^{n+1} \sum_{1 \leq k_1,\cdots,k_n \leq N_{\eps}} 
            \sum_{\sigma\in \{\pm1\}^n}
            \prod_{i=0}^{n-1} i\sigma_i 
            F^{(\sigma_i)}_{-\sigma_i\sigma_{i+1}}(\overline{z_{k_{i+1}}},z_{k_i}) dz^{(\sigma_i)}_{k_i}
            \right|\\
            \leq C^{n}\phi_{\d}(\eps).
    \ee
    for some $C = C(\gamma) >0$, $\phi_{\d}(\eps) \overset{\eps \to 0}{\longrightarrow}$. Indeed, if this holds, the proof of the last section carries through. We prove it using the same general strategy as in the third step of the last section. Let $n \geq 1$ be fixed. Using the convention that $(w_n,j_n) = (w_0,j_0)$, we can write
    \be\label{eq:product:to:develop}
        \tr\big((SK^{-1})^n\big) = 2\sum_{(w_0,j_0), \dots, (w_{n-1},j_{n-1})}\prod_{k=0}^{n-1} (SK^{-1})_{(w_k,j_k),(w_{k+1},j_{k+1})}.
    \ee
    The factors in the product are computed as follows: for all $u, v \in W(\Ged)$, for all $i, j \in \{1,2\}$, using Equation \eqref{eq:def:S:shifted}, and the fact that $K^{-1}_{(w,i),(b,j)}$ vanishes when $i \neq j$, 
    \bestar
            (SK^{-1})_{(u,i),(v,j)} = \sum_{(b,k) \in B(\Ged)} S_{(u,i),(b,k)}K^{-1}_{(b,k),(v,j)} = (-1)^{i-1}\sum_{b: \{ub\} \in E_{\eps}} \zeta_{u,b}\sgn(E_{\eps})_{u,b}(K^j)^{-1}_{b,v}.
    \eestar
    Developing the product in Equation \eqref{eq:product:to:develop} yields
    \be\label{eq:sum:explodes}
            \tr\big((SK^{-1})^n\big) 
            =2 \sum_{j_0,\dots,j_{n-1} \in \{1,2\}}\sum_{\underset{e_k = \{w_kb_k\}}{e_0,\dots,e_{n-1} \in E_{\eps}}} \prod_{k=0}^{n-1}\zeta_{w_k,b_k}\sgn(E_{\eps})_{w_k,b_k}(-1)^{j_k} (K^{j_{k+1}})^{-1}_{b_k,w_{k+1}}
    \ee
    (note that we also reindexed $j_{k+1} = j_k$). We choose the same path $\gamma$ in $U$ as in the preceding section and we group the edges by packets as before, see Figure \ref{fig:zipper:local} and the corresponding paragraph. We use the same notation: 
    \bestar
        E_{\eps} = \bigg(\bigsqcup_{1 \leq k \leq N_{\eps}}E_{\eps}(k)\bigg) \sqcup F_{\eps}.
    \eestar
    Here we need to be more careful than in the last section when neglecting the edges in $F_{\eps}$, since the terms are no more $O(\eps)$ when $w_{k+1}$ and $b_k$ are getting close. By Lemma \ref{lem:russ:boundary} and in particular the weaker consequence Equation \eqref{eq:bound:K:russ}, 
    \bestar
        |(K^{j_{k+1}})^{-1}_{b_k,w_{k+1}}-\Delta(w_{k+1},b_k)| \leq C\eps
    \eestar
    for some constant $C = C(\d)$. In Equation \eqref{eq:sum:explodes}, the error $\Delta(w_{k+1},b_k)$ cancels between the terms $j_k = \pm1$, so we can rewrite it as
    \bestar
        2 \sum_{j_0,\dots,j_{n-1} \in \{1,2\}}\sum_{\underset{e_k = \{w_kb_k\}}{e_0,\dots,e_{n-1} \in E_{\eps}}} \prod_{k=0}^{n-1}\zeta_{w_k,b_k}\sgn(E_{\eps})_{w_k,b_k}(-1)^{j_k} \left((K^{j_{k+1}})^{-1}_{b_k,w_{k+1}}-\Delta(w_{k+1},b_k)\right).
    \eestar
    Now, we can neglect the edges in $F_{\eps}$: since $|F_{\eps}|\leq C$ (upon increasing $C$), it holds that up to $C^n\eps$ (upon increasing $C$ again),
    \be\label{eq:trace:shifted}
        \begin{aligned}
            \tr\big((SK^{-1})^n\big) 
            =2 \sum_{1 \leq k_1,\dots,k_n \leq N_{\eps}} \sum_{\underset{\forall 1\leq i \leq n}{j_i \in \{1,2\}}}\sum_{\underset{\forall 1\leq i \leq n}{e_i = w_ib_i \in E_{\eps}(k_i)}}&\prod_{i=0}^{n-1}\bigg\{\zeta_{w_i,b_i}\sgn(E_{\eps})_{w_i,b_i}(-1)^{j_i-1}\\
            &\times\left((K^{j_{i+1}})^{-1}_{b_i,w_{i+1}}-\Delta(w_{i+1},b_i)\right)\bigg\}.
        \end{aligned}
    \ee
    From now on, we fix $1 \leq k_1,\cdots,k_n \leq N_{\eps}$ and take the corresponding term $s_{k_1,\cdots,k_n}$ in the sum on the right-hand side of Equation \eqref{eq:trace:shifted}:
    \bestar
        \begin{aligned}
        s_{k_1,\dots,k_n} 
        =
        \sum_{j_i \in \{1,2\}}\sum_{e_i = w_ib_i \in E_{\eps}(k_i)}&\prod_{i=0}^{n-1}\zeta_{w_i,b_i}\sgn(E_{\eps})_{w_i,b_i}(-1)^{j_{i}-1}
        \left((K^{j_{i+1}})^{-1}_{b_i,w_{i+1}}-\Delta(w_{i+1},b_i)\right).
        \end{aligned}
    \eestar
Recall that $w_i, b_i$ are within distance $2\sqrt{2}\eps$ of $z_{k_i}$. We show that (up to an error of order $o(\eps)$) 
    \bestar
        s_{k_1,\dots,k_n}  = 2^n\sum_{\sigma\in \{\pm1\}^n}
        \prod_{i=0}^{n-1} i\sigma_i 
        F^{(\sigma_i)}_{-\sigma_i\sigma_{i+1}}(\overline{z_{k_{i+1}}},z_{k_i}) dz^{(\sigma_i)}_{k_i}.
    \eestar
    Note that in this expression, $F_{\pm}$ are the functions defined in Equation \eqref{eq:def:Fpm} and appearing in Lemma \ref{lem:Green:boundary}, and $F_{\pm}^{(\sigma)}$ for $\sigma = \pm1$ denotes complex conjugation (in particular $F_{\pm}^{(\sigma)}$ is \emph{not} the $F_{\pm}^j,~j \in \{1,2\}$ appearing in Lemma \ref{lem:russ:boundary}, though they are related as we will see). 
    
    Let $r_i, s_i \in \{\pm 1\}$ be such that $w_i \in W_{\frac{1-r_i}{2}}, b_i \in B_{\frac{1-s_i}{2}}$. By Lemma \ref{lem:russ:boundary}, 
    \bestar
        \begin{aligned}
        (K^{j_{i+1}})^{-1}_{b_i,w_{i+1}}- \Delta(w_{i+1},b_i) 
        &=
        \tfrac{\eps}{2}\left[F_+^{j_{i+1}}(z_{k_{i+1}},z_{k_i}) +r_{i+1}F_-^{j_{i+1}}(z_{k_{i+1}},z_{k_i})\right.\\
        & \left.+s_i\ol{F_-^{j_{i+1}}(z_{k_{i+1}},z_{k_i})} +r_{i+1}s_i\ol{F_+^{j_{i+1}}(z_{k_{i+1}},z_{k_i})}\right] + \eps\phi_{\delta}(\eps).
        \end{aligned}
    \eestar
    where $\phi_{\delta}(\eps) \overset{\eps \to 0}{\longrightarrow} 0$. We will omit the argument $(z_{k_{i+1}},z_{k_i})$ and the dependence in $j_{i+1}$ for each of the functions $F_+, F_-, \ol{F_-}, \ol{F_+}$. Recall that this implies (see Equation \eqref{eq:bound:K:russ}) that 
    \bestar
    |(K^{j_{i+1}})^{-1}_{b_i,w_{i+1}}- \Delta(w_{i+1},b_i)| \leq C(\delta)\eps.
    \eestar
    Moreover, one can verify that 
    \bestar
        \zeta_{w_i,b_i}\sgn(E_{\eps})_{w_i,b_i}
        =-r_i
        \tfrac{\eps^{-1}}{2}i[r_idz_{k_i} - s_id\ol{z_{k_i}}],
    \eestar
and so we have, up to an error $C^n\eps\phi_{\delta}(\eps)$ (upon increasing $C$),
    \bestar
        s_{k_1,\dots,k_n} 
        = 
        \sum_{r,s\in\{\pm1\}^n, j_i \in \{1,2\}}
        \prod_{i=0}^{n-1}
        \tfrac{r_i(-1)^{j_i}}{4}i[r_i dz_{k_i} - s_i d\ol{z}_{k_i}]
        \left[F_+ +r_{i+1}F_- +s_i\ol{F_-} +r_{i+1}s_i\ol{F_+}\right].
    \eestar
We multiply out this product. Recall that $dz_{k_i}^{(1)}=dz_{k_i}$ and $dz_{k_i}^{(-1)}=d\ol{z_{k_i}}$, and similar for $F_+$ and $F_-$. The above is equal to
    \bestar
        \sum_{j_i \in \{1,2\}}\sum_{\substack{r,s,\tau\in\{\pm1\}^n \\ \kappa,\theta\in\{0,1\}^n}}
        \prod_{i=0}^{n-1}
        \tfrac{(-1)^{j_i}r_i}{4}i
        (-1)^{\theta_i} r_i^{1-\theta_i} r_{i+1}^{\tfrac{1}{2}(1-\tau_i)+\kappa_i} s_i^{\theta_i+\kappa_i}
        F_{\tau_i}^{((-1)^{\kappa_i})} dz_{k_i}^{((-1)^{\theta_i})}.
    \eestar
If for some $i=0,\dots,n-1$, we have $\kappa_i+\theta_i=1$, then summing over $s_i=\pm1$ gives zero. So the only terms that remain have $\theta_i=\kappa_i$ for all $i=0,\dots,n-1$, and then the product is independent of $s$, so the above is
    \bestar
        \sum_{j_i \in \{1,2\}}\sum_{\substack{r,\tau\in\{\pm1\}^n \\ \kappa\in\{0,1\}^n}}
        \prod_{i=0}^{n-1}
        \tfrac{(-1)^{j_i}r_i}{2}i
        (-1)^{\kappa_i} r_i^{1-\kappa_{i+1}+\kappa_i + \tfrac{1}{2}(1-\tau_i)}
        F_{\tau_i}^{((-1)^{\kappa_i})} dz_{k_i}^{((-1)^{\kappa_i})},
    \eestar
after reindexing $r_{i+1}=r_i$. Now similarly to the above, if for some $i=0,\dots,n-1$, we have that $1-\kappa_{i+1}+\kappa_i +\tfrac{1}{2}(1-\tau_i)$ is \emph{even} (contrary to the last section), then the sum over $r_i=\pm1$ gives zero. Hence the only terms that remain have for all $i=0,\dots,n-1$ that $1-\kappa_{i+1}+\kappa_i +\tfrac{1}{2}(1-\tau_i)$ is odd, which is the same as $\mathds{1}\{\kappa_i\neq\kappa_{i+1}\} \neq \mathds{1}\{\tau=1\}$, which is the same as $\tau_i=(-1)^{\kappa_i}(-1)^{\kappa_{i+1}}$. The sum is then independent of $r$, and we obtain
    \bestar
        s_{k_1,\dots,k_n} 
        =
        \sum_{j_i \in \{1,2\}}\sum_{\substack{\kappa\in\{0,1\}^n}}
        \prod_{i=0}^{n-1}(-1)^{j_i}
        i
        (-1)^{\kappa_i} 
        dz_{k_i}^{((-1)^{\kappa_i})} F_{(-1)^{\kappa_i}(-1)^{\kappa_{i+1}}}^{((-1)^{\kappa_i})}.
    \eestar
Finally, reparameterizing as $\sigma_i=(-1)^{\kappa_i}$ gives 
    \bestar
        \sum_{j_i \in \{1,2\}}\sum_{\sigma\in \{\pm1\}^n}
        \prod_{i=0}^{n-1}(-1)^{j_i} i\sigma_i 
        F^{(\sigma_i)}_{\sigma_i\sigma_{i+1}}dz^{(\sigma_i)}_{k_i},
    \eestar
The last step of the argument is specific to the shifted case. It is time to recall the dependence of the $i$-th term in $j_{i+1},z_{k_i},\ol{z_{k_{i+1}}}$ which was hidden in the notation for a while. For $j \in \{1,2\}$, $\sigma \in \{\pm1\}$, $(F_\pm^j)^{(\sigma)}$ is $(F_\pm^1)^{(\sigma)}$ if $j=1$, $(F_\pm^2)^{(\sigma)}$ if $j=2$ (recall that the exponent $(\sigma)$ denotes complex conjugation). Using the explicit expression for $F_\pm^1$ and $F_\pm^2$ (defined on $U$) in terms of $F_{\pm}$ (defined on $U^r$), see Equation \eqref{eq:F^j} of Appendix \ref{app:B}, $s_{k_1,\dots,k_n}$ can be written as
\bestar
    \sum_{j_i \in \{1,2\},\sigma\in \{\pm1\}^n}
        \prod_{i=0}^{n-1}(-1)^{j_i} i\sigma_i 
        \left[F^{(\sigma_i)}_{\sigma_i\sigma_{i+1}}(z_1,z_2)+(-1)^{j_{i+1}}F^{(\sigma_i)}_{-\sigma_i\sigma_{i+1}}(z_1,\ol{z_2})\right]dz^{(\sigma_i)}_{k_i}.
\eestar
Developing the product yields
\bestar
    \sum_{j_i \in \{1,2\},\sigma\in \{\pm1\}^n, \kappa \in \{\pm 1\}^n}
        \prod_{i=0}^{n-1} (-1)^{j_i}i\sigma_i 
        (-1)^{\frac{\kappa_i-1}{2} j_{i+1}}F^{(\sigma_i)}_{\kappa_i\sigma_i\sigma_{i+1}}(z_1,z_2^{(\kappa_i)})dz^{(\sigma_i)}_{k_i}.
\eestar
If for some $i=0,\dots, n-1$ we have $\kappa_i=1$, then summing over $j_{i+1} = \pm 1$ gives zero, so the only terms remaining have $\kappa_i = -1$ for all $i$ and then the product is independent of $\kappa$, so the above is 
\bestar
    \sum_{j_i \in \{1,2\},\sigma\in \{\pm1\}^n}
        \prod_{i=0}^{n-1} i\sigma_i 
        F^{(\sigma_i)}_{-\sigma_i\sigma_{i+1}}(z_1,\ol{z_2})dz^{(\sigma_i)}_{k_i}
        = 2^n\sum_{\sigma \in \{\pm1\}^n}
        \prod_{i=0}^{n-1} i\sigma_i 
        F^{(\sigma_i)}_{-\sigma_i\sigma_{i+1}}(z_1,\ol{z_2})dz^{(\sigma_i)}_{k_i}
\eestar
which concludes the proof.
\end{proof}


\section{The example of a rectangle: computing the number of arcs crossing in the vertical direction} \label{sec:rectangle}

In this section, we consider a specific example of a domain - a rectangle, with $\delta$ given by two opposing sides. This is strongly inspired by section 9 of \cite{Ken14}. This setup doesn't fit the precise definition of our folded or shifted dimers model, but, as in \cite{Ken14}, we can compute new interesting quantities by directly diagonalizing the Kasteleyn matrix and using Proposition \ref{kast:prop:kast}. \\

Let $G$ be the subgraph of $\ZZ^2$ with vertex set  $[0,n] \times [1,m]$. We define its boundary to be  $\partial = \{0,n\} \times [1,m]$. We define a graph $\mathcal{G}^r$ by taking two copies of $G$ with vertices labeled as $(x,y,1)$ and $(x,y,2)$ for $(x,y) \in G$ and by identifying $(x,y,1) = (x,y,2)$ for $(x,y) \in \partial G$. In other words, $\mathcal{G}^r$ is the cylinder $\ZZ/(2n\ZZ) \times [1,m]$ if we identify $(x,y,1) = (x,y)$ and $(x,y,2) = (-x,y)$, with the slight abuse of notation that we identify $x \in \ZZ$ with its projection $\bar{x} \in \ZZ/(2n\ZZ)$ so that in particular, for all $y$, $(n,y,1) = (n,y) = (-n,y) = (n,y,2)$.
A uniformly random dimer configuration on $\mathcal{G}^r$ gives a random configuration of disjoint loops, doubled edges and arcs, obtained by superimposing the two layers. Arcs link two boundary vertices that can be on the same side or on opposite sides of the cylinder.
\begin{remark}
    Note that this is an example where $\partial$ has two connected components. It complements the strip from Example \ref{ex:strip} where $\partial$ has one connected component.
\end{remark}
The arcs always have one black and one white endpoint since the vertices along the boundary (blue bullets and circles on Figure \ref{fig:torus}) are alternating in colour. We define the random variable $N_{n,m}$ to be the number of arcs traversing from one vertical boundary to the other.

\begin{figure}\centering
    \begin{overpic}[angle = 90, abs,unit=1mm,scale=0.5]{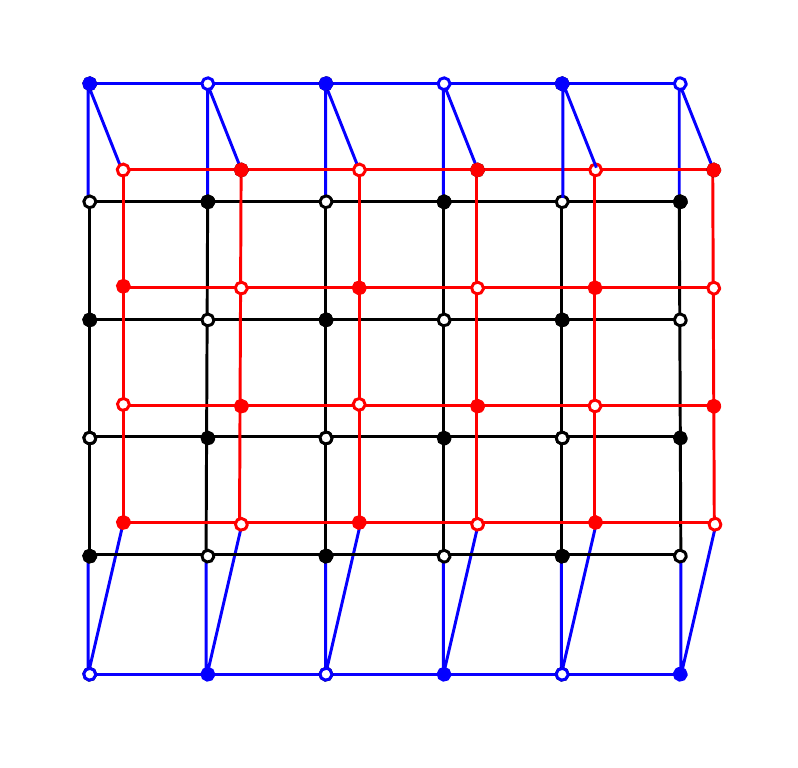}
    \put(2,3){\color{blue}$(0,1)$}
    \put(2,62){\color{blue}$(0,m)$}
    \put(51,2){\color{blue}$\substack{(n,1)\\=(-n,1)}$} 
    \put(13,3){$(1,1)$}
    \put(2,11){\color{red}$(-1,1)$}
    \end{overpic}
    \begin{overpic}[angle = 90,abs,unit=1mm,scale=0.9]{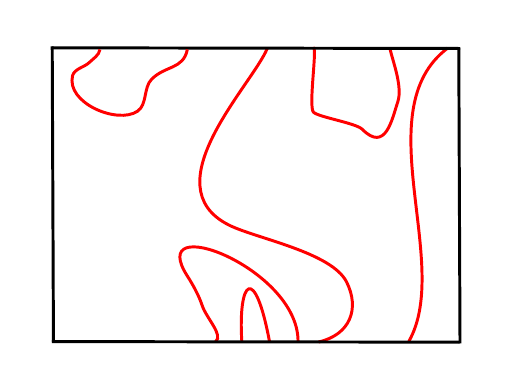}
    \end{overpic}
    \caption{The graph $\mathcal{G}^r$ on the left, with $(x,y,1)$ in black, $(x,y,2)$ in red and the boundary vertices in blue. On the right, a loops and arc configuration, with some of the arcs drawn in red (the loops are not drawn) and with $N_{n,m}=2$.}
    \label{fig:torus}
\end{figure}

\begin{theorem} 

Let $\tau > 0$ be fixed, $q = e^{-\pi \tau}$. When $n,m \to \infty$ with their ratio $\frac{n}{m} \to \tau$, the characteristic function of $N_{m,n}$ converges to an explicit limit
\bestar
\sum_{k} \PP[N_{m,n}=k]Y^k \overset{n \to \infty}{\longrightarrow}\prod_{\underset{j~odd}{j=1}}^{\infty}\frac{1+q^{2j}-2q^{j}+4Y^2q^{j}}{1+q^{2j}+2q^{j}}.
\eestar
\end{theorem}

\begin{proof}
We choose complex Kasteleyn phases $\zeta$ by putting $\zeta = 1$ on all horizontal edges and $\zeta = i$ on all vertical edges. These complex Kasteleyn phases can be obtained by a gauge transform from the ones of Figure \ref{kast-fig-real_weights}, so Proposition \ref{kast:eq:propkast} holds with the complex phases $\zeta$ by Remark \ref{kast:rmk:prop}. Then, for any choice of connection $\phi$ on $G$, if $K_{\phi}$ is the Kasteleyn matrix associated to the connection $\phi$ and the phases $\zeta$, Proposition \ref{kast:prop:kast} holds.
We now turn to the choice of the connection. Let $\rho \in \CC$ be fixed and $a \in \CC$ such that $a^n = \rho$. We can define a connection on the bulk by setting, for any horizontal bulk edge $e$ (pointing right), $\phi_e = \begin{pmatrix} a&0 \\ 0&\frac{1}{a}\end{pmatrix}$ and $\phi_{-e} = \phi_e^{-1}$. For any vertical bulk edge, we define $\phi_e = I_2$. We also define $\psi_e = \begin{pmatrix} a \\ \frac{1}{a}\end{pmatrix}$ for all (horizontal) edges $e$ pointing right linking a boundary point with a bulk point and $\psi_e = \begin{pmatrix} \frac{1}{a} \\ a\end{pmatrix}$ for all (horizontal) edges $e$ linking a boundary point with a bulk point pointing left. Denote by $K_a$ the Kasteleyn matrix associated with this choice of connection and orientation as in the preceding section and $K = K_1$ the Kasteleyn matrix associated with the trivial connection. We can now analyse the formula of Proposition \ref{kast:prop:kast}. The loops have monodromy $I_2$ so they get a weight $2$, and the arcs that come back to the same boundary also get a weight $2$, while the arcs that go from one side to the other get a weight $\psi_{b_A}^\T\phi_A\psi_{w_A} = a^n + a^{-n} = \rho + \rho^{-1}$. Hence, 
\bestar
    \pf K_a = \sum_{\om\in\Om} 2^{\#loops}2^{\#arcs}\left( \frac{\rho + \rho^{-1}}{2}\right)^{\#arcs~traversing},
\eestar
so
\be\label{eq:generating}
    \frac{\pf K_a}{\pf K} = \sum_{k \in \NN} \PP[N_{m,n}=k]\left( \frac{\rho + \rho^{-1}}{2}\right)^k.
\ee
We are left with computing the left-hand side, or rather its square $\frac{\det(K_a)}{\det{K}}$. We define the matrix $\bK_a$ by $\bK_a(w,b) = K_a(w,b)$ for all $b$ black and $w$ white in $\mathcal{G}^r$: it is the symmetric version of $K_a$ (which is by definition antisymmetric). Using the block structure of $K_a$ and $\bK_a$, we get $\det(\bK_a) = - \det(K_a)$ so the ratio of determinants is left unchanged. We now observe that the matrix $\bK_a$ that we have constructed is exactly that of Kenyon. Writing $\lambda = \rho^2$ (which satisfies $a^{2n}=\lambda$, $X = \rho^2 + \frac{1}{\rho^2}$), and since $K$ corresponds to $a=1$ and $X=2$, if $m$ is even, the results of Section 9 of \cite{Ken00} imply
\bestar
    \frac{\det(K_a)}{\det(K)} = \frac{\det(\bK_a)}{\det(\bK)} \underset{\frac{n}{m}\to \tau}{\overset{n \to \infty}{\longrightarrow}} \prod_{\underset{j~odd}{j=1}}^{\infty}\frac{(1+q^{2j}+Xq^j)^2}{(1+q^{j}+q^{2j})^2}.
\eestar
Combining this with Equation \ref{eq:generating} gives
\bestar
    \left(\sum_{k} \PP[N_{m,n}=k]\left( \frac{\rho + \rho^{-1}}{2}\right)^k\right)^2 \underset{\frac{n}{m}\to \tau}{\overset{n \to \infty}{\longrightarrow}} \prod_{\underset{j~odd}{j=1}}^{\infty}\frac{(1+q^{2j}+Xq^j)^2}{(1+q^j+q^{2j})^2}.
\eestar
For $\rho \in \CC$, such that $Y = \frac{\rho + \frac{1}{\rho}}{2} > 0$, $Y^2 = \frac{X + 2}{4}$ so $X > -2$ and all (squared) factors on both sides are non-negative so we can take the square root:
\bestar
    \begin{aligned}
        \sum_{k} \PP[N_{m,n}=k]Y^k \underset{\frac{n}{m}\to \tau}{\overset{n \to \infty}{\longrightarrow}} \prod_{\underset{j~odd}{j=1}}^{\infty}\frac{1+q^{2j}+Xq^j}{1+q^j+q^{2j}} =\prod_{\underset{j~odd}{j=1}}^{\infty}\frac{1-2q^{j}+q^{2j}+4Y^2q^{j}}{1+q^j+q^{2j}}.
    \end{aligned}
\eestar
\end{proof}

\appendix
\section{Computation in the continuum (by Avelio Sep\'ulveda)} \label{app:ALE}
The objective of this section is to compute the expected values of the random variables $n(z)$ and $o(z)$ for the ALE. These random variables are clearly conformally invariant, as the ALE itself is. It therefore suffices to compute them in a specific given domain, which we do below. Recall that the ALE has the same distribution regardless of the Dirichlet/Neumann boundary conditions of the accompanying GFF - this follows from the coupling of \cite{QW}. To be precise, we plan to show the following.
\begin{proposition}\label{p.ALE_moments}
    Let $\Ale$ be an ALE in the strip $\R\times [0,\pi/2]$ with arcs emanating from the lower boundary, $n(z)$ the number of arcs of $\Ale$ that separate the upper boundary from $z$ and $O(z)=(-1)^{n(z)}$. We have that
    \begin{align*}
    &\E\left[ n(z)\right ]=  \frac{1}{4}- \frac{(\Im(z))^2}{\pi^2}-\frac{2}{\pi^2}\log(\sin(\Im (z)))\\
&\mathbb P\left[O(z)=-1 \right] = 1/2-\Im(z)/\pi.    \end{align*}
\end{proposition}
Note importantly that these moments exactly coincide with Example \ref{ex:strip}. Note that $n(z)$ is almost surely finite as the ALE is locally finite (any compact set intersects finitely many arcs almost surely). Unfortunately while we have formulae for all the moments of $n_\eps(z)$ and $o_\eps(z)$ in the $\eps\to0$ limit by Theorem \ref{thm:asymp:ne:oe}, higher moments in the ALE are not currently available to us.\\

Before showing the proposition, let us define the strips $S^+=\R\times [0,\pi/2]$ and $\check S=\R\times[-\pi/2,\pi/2]$. Note that the conformal radius of these strips is given by
\begin{align*}
	&CR(z,\check S)= 2 \cos(\Im(z)),\\
	&CR(z,S^+) = \sin(2\Im(z)) = 2\sin (\Im(z)) \cos(\Im(z)).
\end{align*}

In what follows $\lambda=\sqrt{\pi/8}$. Let $D\subseteq \HH$ be a domain with $\partial D= \partial_1 \cup \partial_2$, $\partial_1\cap \partial_2=\emptyset$ and $\partial_2$ an interval of $\R$. We also define $\check D$ to be the domain obtained by reflecting $D$ on $\HH$. We take $\Phi$ to be a GFF in a domain $D$ with boundary condition $\lambda$ on $\partial_1$ and $0$ on $\partial_2$, and $\check \Phi$ be a GFF in $D$ with boundary condition $0$ on $\partial_1$ and free on $\partial_2$. We couple $\Phi$ and $\check \Phi$ such that $\A_{-\lambda,\lambda}$ of $\Phi$ coincides with $\Ale$ -- the ALE of $\check \Phi$ as in \cite{QW} (See Figure \ref{f.coupling}). Furthermore, we denote by $u(z)$  the harmonic function in $D$ with boundary values $\lambda \mathbf 1_{\partial_1}$ so that $\Phi-u(z)$ is a GFF with $0$ boundary conditions in $D$.
\begin{figure}[h!]
	\includegraphics[width=0.9\textwidth]{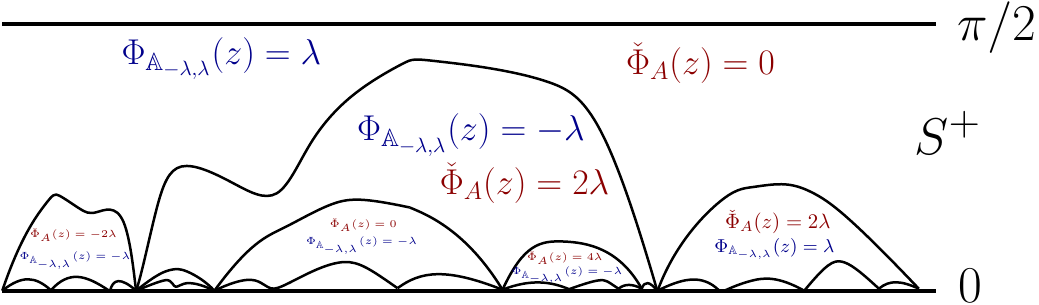}
	\caption{Representation of $\A_{-\lambda,\lambda}$ and $A$ in $S^+$. The harmonic function associated to $\Phi_{\A_{-\lambda,\lambda}}$ takes alternating values in $\pm 2\lambda$, on the other hand $\check \Phi_A$ starts with value $0$ on top and each time it crosses a curve changes values of $\pm 2\lambda$ independently in each connected component.}
	\label{f.coupling}
\end{figure}

\begin{proof}[Proof of Proposition \ref{p.ALE_moments}] Let $\Phi_\eps$ and $\check \Phi_\eps$ denote the $\eps$-circle average of $\Phi$ and $\check \Phi$ respectively. A straightforward computation yields that (see for example Theorem 1.23 of \cite{BP})
 \begin{align}
 	\E\left[(\Phi_\eps(z)-u(z))^2 \right] - \frac{1 }{2\pi }\log(1/\eps) \to  \frac{1 }{2\pi }\log(CR(z;D)) \ \ \ \text{ as } \eps \to 0.
 \end{align}
 Furthermore, as $\Gamma= \frac{1}{\sqrt{2}}(\Phi -u + \check \Phi)$ has the law of a GFF with $0$ boundary conditions on $\check D$ restricted to $D$, we have that
 \begin{align*}
 		\E\left[(\check \Phi_\eps(z))^2 \right] - \frac{1 }{2\pi }\log(1/\eps)  \to \frac{1 }{\pi }\log(CR(z;\check D)) - \frac{1 }{2\pi }\log(CR(z;D )) \ \ \ \text{ as } \eps \to 0.
 \end{align*}

Now, note that the strong Markov property of the GFF implies that
\begin{align*}
\Phi = \Phi_{\A_{-\lambda,\lambda}} + \Phi^{\A_{-\lambda,\lambda}} \text { and }
\check \Phi = \check \Phi_\Ale + \check \Phi^\Ale,
\end{align*}
where $\Phi_{\A_{-\lambda,\lambda}}$ and $\check \Phi_\Ale$ can be represented by harmonic functions thanks to \cite{Sep19}, and conditional on $\Ale$,  $\Phi^{\A_{-\lambda,\lambda}}$ and $\check \Phi^\Ale$ are 0 boundary GFF on $D\backslash \Ale$ (conditionally) independent of $(\Phi_{\A_{-\lambda,\lambda}},\check \Phi_\Ale)$. Furthermore, it is possible to check that
\begin{align*}
\Phi_{\A_{-\lambda,\lambda}}(z) = \lambda O(z)  \text{ and }
\check\Phi_{\Ale}= \sum_{k=1}^{n(z)} 2\lambda\xi_k,
\end{align*}
where $(\xi_k)_{k=1}^\infty$ are Rademacher random variables independent of $\Ale$.

We compute now the expected value of $O(z)$. To do this, we use the strong Markov property of the GFF $\Phi$ with respecto to $\A_{-\lambda,\lambda}$ to see that
\begin{align*}
	u(z) = \E\left[\Phi_{\A_{-\lambda,\lambda}}(z) \right].
	\end{align*}
From the above we obtain that 
\begin{align*}
	\E\left[O(z) \right] = u(z)/\lambda\ \  \text{ and when $D=S$, } \ \
	\mathbb P\left[O(z)=-1 \right] = 1/2-\Im(z)/\pi. 
\end{align*}

We now need to compute the expected value of $n(z)$. To do this, we study the square of both $\Gamma$ and $\check \Gamma$. For the first one the Markov property implies that
\begin{align*}
u^2(z)=\E\left[(\Phi_\eps(z) )^2\right] = \E\left[( (\Phi_{\A_{-\lambda,\lambda}})^D_\eps(z))^2 \right]   +\E\left[( (\Phi^{\A_{-\lambda,\lambda}})_\eps(z))^2 \right].
\end{align*}
Thus
\begin{align}
\nonumber u^2(z)&=\E\left[( (\Phi_{\A_{-\lambda,\lambda}})_\eps(z))^2 \right]  + \E\left[( (\Phi^{\A_{-\lambda,\lambda}})_\eps(z))^2 -(\Phi_\eps(z)-u(z))^2\right] \\
\label{e.0-ALE}&\to \lambda^2 + \frac{1 }{2\pi }\E\left[ \log\left( \frac{CR(z;D\backslash \A_{ -\lambda, \lambda})}{CR(z;D)}\right)\right] \ \ \text{ as } \eps \to 0.
\end{align}

On the other hand, using $\Ale$ instead of $\A_{-\lambda,\lambda}$ and $ \check \Phi$ instead of $\Phi$  we have that
\begin{align}\label{e.SMP}
	\E[(\check \Phi_\eps(z))^2]=\E\left[( (\check \Phi_{\Ale})_\eps(z))^2 \right]   +\E\left[( (\check \Phi^{\Ale})_\eps(z))^2 \right].
\end{align}
From here we obtain that
\begin{align*}
0&=\E\left[( (\check \Phi_{\Ale})_\eps(z))^2 \right]  + \E\left[( (\check \Phi^{\Ale})_\eps(z))^2 -(\check \Phi_\eps(z))^2\right] \\
&\to \E\left[\left ( 2\lambda\sum_{k=1}^{n(z) }\xi_k \right)^2 \right] +\frac{1 }{2\pi }\E\left[ -2\log(CR(z;\check D))+\log(CR(z;D)) +\log\left( CR(z;D\backslash \Ale)\right)\right] \\
&=4\lambda^2\E\left[n(z)  \right]+ \frac{1 }{2\pi }\E\left[ -2\log(CR(z;\check D))+\log(CR(z;D)) +\log\left( CR(z;D\backslash \Ale)\right)\right].
\end{align*}

As $\Ale=\A_{-\lambda,\lambda}$, we have as a consequence 
\begin{align*}
	\E\left[n(z)  \right] = \frac{1}{4}- \frac{2u^2(z)}{\pi}-\frac{2}{\pi^2}(\log(CR(z;D))-\log(CR(z;\check D))).
\end{align*}

In particular, this means that when $D=S^+$ we have that
\begin{align*}
\E\left[n(z) \right] = \frac{1}{4}- \frac{(\Im(z))^2}{\pi^2}-\frac{2}{\pi^2}\log(\sin(\Im (z))).
\end{align*}
\end{proof}

\section{Inverse Kasteleyn matrix in Temperleyan domains}\label{appendix:A}

In Appendices \ref{appendix:A} and \ref{app:B} we present proofs of the behaviour of the inverse Kasteleyn Matrix in the small mesh limit - that it can be written as a sum of functions expressed in terms of the continuum Green function. These are technical proofs and are similar to computations in, for example, \cite{Ken00}. In this section we deal with ordinary Temperleyan domains, which is needed for our folded dimer model, and in particular show that the statement of the behaviour of the inverse Kasteleyn matrix extends to certain neighbourhoods of the boundary. In Appendix \ref{app:B} we deal with piecewise Temperleyan domains, which is needed for our shifted dimer model.

In this section, the inverse Kasteleyn matrix on Temperleyan domains will often be called the \emph{coupling function} as in \cite{Ken00}. We use the notation and setting of Section \ref{section:l&a}. We start by recalling some well-known facts on the coupling function and discrete holomorphic functions. The acquainted reader can jump directly to the statements of Lemma \ref{lem:Green} and Lemma \ref{lem:Green:boundary}.\\

Let $U$ be an open set, $\Ue$ an approximating sequence as in Definition \ref{def:approx} and $\Ge$ the associated Temperleyan approximation with a boundary vertex $b_0$ removed as in Section 3. Note that these results are applied to $U = U^r$ in Section \ref{sec:proof-main-thm}. To simplify notations, we will write $B = B(\Ge)$, $B_0 = B_0(\Ge)$ etc.  We use the usual Kasteleyn phases of discrete holomorphy described in Equation \eqref{kast-eq-conplex-kast-weights} and denote by $\tK$ the associated Kasteleyn matrix as in the proof of Theorem \ref{thm:asymp:ne:oe}. It can be checked that 
\bestar
    \tK^{*}\tK = \Delta,
\eestar
where the Laplacian has Dirichlet boundary conditions on $B_0$ and Neumann boundary conditions on $B_1$ except at $b_0$ where it has Dirichlet boundary conditions. The determinant of the right hand side counts tiling of a Temperleyan domain which are in bijection with spanning trees of $B_1$ rooted at $b_0$, hence it is positive, so $\det \tK$ is also positive and $\tK$ is invertible: we call its inverse $\tK^{-1}$ the coupling function.\par
In this section only, we use the operator notation $K(w,b)$ instead of the matrix notation $K_{w,b}$ to align with the notation of Kenyon. Let us fix $w \in W_0$ (the same holds for $w \in W_1$). We first recall how the coupling function $b \to K^{-1}(b,w)$ can be seen as a discrete meromorphic function: we call the \emph{real part} the restriction to $b \in B_0$ and the \emph{imaginary part} the restriction to $b \in B_1$. This is justified by the following:
    \bestar
    \begin{aligned}
        \tK^{-1}(b,w) 
        &= (G\tK^{*})(b,w)\\
        &= G\left(b, w+\frac{\eps}{2}\right) - G\left(b, w-\frac{\eps}{2}\right) + i\left[G\left(b, w-\frac{i\eps}{2}\right) - G\left(b, w+\frac{i\eps}{2}\right)\right],
        \end{aligned}
    \eestar
    where $G = \Delta^{-1}$ is the Green function. The function $G$ takes real values because the Laplacian $\Delta$ takes real values. Further, since the Laplacian has a block structure (it is non zero only from $B_0$ to $B_0$ and $B_1$ to $B_1$), $G$ has the same block structure. Hence, since $w \in W_0$, $w \pm \frac{\eps}{2} \in B_0$ while $w \pm \frac{i\eps}{2} \in B_1$, so $\tK^{-1}(b,w)$ is real for $b \in B_0$ and pure imaginary for $b \in B_1$ (justifying the terminology ``real'' and ``imaginary'' part).\par

Let us define $\partial^{out} B_0$ (the full red and blue circles on Figure \ref{fig:reflection}) to be the vertices in $B_0(\eps\ZZ^2) \setminus B_0$ (where $B_0(\eps\ZZ^2)$ are the vertex of the dual graph of $\eps\ZZ^2$ i.e. the faces of $\eps\ZZ^2$) at distance $\eps$ from a vertex of $B_0$ (see Kenyon, section 4.1) and $\partial^{out} B_1 = \{b_0 \}$ (full red square on Figure \ref{fig:reflection}). Define $\overline{B_0} = B_0 \cup \partial^{out} B_0$, $\overline{B_1} = B_1 \cup \partial^{out} B_1$ and $\overline{B} = \overline{B_0} \cup \overline{B_1}$. If we extend $b \in B_0 \to \tK^{-1}(b,w)$ by $0$ on the vertices of $\partial^{out} B_0$, it is a harmonic function on $B_0$ except at two points: $\Delta \tK^{-1}(\cdot, w)_{|B_0} = \d_{w + \eps/2} - \d_{w - \eps/2}$. Indeed, for $b \in B_0$, 
\be\label{eq:W0}
    \Delta \tK^{-1}(b,w) = \tK^{*}\tK \tK^{-1}(b,w) = \tK^{*}(b,w) = \overline{\tK(w,b)} = \d_{b=w + \eps/2} - \d_{b=w - \eps/2}
\ee
as a function of $b \in B_0$. Observe that this already fully characterizes the real part: there is a unique function which satisfies \eqref{eq:W0} and is $0$ on $\partial^{out}B_0$ (by the maximum principle). If we extend the imaginary part $b \in B_1 \to \tK^{-1}(b,w)$ by $0$ on $b_0$, the imaginary and real part are \emph{harmonically conjugated}: $\tK\tK^{-1} = I$ implies that for $w' \neq w$,
\be\label{eq:discrete:CR}
    \tK^{-1}\left(w'+\frac{\eps}{2},w\right)-\tK^{-1}\left(w'-\frac{\eps}{2},w\right) + i\left[\tK^{-1}\left(w'+\frac{i\eps}{2},w\right)-\tK^{-1}\left(w'-\frac{i\eps}{2},w\right)\right] = 0
\ee
(which also holds when $w'$ is on the boundary due to the choice of boundary conditions). This is the discrete Cauchy Riemann equation for $\tK^{-1}(\cdot,w)$ around the point $w'$.\par
In our case (and unlike Kenyon who considers multiply connected domains), the imaginary part is fully defined by saying that it is the harmonic conjugate of the real part and it is $0$ at the removed vertex $b_0$: by discrete integration along paths which avoid the singularity at $w$, the values of $\tK^{-1}(\cdot,w)$ on $B_1$ can be uniquely recovered from the values on $B_0$ and the boundary conditions at $b_0$. We summarize all this in a lemma:

\begin{lemma}\label{lem:coupling:function}
    The inverse Kasteleyn matrix $\tK^{-1}$ is uniquely defined, and for $w \in W_0$ fixed, it is characterized by the following:
    \begin{itemize}
        \item its real part is the (unique) harmonic function on $\overline{B_0}$ which is $0$ on $\partial^{out}B_0$ (Dirichlet boundary conditions) and satisfies
        \bestar
            (\Delta \tK^{-1})(\cdot, w) = \d_{w + \eps/2} - \d_{w - \eps/2} 
        \eestar
        \item its imaginary part is $0$ at $b_0$
        \item its real and imaginary part satisfy the discrete Cauchy Riemann equation \eqref{eq:discrete:CR} at all points $w' \in W \setminus \{w\}$.     
    \end{itemize}
\end{lemma}

We are now ready to give estimates of the coupling function. Let $\tilde{g}(u,v)$ be the analytic function of $v$ whose real part is the Dirichlet Green function $g(u,v)$ on $U$. Also define, as in Section 6.3 of \cite{Ken14},
\be\label{eq:def:Fpm}
    F_+(u,v) = \frac{\partial\tilde{g}(u,v)}{\partial u} \quad ; \quad F_-(u,v) = \frac{\partial\tilde{g}(u,v)}{\partial \bar{u}}.
\ee
Then, the following lemma holds:

\begin{lemma}\label{lem:Green}[\cite{Ken00}, Theorem 13]
  Let $\eta >0$ be fixed. If $u \in V(\Ge)$ is at distance at least $\eta$ from the boundary and $v \in V(\Ge)$ is at distance at least $\eta$ from $u$, and if $b$ and $w$ are within $O(\eps)$ of $u$ and $v$ respectively, then
  \bestar
    \tK^{-1}(b,w) = \left\{
    \begin{array}{ll}
        \eps Re(F_+(u,v) + F_-(u,v)) + O(\eps^2) & \text{if } w \in W_0, b \in B_0 \\
        \eps i Im(F_+(u,v) + F_-(u,v)) + O(\eps^2) & \text{if } w \in W_0, b \in B_1 \\
        \eps Re(F_+(u,v) - F_-(u,v)) + O(\eps^2) & \text{if } w \in W_1, b \in B_1\\
        \eps i Im(F_+(u,v) - F_-(u,v)) + O(\eps^2) & \text{if } w \in W_1, b \in B_0
    \end{array}
    \right.
  \eestar
  or for short, if $r,s \in \{\pm 1\}$ are such that $w \in W_{\frac{1-r}{2}}$, $b \in W_{\frac{1-s}{2}}$,
  \bestar
      \tK^{-1}(b,w) = \frac{1}{2}\left(F_+(u,v)+rF_-(u,v)+s\ol{F_-(u,v)}+rs\ol{F_+(u,v)}\right)+O(\eps^2)  
  \eestar
  where the $O(\eps^2)$ is uniform once $\eta$ is fixed. 
\end{lemma}
We only need the simplest part of the proof since the points $u,v$ we consider are at distance of order $1$ (and not $\eps$) from each other and the domain is simply connected so the harmonic conjugate is single-valued. Actually, this lemma also works for points near the boundary: this is mostly Theorem 14 of \cite{Ken14}, but we make some arguments more precise.

\begin{lemma}\label{lem:Green:boundary}
    Assume (as in Section \ref{section:l&a}) that there exists $\zp \in \partial U$ and $\d > 0$ such that the boundary of $U$ is horizontal in the neighbourhood $B(\zp, \d)$ and that $U$ is contained in the half plane below this flat horizontal boundary. Assume that $\Ue$ also has a flat and horizontal boundary in $B(\zp,\d)$. Then for all $r, s \in \{\pm1\}$, $ w \in W_{\frac{1-r}{2}}(\Ge), b \in B_{\frac{1-s}{2}}(\Ge)$ each at distance at most $\delta/2$ of $\zp$ or at distance at least $\d/2$ from the boundary of $U$, if $|b-w|\geq \d/2$ and if $w$ and $b$ are within $O(\eps)$ of $u, v \in U$ respectively, then
  \bestar
      \tK^{-1}(b,w) = \frac{1}{2}\left(F_+(u,v)+rF_-(u,v)+s\ol{F_-(u,v)}+rs\ol{F_+(u,v)}\right)+O(\eps^2)  
  \eestar
    where the $O(\eps^2)$ is uniform once $\d$ is fixed. 
\end{lemma}
In particular, with the notation of the lemma,
\be\label{eq:K:bounded}
    |\tK^{-1}(b,w)| \leq C(\delta)\eps
\ee
for some constant $C$ depending only on $\delta$.
\begin{remark}
These two lemmas could actually be proved in a slightly more general setting when $\d = \d(\eps) \overset{\eps \to 0}\longrightarrow 0$. 
\end{remark}

The proof consists in detailing some arguments in the proof of Theorem 14 of \cite{Ken00} (more precisely we use the argument of Corollary 19) but our statement applies to all points at distance at most $\d/2$ from $\zp$, unlike the statement of Kenyon which applies only to the points within $O(\eps)$ of the boundary.
\begin{proof}
    \begin{figure}\centering
    \begin{overpic}[abs,unit=1mm,scale=0.5]{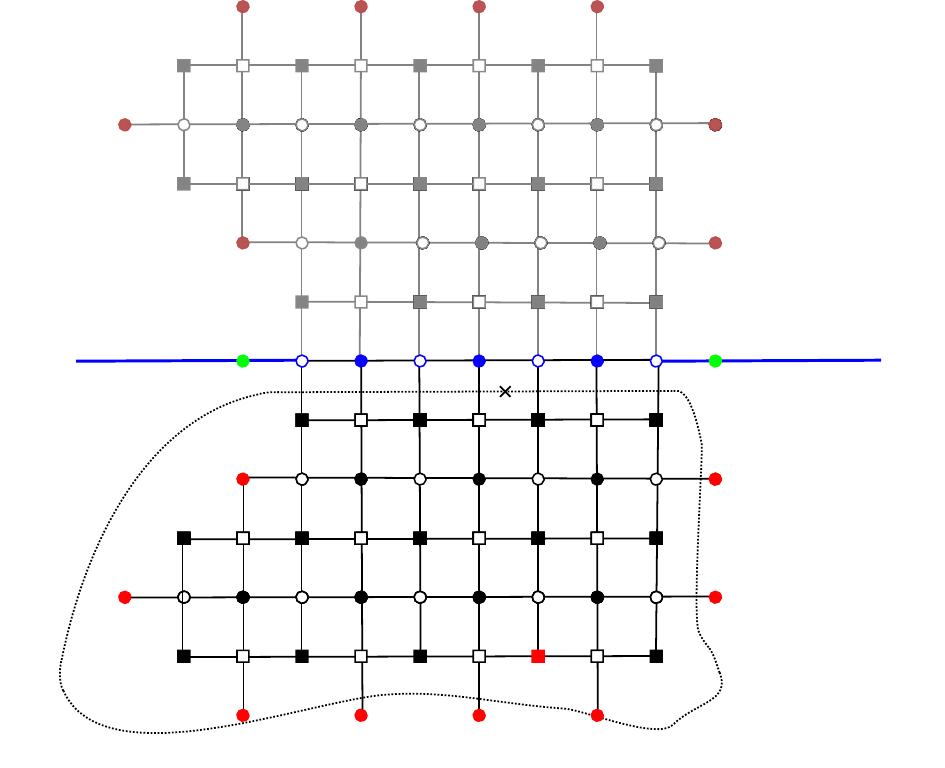}
    \put(8,25){$U$}
    \put(44,31){$\zp$}
    \put(20,36){\color{green}$b_l$}
    \put(60,36){\color{green}$b_r$}
    \put(45,6){\color{red}$b_0$}
    \end{overpic}
    \caption{The graph $\Ge$ near the point $\zp$ and the flat boundary is drawn in black. The boundary points on the flat boundary $\partial^{flat}B_0$ are represented as full blue circles. The rest of the boundary $\partial^{out}B_0$ is represented by full red circles. These are the point with Dirichlet boundary condition for the coupling function. The reflection axis is drawn as a blue line and the reflected vertices are drawn in grey for the bulk vertices and light red for the boundary vertices.}
    \label{fig:reflection}
\end{figure}
    We have to deal with the fact that $w$ might be close to the boundary so Lemma \ref{lem:Green} would not apply. Let $\partial^{flat}B_0$ be the set of the uppermost points of $\partial^{out}B_0$ (full blue circles in Figure \ref{fig:reflection}). They are all situated on the same horizontal line (full blue line in Figure \ref{fig:reflection}). We call $s$ the reflection along this line. We also define $\partial^{flat}W_0$ (the empty blue circles in Figure \ref{fig:reflection}) to be the vertices in $W_0(\ZZ^2)$ (the horizontal edges of $\eps\ZZ^2$) which are at distance $\eps/2$ from a point in $\partial^{flat}B_0$, and we define $\overline{W_0} = W_0 \cup \partial^{flat}W_0$, $\overline{W} = \overline{W_0} \cup W_1$. We consider a new graph $\Ge^s$ which has vertex set $B \cup \overline{W} \cup \partial^{flat}B_0 \cup s(B) \cup s(\overline{W}) \cup \{s(b_0)\}$ (the grey, black and blue points in Figure \ref{fig:reflection}) where we identify the points of $\partial^{flat} B_0$ and $\partial^{flat} W_0$ and their image (they are fixed by the reflection). The edges of $\Ge^s$ link points at distance $\eps/2$. The boundary of $\Ge^s$ is $\partial^{out} B_0 \cup s(\partial^{out}B_0) \cup \{b_0,b_l,b_r\} \setminus \partial^{flat}B_0$ (the red, light red and green points in Figure \ref{fig:reflection}) where $b_l,b_r$ are two added points (in green on Figure \ref{fig:reflection}): $b_l$ at distance $\eps$ to the left of the leftmost point in $\partial^{flat}B_0$ and $b_r$ at distance $\eps$ to the right of the rightmost point in $\partial^{flat}B_0$. We will write $B^s$ for the black vertices of $\Ge^s$, $\overline{\Ge^s}$ for the union of $\Ge^s$ and its boundary, etc.\par
    We can extend the discrete meromorphic function $\tK^{-1}(\cdot,w)$ on $\overline{B^s}$ by a discrete analogue of the Schwarz reflection principle for holomorphic functions. We define, for $b \in \overline{B}$,
    \be\label{def:extension}
        \tK^{-1}_s(s(b),w) = \left\{
        \begin{array}{ll}
            -\tK^{-1}(w,b) & \text{if } b \in \overline{B_0} \\
            \tK^{-1}(w,b)  & \text{if } b \in \overline{B_1} \\
        \end{array}
        \right.
    \ee
    and set $\tK^{-1}(b_l,w) = \tK^{-1}(b_r,w) = 0$. This extended function is well defined (the only problem could be for $b \in \partial^{flat} B_0$ for which $s(b)=b$, but in this case $\tK^{-1}(w,b) = \tK^{-1}(w,s(b)) = 0$). It is discrete meromorphic on $\G^s$: harmonicity of the real and imaginary parts (except at $w$ and $s(w)$) is preserved on and near the reflection axis (Definition \eqref{def:extension} is designed to match the Dirichlet and Neumann boundary conditions). Besides, $b \in B^s_1 \to \tK^{-1}_s(b,w)$ remains harmonically conjugated to $b \in B^s_0 \to \tK^{-1}_s(b,w)$ except at $w$ and $s(w)$: using the definition of the extension $\tK^{-1}$ and the discrete Cauchy Riemann equation \eqref{eq:discrete:CR}, if $w' \in W_0 \setminus \{w \}$, $s(w) \in W^s_0$ and if we let
    \bestar
        \partial_x \tK^{-1}(s(w'),w) = \tK^{-1}\left(s(w')+\frac{\eps}{2},w\right)-\tK^{-1}\left(s(w')-\frac{\eps}{2},w\right)
    \eestar
    then
    \bestar
        \begin{aligned}
            \partial_x \tK^{-1}(s(w'),w)
            &= \tK^{-1}\left(s\left(w'+\frac{\eps}{2}\right),w\right)-\tK^{-1}\left(s\left(w'-\frac{\eps}{2}\right),w\right)\\
            &= -\left[\tK^{-1}\left(w'+\frac{\eps}{2},w\right)-\tK^{-1}\left(w'-\frac{\eps}{2},w\right)\right]\\
            &= i\left[\tK^{-1}\left(w'+\frac{i\eps}{2},w\right)-\tK^{-1}\left(w'-\frac{i\eps}{2},w\right)\right]\\
            &= i\left[\tK^{-1}\left(s\left(w'+\frac{i\eps}{2}\right),w\right)-\tK^{-1}\left(s\left(w'-\frac{i\eps}{2}\right),w\right)\right]\\
            &= i\left[\tK^{-1}\left(s(w')-\frac{i\eps}{2},w\right)-\tK^{-1}\left(s(w')+\frac{i\eps}{2},w\right)\right]\\
            &= -i\left[\tK^{-1}\left(s(w')+\frac{i\eps}{2},w\right)-\tK^{-1}\left(s(w')-\frac{i\eps}{2},w\right)\right]
        \end{aligned}
    \eestar
    which is exactly the discrete Cauchy Riemann equation \eqref{eq:discrete:CR} at $s(w')$ (note how the reflection changes the signs only for vertical differences). This also holds for $w \in W_1$ (but there is an additional minus sign in the first line and no sign in the last line, which compensate).\\
    
    We now show that the extended coupling function $b \in B^s \to \tK^{-1}_s(b,w)$ coincides with the difference of two coupling functions on $B^s$. Since $\Ge^s$ is a Temperleyan approximation of $U \cup s(U)$ if we remove the vertex $b_0$, by Lemma \ref{lem:coupling:function} for fixed $w \in W_0 \cup s(W_0)$, the coupling function $b \in B^s \to H^s(b,w)$ is uniquely defined: its real part has $0$ boundary conditions on $\partial^{out} B_0 \cup s(\partial^{out}B_0) \cup \{b_l,b_r\} \setminus \partial^{flat}B_0$ and satisfies $\Delta H^s(\cdot,w) = \d_{w + \eps/2} - \d_{w - \eps/2}$ at all inner points $b \in B^s_0$. Its imaginary part is $0$ on $b_0$ and is harmonically conjugated with the real part except at $w$. We claim that for all $b \in B^s_0$,
    \be\label{eq:extended}
        \tK^{-1}_s(b,w) = H^s(b,w) - H^s(b,s(w)).
    \ee 
    Indeed, the two sides of the equation have the same Laplacian $\d_{w + \eps/2} - \d_{w - \eps/2} - \d_{s(w) + \eps/2} + \d_{s(w) - \eps/2}$ and both have zero boundary conditions  on $B_0^s$, so the real parts coincide. Hence, the difference $b \in B^s_1 \to H^s(b,w) - H^s(b,s(w))$ is harmonically conjugated (except at $w$ and $s(w)$) with $b \in B^s_0 \to \tK^{-1}_s(b,w) = H^s(b,w) - H^s(b,s(w))$ and it is zero at $b_0$, so it coincides with $b \in B^s_1 \to \tK^{-1}_s(b,w)$ by integration along any path from $b_0$ to $b$ avoiding $w$ and $s(w)$. Finally, Equation \ref{eq:extended} also holds for $b \in B^s_1$. 
\begin{remark}
    Observe that we did not impose any condition at $s(b_0)$ for $H^s(b,w)-H^s(b,s(w))$, but since equality \eqref{eq:extended} holds it must be $0$ a posteriori. This can be checked using symmetry arguments. 
\end{remark}
    $B_1^s$ is an approximating sequence for $U \cup s(U)$ so we can apply Lemma \ref{lem:Green} to conclude, because (and this is the key point of the proof) if $w$ is as in the statement of the Lemma, then both $w$ and $s(w)$ are at distance at least $\d/2$ from the boundary of $B^s$: the points near the flat part of the boundary around $\zp$ have become inner points of the graph $B^s$. In particular, the points in $B$ at distance at most $\d/2$ from $\zp$ are at distance at least $\d/2$ from the boundary of $B^s$, so we can apply Lemma \ref{lem:Green} with $\eta = \d/2$ to $H^s(\cdot,w)$ and $H^s(\cdot,s(w))$ since they are the coupling function for the Temperleyan approximation of $U \cup s(U)$ with a black vertex removed at the same location as before. 
    
    If we denote by $\wtilde{g^s}(u,v)$ the analytic function of $v$ whose real part is the Dirichlet Green function $g^s(u,v)$ on $U \cup s(U)$, we define $F_+^s$ and $F^s_-$ similarly to the definition before Lemma \ref{lem:Green}. If $w$ is at distance $O(\eps)$ from $u$ and $b$ is at distance $O(\eps)$ from $v$,
    \bestar
    \tK^{-1}(b,w) = \left\{
    \begin{array}{ll}
        \eps Re\left[F^s_+(u,v) + F^s_-(u,v) - F^s_+(s(u),v) - F^s_-(s(u),v)\right]& + O(\eps^2)\\
        &\text{if } b \in B_0, \\
        \eps i Im\left[F_+(u,v) + F_-(u,v) -F^s_+(s(u),v) - F^s_-(s(u),v)\right]& + O(\eps^2)\\
        & \text{if } b \in B_1. 
    \end{array}
    \right.
  \eestar
    where the $O(\eps^2)$ is uniform once $\d$ is fixed. Observe now (this is the continuous equivalent of what we have just done, i.e. Schwarz reflection principle in the continuum) that $\wtilde{g^s}(u,v) -\wtilde{g^s}(s(u),v) = \tilde{g}(u,v)$ when $u,v \in U$, because the left-hand side is analytic in $v$ on $U$, its real part $g^s(u,v)-g^s(s(u),v)$ has the same pole as $g(u,v)$ in $v$ (by definition as a Green function) and it is $0$ on the boundary of $U$. On the flat portion, it is true by symmetry: if $v$ is on the reflection axis, $g^s(u,v) = g^s(s(u),s(v)) = g^s(s(u),v)$. On the rest of the boundary both terms of the difference vanish since $v$ is also on the boundary of $U \cup s(U)$. Hence, $F_+^s(u,v) - F_+^s(s(u),v) = F_+(u,v)$ and $F_-^s(u,v)-F_-^s(u,v) = F_-(u,v)$ which concludes the proof in the case $w \in W_0$; Lemma \ref{lem:Green} holds up to the boundary.\par
    To deal with the case $w \in W_1$, the same argument works but now $\tK^{-1}(\cdot, w)_{|B_0}$ is the unique function on $B_0$ which is $0$ on $\partial^{out} B_0$ and whose Laplacian is $-i\d_{b=w + i\eps/2} +i \d_{w - i\eps/2}$ by equation \eqref{eq:W0} applied when $w \in W_1$ instead of $W \in W_0$. The remaining changes are straightforward.
\end{proof}

\section{Coupling function in piecewise Temperleyan domains.}\label{app:B}

In this final Appendix, we continue our analysis of the inverse Kasteleyn matrix in the small mesh limit, and in particular we deal with piecewise Temperleyan domains, which is needed for our shifted dimer model.\\

We use the setting of Section \ref{section:l&a}: $U^r$ is a simply connected open set symmetric with respect to the horizontal axis, $U = U^r \cap (\RR \times \RR_{>0})$. The boundary of $U$ can be partitioned in two connected arcs: $\partial U = D \sqcup N$ where $N = \partial U \cap (\RR \times \{0\})$ is the intersection with the horizontal axis. Recall that we defined two graphs approximating $U$: the Temperleyan approximation $\Ge^1$ with Dirichlet boundary conditions (and a vertex $b_0$ removed) and the piecewise Temperleyan approximation $\Ge^2$ with two convex white corners $v_1^*, v_2^*$. In both cases, we use the usual Kasteleyn phases of discrete holomorphy described in Equation \eqref{kast-eq-conplex-kast-weights}, and denote by $K^1, K^2$ the corresponding Kasteleyn matrices on $\Ge^1, \Ge^2$. Both $\Ge^1$ and $\Ge^2$ have at least one dimer configuration (for example because from any loops and arcs configuration $\om \in \Omega(\Ge)$ we can obtain two dimer covers of $\Ge^1$ and $\Ge^2$), hence $K^1$ and $K^2$ are invertible by \cite{Russ}. Moreover, the asymptotics of $(K^j)^{-1}$ on piecewise Temperleyan domains are computed and expressed in terms of $f_i^j(u,v)$ for $i \in \{0,1\},~j \in \{1,2\}$. These are the unique functions which are analytic in $u$ for fixed $v$, have a simple pole of residue $1/\pi$ at $u=v$ with no other pole on $\bar{U}$, are bounded near $v_1^*, v_2^*$ (for the $j=2$ case), and have the following boundary conditions: 
\begin{itemize}
    \item $\Re(f^1_0(\cdot,v) = 0$ on $D \sqcup N$
    \item $\Im(f^1_1(\cdot,v) = 0$ on $D \sqcup N$
    \item $\Re(f^2_0(\cdot,v) = 0$ on $D$ and $\Im(f^1_0(\cdot,v) = 0$ on $N$
    \item $\Im(f^2_1(\cdot,v) = 0$ on $D$ and $\Re(f^1_1(\cdot,v) = 0$ on $N$.
\end{itemize}
For $\tau = \pm 1$, $i \in \{1,2\}$, let
\bestar
    F_{\tau}^j = \frac{1}{2}(f_0^j + \tau f_1^j).
\eestar
Then, Theorem 6.1 of \cite{Russ} has the following corollary:
\begin{lemma}\label{lem:russ}
    Let $\eta > 0$ be fixed. If $u, v \in V(\Ge)$ are at distance at least $\eta$ from the boundary, and if $w$ and $b$ are within $O(\eps)$ of $u$ and $v$ respectively, for $j \in \{1,2\}$, $r,s \in \{\pm 1\}$, $w \in W_{\frac{1-r}{2}}$, $b \in B_{\frac{1-s}{2}}$,
    \bestar
    (K^j)^{-1}_{b,w} = \eps\left(F_+^j(u,v)+rF_-^j(u,v)+s\ol{F_-^j(u,v)}+rs\ol{F_+^j(u,v)}\right) + \Delta(w,b) + o(\eps)
    \eestar
    where the $o$ is uniform once $\eta$ is fixed and $\Delta$ is an error term depending only on $\eps$, $w \in W_0 \sqcup W_1, b \in B_0 \sqcup B_1$ (in particular $\Delta$ does not depend on $U$ or its Temperleyan discretization $\Ge$).
\end{lemma}
\begin{remark}
     In the original statement of \cite{Russ},
\bestar
    \Delta(w,b) = \frac{2\eps}{\lambda}F^{\eps}_{\CC, w}(b)-\frac{\eps}{u-v}
\eestar
where $\lambda = e^{i\frac{\pi}{4}}$ and $F^{\eps}_{\CC,w}$ is the unique discrete holomorphic function on $\eps \ZZ^2$ tending to $0$ at infinity and with a pole $\lambda$ at $w$, which is known to be asymptotically equal to $\frac{\lambda}{2\pi(b-w)}$ when $\eps \to 0$ and $|b-w|\geq \eta$ for some $\eta >0$ independent of $\eps$. See Section 5 of \cite{Russ} for details. The error term $\Delta$ is also present in the original statement of Lemmas \ref{lem:Green} and \ref{lem:Green:boundary}, but we included it in the $o(\eps)$ since we never use these lemmas for close points and for all $\eta > 0$, for all $b,w$ such that $|b-w| \geq \eta$,
    \bestar
        |\Delta(w,b)| = o_{\eps \to 0}(\eps).
    \eestar
    uniformly once $\eta$ is fixed. Since we need to apply Lemma \ref{lem:russ} also for points $b,w$ at distance $O(\eps)$, we must include the error term.
\end{remark}

As in Lemma \ref{lem:Green:boundary}, when $U$ has some specific boundary conditions, the asymptotic estimates hold up to the boundary.
\begin{lemma}\label{lem:russ:boundary}
    Assume (as in Section \ref{section:l&a}) that there exists $\zp \in D$ and $\d > 0$ such that the boundary of $U$ is horizontal in the neighbourhood $B(\zp, \d)$ and that $U$ is contained in the half plane below this flat horizontal boundary. Assume that $\Ue$ also has a flat and horizontal boundary in $B(\zp,\d)$. If $u, v \in V(\Ge)$ are each at distance at most $\d/2$ of $\zp$ or at least $\d/2$ from the $\partial U$, and if $w$ and $b$ are within $O(\eps)$ of $u$ and $v$ respectively, for $j \in \{1,2\}$, $r,s \in \{\pm 1\}$, $w \in W_{\frac{1-r}{2}}$, $b \in B_{\frac{1-s}{2}}$,
    \bestar
    (K^j)^{-1}_{b,w} = \eps\left(F_+^j(u,v)+rF_-^j(u,v)+s\ol{F_-^j(u,v)}+rs\ol{F_+^j(u,v)}\right) + \Delta(w,b) + o(\eps)
    \eestar
    where the $o$ is uniform once $\d$ is fixed and $\Delta$ is the error term from Lemma \ref{lem:russ}.
\end{lemma}
In particular, it holds that for all $b,w$ at distance at most $\d/2$ of $\zp$ or at least $\geq \d/2$ from the $\partial U$, for all $j \in \{1,2\}$,
\be\label{eq:bound:K:russ}
    |(K^j)^{-1}_{w,b}-\Delta(w,b)| \leq C(\d)\eps
\ee
for some constant $C(\d)$ depending only on $\d$.

\begin{proof}
    The same argument as in the proof of Lemma \ref{lem:Green:boundary} can be applied: we skip most of the details. For fixed $w \in W(\Ge^j)$, for $j \in \{1,2\}$, the coupling function $(K^j)^{-1}(\cdot, w)$ can be extended into a discrete meromorphic function $(K^j_s)^{-1}(\cdot,w)$ on the graph $(\Ge^j)^s$ obtained by gluing together $\Ge^j$ and $s(\Ge^j)$ (the reflection of $\Ge^j$ along the flat part of $\partial U$). Then, $(K^j_s)^{-1}(\cdot, w)$ is a discrete meromorphic function on $B((\Ge^j)^s)$, with mixed Dirichlet/Neumann boundary conditions and two poles of respective residues $1/\pi$ and $-1/\pi$ at $w$ and $s(w)$, hence it coincides with the difference of the two coupling functions $H^s(\cdot, w)-H^s(\cdot, s(w)$ with the same mixed boundary conditions. Indeed, the difference 
    \bestar
        (K^j_s)^{-1}(\cdot, w) - (H^s(\cdot, w)-H^s(\cdot, s(w))
    \eestar
    is in the kernel of the (invertible) matrix $K^j$, hence it must vanish. We conclude by applying Lemma \ref{lem:russ} to $H^s(\cdot,w)$ and $H^s(\cdot,s(w))$.
\end{proof}

We finally explain how $F^j_{\tau}, \tau = \pm 1$ can be expressed in terms of the functions $F_{\tau}, \tau = \pm 1$ defined in the preceding section (see Equation \ref{eq:def:Fpm}). Recall that $U^r$ denotes the reflected domain, and denote by $f_0(u,v), f_1(u,v)$ the unique analytic functions of $u \in U^r$ having a simple pole of residue $1/\pi$ at $u=v$ and boundary conditions $\Re(f^r_0(\cdot,v) = 0$ on $\partial U^r$, $\Im(f^r_1(\cdot,v) = 0$ on $\partial U^r$. They satisfy
\bestar
    F_+(u,v) = f_0(u,v)+f_1(u,v) \quad ; \quad F_-(u,v) = f_0(u,v)-f_1(u,v).
\eestar
Since $U^r$ is invariant by reflection along the horizontal axis,
\bestar
\forall i \in \{1,2\},~\forall u,v \in U^r,~\ol{f^r_i(\ol{u},\ol{v})} = f^r_i(u,v).
\eestar
Checking the boundary conditions in each case, it holds that:
\begin{itemize}
    \item $f_0^1(u,v) = f_0(u,v)-f_0(u,\ol{v})$
    \item $f_1^1(u,v) = f_1(u,v)+f_1(u,\ol{v})$
    \item $f_0^2(u,v) = f_0(u,v)+f_0(u,\ol{v})$
    \item $f_1^2(u,v) = f_1(u,v)-f_1(u,\ol{v})$
\end{itemize}
Or, to put it shortly, for $i \in \{0,1\}$, $j \in \{1,2\}$,
\bestar
    f_i^j(u,v) = f_i(u,v) + (-1)^{i+j}f_i(u, \ol{v})
\eestar
which implies that for $\tau \in \{\pm 1\}$, $u,v \in U$,
\be\label{eq:F^j}
    F^j_{\tau} 
    = F_{\tau}(u,v) + (-1)^j F_{-\tau}(u, \ol{v}).
\ee

\def\cprime{$'$}


\begin{thebibliography}{10}

\bibitem{Ahlfors1966}
L.~V. Ahlfors.
\newblock {\em Complex Analysis}.
\newblock McGraw-Hill Book Company, 2 edition.

\bibitem{FPS}
J.~Aru, T.~Lupu, and A.~Sep{\'u}lveda.
\newblock {First passage sets of the 2D continuum Gaussian free field}.
\newblock {\em Probability Theory and Related Fields}, 176(3):1303--1355, 2020.

\bibitem{AruSep}
J.~Aru and A.~Sep{\'u}lveda.
\newblock {Two-valued local sets of the 2D continuum Gaussian free field:
  connectivity, labels, and induced metrics}.
\newblock {\em Electronic Journal of Probability}, 23:1--35, 2018.

\bibitem{ASW}
J.~Aru, A.~Sep{\'u}lveda, and W.~Werner.
\newblock {On bounded-type thin local sets of the two-dimensional Gaussian free
  field}.
\newblock {\em Journal of the Institute of Mathematics of Jussieu},
  18(3):591--618, 2019.

\bibitem{BasChe}
Mikhail Basok and Dmitry Chelkak.
\newblock Tau-functions {\`a} la dub{\'e}dat and probabilities of cylindrical
  events for double-dimers and cle (4).
\newblock {\em Journal of the European Mathematical Society}, 23(8):2787--2832,
  2021.

\bibitem{BLR}
Nathana{\"e}l Berestycki, Benoit Laslier, and Gourab Ray.
\newblock Dimers and imaginary geometry.
\newblock {\em The Annals of Probability}, 48(1):1--52, 2020.

\bibitem{BLQ}
Nathana{\"e}l Berestycki, Marcin Lis, and Wei Qian.
\newblock Free boundary dimers: random walk representation and scaling limit.
\newblock {\em Probability Theory and Related Fields}, 186(3):735--812, 2023.

\bibitem{BerLiu}
Nathana{\"e}l Berestycki and Mingchang Liu.
\newblock {Piecewise Temperleyan dimers and a multiple SLE $ \_8$}.
\newblock {\em arXiv preprint arXiv:2301.08513}, 2023.

\bibitem{BP}
Nathana{\"e}l Berestycki and Ellen Powell.
\newblock {Gaussian free field and Liouville quantum gravity}.
\newblock {\em arXiv preprint arXiv:2404.16642}, 2024.

\bibitem{KastOrient}
David Cimasoni and Nicolai Reshetikhin.
\newblock Dimers on surface graphs and spin structures. i.
\newblock {\em Communications in Mathematical Physics}, 275:187--208, 10 2007.

\bibitem{webs}
Daniel Douglas, Richard Kenyon, and Haolin Shi.
\newblock Dimers, webs, and local systems.
\newblock {\em Transactions of the American Mathematical Society}, 377, 05
  2022.

\bibitem{DubDD}
Julien Dub{\'e}dat.
\newblock Double dimers, conformal loop ensembles and isomonodromic
  deformations.
\newblock {\em Journal of the European Mathematical Society}, 21(1):1--54,
  2018.

\bibitem{DCLQ1}
H.~Duminil-Copin, M.~Lis, and W.~Qian.
\newblock {Conformal invariance of double random currents I: identification of
  the limit}.
\newblock arXiv:2107.12985, 2021.

\bibitem{Ken00}
Richard Kenyon.
\newblock Conformal invariance of domino tiling.
\newblock {\em Ann. Probab.}, 28(2):759--795, 2000.

\bibitem{Ken01}
Richard Kenyon.
\newblock Dominos and the {G}aussian free field.
\newblock {\em Ann. Probab.}, 29(3):1128--1137, 2001.

\bibitem{Ken14}
Richard Kenyon.
\newblock {Conformal Invariance of Loops in the Double-Dimer Model}.
\newblock {\em Communications in Mathematical Physics}, 326(2):477--497, Mar
  2014.

\bibitem{KPW}
Richard~W. Kenyon, James~G. Propp, and David~B. Wilson.
\newblock Trees and matchings.
\newblock {\em Electron. J. Combin.}, 7:Research Paper 25, 34, 2000.

\bibitem{QW}
Wei Qian and Wendelin Werner.
\newblock {Coupling the Gaussian Free Fields with Free and with Zero Boundary
  Conditions via Common Level Lines}.
\newblock {\em Communications in Mathematical Physics}, 361(1):53--80, 2018.

\bibitem{Russ}
Marianna Russkikh.
\newblock {Dimers in Piecewise Temperleyan Domains}.
\newblock {\em Communications in Mathematical Physics}, 359(1):189--222, 2018.

\bibitem{SchShe}
Oded Schramm and Scott Sheffield.
\newblock {A contour line of the continuum Gaussian free field}.
\newblock {\em Probability Theory and Related Fields}, 157(1-2):47--80, 2013.

\bibitem{Sep19}
Avelio Sep{\'u}lveda.
\newblock On thin local sets of the {G}aussian free field.
\newblock {\em Ann. Inst. Henri Poincare Probab. Stat.}, 55(3):1797--1813,
  August 2019.

\bibitem{TemFish}
H.~N.~V. Temperley and Michael~E. Fisher.
\newblock Dimer problem in statistical mechanics---an exact result.
\newblock {\em Philos. Mag. (8)}, 6:1061--1063, 1961.

\end{thebibliography}
\end{document}